\documentclass[a4paper, 9pt]{article}
\usepackage[utf8]{inputenc}
\usepackage[english]{babel}
\usepackage{amsmath, amsfonts, amsthm, amssymb}
\usepackage{algorithm}
\usepackage{algorithmic}
\usepackage{bm}
\usepackage[font={small,it}]{caption}
\usepackage{changepage}
\usepackage{cite}
\usepackage{dsfont}
\usepackage{geometry}
\usepackage{graphicx}
\usepackage{hyperref}
\usepackage{mathtools}
\usepackage{multicol}
\usepackage{multirow}
\usepackage{tikz}
\reversemarginpar

\newcommand{\pO}{\partial \Omega}
\newcommand{\mN}{\mathcal{N}}

\newcommand{\Gammaj}{\Gamma^{(j)}}

\newcommand{\diej}{ \delta_j^{i,e \ \dagger}}

\newcommand{\dwj}{ \delta_j^{w \ \dagger}}

\newcommand{\RF}{R_{\mathcal{F}}}
\newcommand{\RE}{R_{\mathcal{E}}}

\newcommand{\uag}{u_{\mathcal{A}, \Gamma}}

\newcommand{\vag}{v_{\mathcal{A}, \Gamma}}
\newcommand{\wag}{w_{\mathcal{A}, \Gamma}}

\newcommand{\uG}{u_\Gamma}
\newcommand{\vG}{v_\Gamma}
\newcommand{\wG}{w_\Gamma}

\newcommand{\uFj}{u_{j, \mathcal{F}}}
\newcommand{\uFk}{u_{k, \mathcal{F}}}

\newcommand{\uEj}{u_{j, \mathcal{E}}}
\newcommand{\uEjuno}{u_{j_1, \mathcal{E} } }
\newcommand{\uEjdue}{u_{j_2, \mathcal{E}}}
\newcommand{\uEjtre}{u_{j_3, \mathcal{E}}}
\newcommand{\umeanE}{\bar{u}_\mathcal{E} }

\newcommand{\uiewmeanEj}{\bar{u}^{i,e,w}_{j,\mathcal{E} } }
\newcommand{\uiewmeanEk}{\bar{u}^{i,e,w}_{k,\mathcal{E} } }

\newcommand{\uiewmeanEjuno}{\bar{u}^{i,e,w}_{j_1,\mathcal{E} } }
\newcommand{\umeanF}{\bar{u}_\mathcal{F}}

\newcommand{\umeanieFj}{\bar{u}^{i,e}_{j,\mathcal{F}} }

\newcommand{\umeaniewFj}{\bar{u}^{i,e,w}_{j,\mathcal{F}} }

\newcommand{\umeaniewFk}{\bar{u}^{i,e,w}_{k,\mathcal{F}} }

\newcommand{\vnn}{v^{n+1}}
\newcommand{\uinn}{u_i^{n+1}}
\newcommand{\uenn}{u_e^{n+1}}
\newcommand{\vn}{v^n}

\newcommand{\uikk}{u_i^{k+1}}
\newcommand{\uekk}{u_e^{k+1}}

\newcommand{\uik}{u_i^k}
\newcommand{\uek}{u_e^k}
\newcommand{\uil}{u_{i,l}}
\newcommand{\uel}{u_{e, l}}
\newcommand{\wnn}{w^{n+1}}
\newcommand{\wn}{w^n}
\newcommand{\wkk}{w^{k+1}}
\newcommand{\wk}{w^k}
\newcommand{\si}{s_i}
\newcommand{\se}{s_e}
\newcommand{\sw}{s_w}
\newcommand{\sil}{s_{i,l}}
\newcommand{\sel}{s_{e,l}}
\newcommand{\swl}{s_{w,l}}

\newcommand{\sikk}{s_i^{k+1}}
\newcommand{\sekk}{s_e^{k+1}}
\newcommand{\swkk}{s_w^{k+1}}
\newcommand{\silkk}{s_{i,l}^{k+1}}
\newcommand{\selkk}{s_{e,l}^{k+1}}
\newcommand{\swlkk}{s_{w,l}^{k+1}}

\newcommand{\xx}{\mathbf{x}}

\newcommand{\aal}{\mathbf{a}_l}
\newcommand{\aat}{\mathbf{a}_t}
\newcommand{\aan}{\mathbf{a}_n}

\newcommand{\Pd}{P_D}
\newcommand{\Ed}{E_D}
\newcommand{\Sj}{S^{(j)}}
\newcommand{\Sk}{S^{(k)}}
\newcommand{\SjF}{S_{\mathcal{F}}^{(j)}}
\newcommand{\SkF}{S_{\mathcal{F}}^{(k)}}
\newcommand{\SjallE}{S_{\mathcal{E}}^{(j_{123})}}
\newcommand{\SjunoE}{S_{\mathcal{E}}^{(j_1)}}
\newcommand{\SjdueE}{S_{\mathcal{E}}^{(j_2)}}
\newcommand{\SjtreE}{S_{\mathcal{E}}^{(j_3)}}

\newcommand{\dvdt}{\frac{\partial v}{\partial t}}
\newcommand{\dwdt}{\frac{\partial w}{\partial t}}
\newcommand{\sumcoupdidv}{\sum_l \frac{\partial I_{\text{ion}}}{\partial v_l} (v,w)}
\newcommand{\sumcoupdidw}{\sum_l \frac{\partial I_{\text{ion}}}{\partial w_l} (v,w)}
\newcommand{\sumcoupdRdv}{\sum_l \frac{\partial R}{\partial v_l} (v,w)}
\newcommand{\sumcoupdRdw}{\sum_l \frac{\partial R}{\partial w_l} (v,w)}
\newcommand{\sumcoupdidvk}{\sum_l \frac{\partial I_{\text{ion}}}{\partial v_l} (v^k,w^k)}
\newcommand{\sumcoupdidwk}{\sum_l \frac{\partial I_{\text{ion}}}{\partial w_l} (v^k,w^k)}
\newcommand{\sumcoupdRdvk}{\sum_l \frac{\partial R}{\partial v_l} (v^k,w^k)}
\newcommand{\sumcoupdRdwk}{\sum_l \frac{\partial R}{\partial w_l} (v^k,w^k)}

\newcommand{\sumNx}[1]{\sum_{k \in \mathcal{N}_x} #1 }

\newcommand{\sumLNnod}[1]{\sum_{l = 1}^{N_h} #1}
\newcommand{\sumJN}[1]{\sum_{j = 1}^N #1}
\newcommand{\sumL}[1]{\sum_l #1}

\newcommand{\Iapp}{I_{\text{app}}}
\newcommand{\Ion}{I_{\text{ion}}}

\newcommand{\harmextie}[2]{\mathcal{H}^{i,e}_{#1} #2}

\newcommand{\innerprod}[2]{\left( #1, #2 \right)}

\newcommand{\bilformi}[2]{a_i \left( #1,  #2 \right)}
\newcommand{\bilforme}[2]{a_e \left( #1,  #2 \right)}
\newcommand{\norm}[2]{|| #1 ||_{L^2(#2)} }
\newcommand{\normB}[1]{| #1 |_{B} }
\newcommand{\normBtilde}[1]{| #1 |_{\widetilde{B}} }
\newcommand{\normBgamma}[1]{| #1 |_{B_\Gamma} }
\newcommand{\normBtildeGamma}[1]{| #1 |_{\widetilde{B}_\Gamma} }
\newcommand{\normBtildeGammaj}[1]{| #1 |_{\widetilde{B}_\Gamma^{(j)}} }

\newcommand{\normBj}[1]{| #1 |_{B^{(j)}} }

\newcommand{\seminormHone}[2]{| #1 |_{H^1(#2)} }

\newcommand{\seminormSj}[1]{| #1 |_{S_\Gamma^{(j)}} }

\newcommand{\seminormSjF}[1]{| #1 |_{S_\mathcal{F}^{(j)}} }
\newcommand{\seminormSkF}[1]{| #1 |_{S_\mathcal{F}^{(k)}} }
\newcommand{\seminormSjuno}[1]{| #1 |_{S_\Gamma^{(j_1)}} }

\newcommand{\normtaucoup}[1]{||| #1 |||_{\tau, coup} }

\newtheorem{theorem}{Theorem}[section]

\newtheorem{lemma}{Lemma}[section]

\newtheorem{remark}{Remark}[section]

\begin{document}
	
	\title{Newton-Krylov-BDDC deluxe solvers for non-symmetric fully implicit time discretizations of the Bidomain model}
	\author{
		Ngoc Mai Monica Huynh
		\thanks{Dipartimento di Matematica, Universit\`a degli Studi di Pavia, 
			Via Ferrata, 27100 Pavia, Italy. 
			E-mail: {\sf ngocmaimonica.huynh01@universitadipavia.it}
%			{\sf luca.pavarino@unipv.it}.
		}
%%		\and
%		Luca Franco Pavarino
%		\footnotemark[1],
%		%	\thanks{Dipartimento di Matematica, Universit\`a degli Studi di Pavia, 
%		%		Via Ferrata, 27100 Pavia, Italy.
%		%		E-mail: {\sf luca.pavarino@unipv.it}.
%		%		This work was supported by grants of MIUR (PRIN 2017AXL54F\_002) and
%		%		Istituto Nazionale di Alta Matematica (INDAM-GNCS).}
%%		\and
%		Simone Scacchi
%		\thanks{Dipartimento di Matematica, Universit\`a degli Studi di Milano,
%			Via Saldini 50, 20133 Milano, Italy.
%			E-mail: {\sf simone.scacchi@unimi.it}.
%		}
	}

\maketitle

\begin{abstract}
	A novel theoretical convergence rate estimate for a Balancing Domain Decomposition by Constraints algorithm is proven for the solution of the cardiac Bidomain model, describing the propagation of the electric impulse in the cardiac tissue. The non-linear system arises from a fully implicit time discretization and a monolithic solution approach. 
	The preconditioned non-symmetric operator is constructed from the linearized system arising within the Newton-Krylov approach for the solution of the non-linear problem; we theoretically analyze and prove a convergence rate bound for the Generalised Minimal Residual iterations’ residual. The theory is confirmed by extensive parallel numerical tests, widening the class of robust and efficient solvers for implicit time discretizations of the Bidomain model.
\end{abstract}

\pagestyle{myheadings}
\thispagestyle{plain}
%\markboth{N. M. M. Huynh, L. F. Pavarino, S. Scacchi}{Newton-Krylov-Dual-Primal Bidomain solvers}

%--------------------------------------------------------------------------
\section{Introduction}

In the last decade, the increasing need for understanding the intrinsic mechanisms behind cardiac diseases has resulted in the growth of inter-disciplinary studies, where, for example, physiological phenomena are translated into mathematical models \cite{carusi2012bridging, corrado2018work, niederer2019computational, quarteroni2017integrated, trayanova2011whole}. 
Modern medicine employs computational tools and models to simulate, predict and analyze risk situations with a non-invasive approach. An increasing number of studies, for example, have addressed the understanding of dysfunctions of the heart and the interaction between bio-electrical and mechanical phenomena \cite{colli2017effects, colli2019role, franzone2016joint, coronel2005right, mendonca2018modeling}. 

However, simulation of these models represents a tough challenge for personal laptops, as a huge amount of computational resources is required for numerical calculations. Even more so if one wants to use realistic heart geometries, where there is the need to represent accurately all the many facets.
For this reason, the employment of supercomputers and large scale system architectures in these studies has quickly spread (see for example works on cardiac mechanics \cite{colli2018numerical, klawonn2010highly, jiang2020highly}).

In this perspective, the present work seeks to design, theoretically analyze and validate numerically a Newton-Krylov solver for fully implicit time discretizations of the Bidomain model, preconditioned by a Balancing Domain Decomposition by Constraints (BDDC) algorithm.
This model consists of a degenerate parabolic system of two non-linear reaction diffusion Partial Differential Equations (PDEs), which describes the propagation of the electric signal in the cardiac tissue \cite{franzone2014mathematical, pennacchio2005multiscale, quarteroni2017integrated}. This system is coupled, by means of a non-linear reaction term, with a model of ionic current flows and associated gating variables, modeled as a system of Ordinary Differential Equations (ODEs).

The main contribution of this paper is a novel theoretical analysis for the convergence rate bound of the preconditioned non-symmetric operator coming from a coupled approach for the solution of the non-linear system arising from a fully implicit time discretization of the cardiac electrical model. 

Common alternatives in the literature use semi-implicit time discretizations \cite{colli2018numerical, zampini2014dual} and/or operator splitting \cite{chen2017splitting, chen2019two, sundnes2005operator}, as fully implicit schemes are more expensive from a computational point of view if complex and high-dimensional non-linear ionic models (e.g. \cite{di1985model, luo1991model, ten2004model}) are coupled with the Bidomain system. 
Other choices consist in decoupling strategies, where the ionic model is solved prior to the Bidomain, for example in the previous work of the Author \cite{huynh2021parallel} or in \cite{ dickopf2014design, munteanu2009decoupled, munteanu2009scalable, scacchi2011multilevel}.
In Ref. \cite{murillo2004fully}, an attempt at developing a solver for fully implicit time discretizations of the Bidomain model was done, in the framework of additive Schwarz preconditioners. 

In this work we extend this solution strategy to the class of dual-primal Domain Decomposition (DD) algorithms, with particular focus on BDDC preconditioners.

BDDC preconditioners were introduced by \cite{dohrmann2003preconditioner} as an alternative to FETI-DP (Dual-Primal Finite Elements Tearing and Interconnecting \cite{farhat2001feti}) methods for scalar elliptic problems and then analyzed by \cite{mandel2003convergence, mandel2005algebraic}. Within this field of applications, BDDC algorithms have been employed for the solution of the linearized semi-implicit Bidomain system in \cite{zampini2014dual, zampini2014inexact} and for cardiac mechanics in Refs. \cite{colli2018numerical, pavarino2015newton}. 

Instead of using a non-linear BDDC algorithm (such as the non-linear FETI-DP and BDDC proposed in \cite{klawonn2017nonlinear}) or, more generally, non-linear preconditioning strategies (e.g. \cite{liu2018note}), we present hereby Newton-Krylov-BDDC approach for the solution of a coupled solution strategy for fully implicit time discretizations of the Bidomain system, including the ionic model. At each time step we solve and update a non-linear problem, where the Jacobian system arising from the linearization of the non-linear problem is non-symmetric, thus forcing us to use a Generalized Minimal Residual (GMRES) \cite{saad1986gmres} method for its solution, preconditioned by BDDC algorithms in order to accelerate the convergence.

We propose a theoretical estimate of the convergence rate, based on the work in \cite{eisenstat1983variational} (which provided a theoretical bound for the residual of the GMRES iterations) and in \cite{tu2008balancing}, whose BDDC preconditioner addressed the solution of non-symmetric systems arising from the discretization of advection-diffusion PDEs. This analysis is enriched by the employment of the recently-introduced \emph{deluxe} scaling \cite{dohrmann2016bddc}.
The robustness and efficiency of the proposed solver is then confirmed by extensive parallel numerical tests on the Bidomain model, using the Portable, Extensible Toolkit for Scientific Computation (PETSc) library \cite{balay2019petsc}, thereby encouraging further investigations with realistic heart geometries and the tailoring of these kinds of solvers for the solution of the electro-mechanical model.

The work is structured as follows. In Sec. \ref{cardiacmodel}, we introduce the Bidomain system, describing the propagation of the electric signal in the cardiac tissue. In Sec. \ref{numerical methods} we give an insight into the space discretization and we formulate the fully implicit time scheme. Moreover we provide some properties related to the system arising from the discretization. A brief overview of non-overlapping DD spaces and objects as well as an introduction to BDDC preconditioner is provided in Sec. \ref{dual-primal methods}. The novel convergence rate estimate is then proved in Sec. \ref{convergence coupled}, followed by extensive parallel numerical tests in Section \ref{parallel tests}.

%% --------------------------

\section{The cardiac electrical model} \label{cardiacmodel}

\subsection{Bidomain model}
We consider here the macroscopic Bidomain representation of the cardiac tissue, which is represented as two interpenetrating domains \cite{franzone2014mathematical, pennacchio2005multiscale}.  
These two anisotropic continuous media, named intra- and extracellular domains, are assumed to coexist at every point of the cardiac tissue and to be connected by a distributed continuous cellular membrane which fills the complete volume. 

The cardiac tissue consists of a setting of fibers that rotates counterclockwise and that is arranged in laminar sheets running radially from the epi- to the endocardium (the outer and inner surface of the heart respectively).

At each point $\xx$ of the cardiac domain $\Omega$ it is possible to define an orthonormal triplet of vectors $\aal(\xx)$ parallel to the local fiber direction,  $\aat(\xx)$ and $\aan(\xx)$ tangent and orthogonal to the laminar sheets respectively and transversal to the fiber axis (\cite{legrice1995laminar}). Moreover, if $\sigma_{l, t, n}^{i,e}$ are conductivity coefficients in the intra- and extracellular domain along the corresponding direction, it is possible to define the conductivity tensors $D_i$ and $D_e$ of the two media as
\begin{equation*}
	D_i (\xx) = \sigma_l^i \aal (\xx) + \sigma_t^i \aat (\xx)  + \sigma_n^i \aan (\xx)  , 
	\qquad
	D_e (\xx) = \sigma_l^e \aal (\xx) + \sigma_t^e \aat (\xx)  + \sigma_n^e \aan (\xx) ,
\end{equation*}
which describe the anisotropy of the intra- and extracellular media. For our theoretical purpose, we assume that $\sigma_{l,t,n}^{i,e}$ are constant in space. 

With these premises, we can obtain the parabolic-parabolic formulation of the Bidomain model as the following non-linear parabolic reaction-diffusion system, 
\begin{align}\label{bido}
	\begin{dcases}
		\chi C_m  \dvdt - \text{div} \left( D_i \cdot \nabla u_i \right) + \Ion (v,w) = \Iapp^i	  & \text{in } \Omega \times [ 0,T ],	\\
		-\chi C_m  \dvdt - \text{div} \left( D_e \cdot \nabla u_e \right) - \Ion (v,w) = \Iapp^e	  &  \text{in } \Omega \times [ 0,T ],	\\
		\dwdt - R(v,w) = 0 		&\text{in } \Omega \times [ 0,T ], \\
		v(x,t) = u_i(x,t) - u_e(x,t) 		&\text{in } \Omega \times [ 0,T ], \\
	\end{dcases}
\end{align}
being $u_i$ and $u_e$ the intra and extracellular electric potential and $w$ the gating variables. 
The equations describing the propagation of the electric signal through the cardiac tissue are coupled through the reaction term to a system of Ordinary Differential Equations (ODEs) which describes the ionic currents flowing inward and outward the cell membrane.
Regarding the boundaries, we assume that the heart is electrically insulated by requiring zero-flux boundary conditions
\begin{equation*}
	\textbf{n}^T D_i \nabla u_i  = 0, \qquad \textbf{n}^T D_e \nabla u_e  = 0			\qquad \text{on } \partial \Omega \times [0,T]
\end{equation*}
and compatibility condition
$
%\begin{equation*}
\int_{\Omega} \Iapp^i dx = \int_{\Omega} \Iapp^e dx,
%\end{equation*}
$
where $\Iapp^{i,e}$ are the intra- and extracellular applied currents, with initial values
\begin{equation*}
	v(x,0) =   u_i(x,0) - u_e(x,0) = u_{i,0}(x) - u_{e,0}(x) \qquad w(x,0) = w_0(x).
\end{equation*}

See \cite{franzone2002degenerate} for results on existence, uniqueness and regularity of the solution of (\ref{bido}).

\subsection{Ionic current model}
In this work we consider a phenomenological ionic model, derived from a modification of the renowned FitzHugh--Nagumo model \cite{fitzhugh1961impulses, fitzhugh1969mathematical}: indeed, the Roger--McCulloch ionic model \cite{rogers1994collocation} overcomes the hyperpolarization of the cell during the repolarization phase by adding a non-linear dependence between the transmembrane potential and the gating. 
In this case, $\Ion(v,w)$ and $R(v,w)$ are given by
\begin{equation*}
	\Ion (v,w)  = G \ v \left( 1 - \dfrac{v}{v_{th}}  \right) \left( 1 - \dfrac{v}{v_{p}}  \right) + \eta_1 v w, 
	\qquad 
	R(v,w) = \eta_2 \left( \dfrac{v}{v_p} - w \right),
\end{equation*}
where $G$, $v_{th}$, $v_p$, $\eta_1$ and $\eta_2$ are constant coefficients.

\section{Space discretization and implicit time scheme}\label{numerical methods}

% ---------------------------------------------------%
\subsection{Weak formulation and space discretization} 
Let the cardiac domain $\Omega \subset \mathbb{R}^3$ be a bounded open Lipschitz set with Lipschitz continuous boundary. Consider the functional space $	V = H^1(\Omega) $ and define the elliptic bilinear form associated with the intra- and extracellular conductivity tensors
\begin{equation*}
	a_{i,e} \left( \varphi, \psi \right) = \int_{\Omega} D_{i,e} \nabla \varphi \cdot \nabla \psi, 
	\qquad
	\forall \varphi, \psi \in V.
\end{equation*}
Then, find $u_{i,e} \in  L^2(0,T; V) $ and $w \in L^2(0,T; L^2(\Omega))$ such that $\forall t \in (0,T)$ 
\begin{equation}\label{bidoweak}
	\begin{dcases*}
		%			\begin{align*}
		\chi C_m \dfrac{\partial}{\partial t} \innerprod{v}{\hat{u}_i} +  \bilformi{u_i}{\hat{u}_i} +  \innerprod{\Ion (v, w) }{\hat{u}_i} = \innerprod{I_{\text{app}}^i}{\hat{u}_i} 			\\
		-\chi C_m \dfrac{\partial}{\partial t} \innerprod{v}{\hat{u}_e} +  \bilforme{u_e}{\hat{u}_e} - \innerprod{\Ion (v, w) }{\hat{u}_e} = \innerprod{I_{\text{app}}^e}{\hat{u}_i} 	\\
		\dfrac{\partial}{\partial t} \innerprod{w}{\hat{w}} -  \innerprod{R(v, w)}{\hat{w}} = 0 		
		%			\end{align*}
	\end{dcases*}
\end{equation} 
$\forall \hat{u}_{i,e} \in V$ and $\forall \hat{w} \in L^2(\Omega)$.\\

The model (\ref{bidoweak}) is discretized in space by the finite element method, where the domain $\Omega$ is discretized by a structured quasi-uniform grid of hexaedral isoparametric $Q_1$ elements. Denote by $V_h \subset V$ be the associated finite element space, with the same basis functions $\left\{ \varphi_p \right\}_{p=1}^{N_h}$ for all variables $u_{i,e}$ and $w$ and denote by $A_{i,e}$ and $M$ be the stiffness and mass matrices with entries
\begin{equation}\label{stiffandmass}
	\left\{ A_{i,e} \right\}_{nm} = \int_{\Omega} \left(  \nabla \varphi_n  \right)^T D_{i,e} \cdot \nabla \varphi_m , 	\qquad 		\left\{ M \right\}_{nm} = \int_{\Omega} \varphi_n \varphi_m .
\end{equation}
The non-linear term is approximated by ionic current interpolation
\begin{equation*}
	\innerprod{\Ion (v,w) }{\varphi_p}  = \sumLNnod{ \Ion (v_l, w_l)} \innerprod{\varphi_l}{\varphi_p}.
\end{equation*}
With these choices, we thus need to solve at each time step, the semi-discrete Bidomain model
\begin{equation*}
	\begin{dcases*}
		\displaystyle
		\chi C_m \mathcal{M} \dfrac{\partial}{\partial t} 
		\begin{pmatrix}
			\bm{u}_i \\ \bm{u}_e
		\end{pmatrix}
		+ \mathcal{A} 
		\begin{pmatrix}
			\bm{u}_i \\ \bm{u}_e
		\end{pmatrix}
		+ 
		\begin{pmatrix}
			M \ \Ion (\bm{v}, \bm{w}) \\ -M \ \Ion (\bm{v}, \bm{w}) 
		\end{pmatrix}
		= 
		\begin{pmatrix}
			M \ \Iapp^i  \\ -M \ \Iapp^e
		\end{pmatrix},
		\\
		\displaystyle
		\dfrac{\partial \bm{w}}{\partial t} = R \left( \bm{v}, \bm{w} \right),
	\end{dcases*} 
\end{equation*}
where 
\begin{equation*}
	\mathcal{A} =
	\begin{bmatrix}
		A_i 	& 0 	\\
		0 		& A_e
	\end{bmatrix},
	\qquad
	\mathcal{M} = 
	\begin{bmatrix}
		M 		& -M 	\\
		-M 		& M
	\end{bmatrix}.
\end{equation*}

For simplicity, from now on we write
$
\sumL{} := \sumLNnod{}
$
assuming that we are adding contributions from all the $N_h$ nodes of the discretization.

% ---------------------------------------------------%
\subsection{Fully implicit time scheme}
In the literature, common alternatives takes into account implicit-explicit (IMEX) time discretization schemes \cite{colli2018numerical, zampini2014dual}, where the diffusion term is treated implicitly while the remaining terms are treated explicitly, or, more generally, operator splitting \cite{chen2017splitting, chen2019two, sundnes2005operator}. Others effective strategies rely on a decoupling strategy (see e.g. Refs. \cite{dickopf2014design, huynh2021parallel, munteanu2009decoupled, munteanu2009scalable, scacchi2011multilevel}) where at each time step the microscopic and macroscopic models are solved successively. We propose here a fully implicit time discretization of the Bidomain system, in the same fashion as in \cite{murillo2004fully}.
At the $n$-th time step,

\begin{enumerate}
	\item compute the intra- and extracellular potentials as well as the gating by solving the non-linear system \mbox{$F_{bido}(\uinn, \uenn, \wnn) = 0$} derived from the Backward Euler scheme applied to the Bidomain system, 
	\begin{adjustwidth*}{-5mm}{0mm}
		\begin{align}\label{nonlinsysbidocoupled}
			F (&\uinn, \ \uenn, \ \wnn) := 
			\begin{pmatrix}
				F_1 (\uinn, \uenn, \wnn) \\
				F_2 (\uinn, \uenn, \wnn) \\
				F_3 (\uinn, \uenn, \wnn)
			\end{pmatrix} \\
			&=\begin{pmatrix}
				\chi C_m \innerprod{\vnn}{\varphi_i} + \tau \bilformi{\uinn}{\varphi_i} + \tau \innerprod{\Ion (\vnn, w) }{\varphi_i} \\
				\qquad \qquad \qquad \qquad \qquad \qquad \qquad \qquad - \left[ \chi C_m \innerprod{\vn}{\varphi_i} + \tau \innerprod{\Iapp}{\varphi_i} \right]
				\\
				-\chi C_m \innerprod{\vnn}{\varphi_e} + \tau \bilforme{\uenn}{\varphi_e} - \tau \innerprod{\Ion (\vnn, w) }{\varphi_e} \\
				\qquad \qquad \qquad \qquad \qquad \qquad \qquad \qquad - \left[ \chi C_m \innerprod{\vn}{\varphi_e} + \tau \innerprod{\Iapp}{\varphi_e} \right]
				\\
				\innerprod{\wnn}{\varphi_w} - \tau \innerprod{R(\vnn, \wnn)}{\varphi_w} - \innerprod{\wn}{\varphi_w}
			\end{pmatrix}, \nonumber
		\end{align}
	\end{adjustwidth*}
	where $\varphi_i, \varphi_e$ and $\varphi_w$ are the test functions related to $u_i$, $u_e$ and $w$ respectively.
	
	\begin{enumerate}
		\item[2.1] Apply an exact Newton method for the solution of the nonlinear system (\ref{nonlinsysbidocoupled}); given the initial guess $(u_i^0, u_e^0, w^0)$, at the $k^{th}$ iteration of the Newton loop, solve the linear system of equations
		\begin{align}\label{jacobcoupledbido}
			\begin{dcases}
				\sumL{ \dfrac{\partial F_1}{\partial \uil}} (\uik, \uek, \wk) \silkk + \sumL{ \dfrac{\partial F_1}{\partial \uel} } (\uik, \uek, \wk) \silkk + \\
				\qquad \qquad \qquad \qquad + \sumL{ \dfrac{\partial F_1}{\partial w_l} } (\uik, \uek, \wk) \silkk =  -F_1(\uik, \uek, \wk) \\
				\sumL{ \dfrac{\partial F_2}{\partial \uil}} (\uik, \uek, \wk) \selkk + \sumL{ \dfrac{\partial F_2}{\partial \uel} } (\uik, \uek, \wk) \selkk + \\
				\qquad \qquad \qquad \qquad + \sumL{ \dfrac{\partial F_2}{\partial w_l} } (\uik, \uek, \wk) \selkk =  -F_2(\uik, \uek, \wk) \\
				\sumL{ \dfrac{\partial F_3}{\partial \uil}} (\uik, \uek, \wk) \swlkk + \sumL{ \dfrac{\partial F_3}{\partial \uel} } (\uik, \uek, \wk) \swlkk + \\
				\qquad \qquad \qquad \qquad + \sumL{ \dfrac{\partial F_3}{\partial w_l} } (\uik, \uek, \wk) \swlkk =  -F_3(\uik, \uek, \wk) 
			\end{dcases},
		\end{align}
		explicitly written as
		\begin{adjustwidth*}{0mm}{-10mm}
			\begin{align*}
				\begin{dcases}
					\chi C_m \innerprod{\sikk - \sekk}{\varphi_i}  + \tau \innerprod{\sumcoupdidvk  \left( \silkk - \selkk \right) \psi_l}{\varphi_i}  \\
					\qquad \qquad + \tau \bilformi{\sikk}{\varphi_i} + \tau \innerprod{\sumcoupdidwk \swlkk \ \psi_l}{\varphi_i} = -  F_1(\uik, \uek, \wk) \\
					-\chi C_m \innerprod{\sikk - \sekk}{\varphi_e} - \tau \innerprod{\sumcoupdidvk \left( \silkk - \selkk \right)\psi_l}{\varphi_e}  \\
					\qquad \qquad + \tau \bilforme{\sekk}{\varphi_e}  - \tau \innerprod{\sumcoupdidwk \swlkk \ \psi_l}{\varphi_e} = -  F_2(\uik, \uek, \wk) \\
					\innerprod{\swkk}{\varphi_k} - \tau \innerprod{\sumcoupdRdvk \left( \silkk - \selkk \right) \psi_l }{\varphi_w} \\
					\qquad 	\qquad- \tau \innerprod{\sumcoupdRdwk \swkk \ \psi_l}{\varphi_w} = -  F_3(\uik, \uek, \wk) 
				\end{dcases}
			\end{align*}
		\end{adjustwidth*}
		where 
		\begin{equation*}
			\sikk = \sumL{\silkk} = \delta \uikk ,		 \quad 		\sekk = \sumL{\selkk} = \delta \uekk,		\quad			\swkk = \sumL{\swlkk} = \delta \wkk
		\end{equation*}
		are the increments at time step $k$ and $\psi_l$ the $l$-th nodal basis function. \\
		In matricial form, this means to solve the linear system 
		\begin{eqnarray}\label{Jacobiancoupledsystem}
			\mathbf{JF}^k \mathbf{s}^{k+1} = - \mathbf{F} (\mathbf{u}^k)
		\end{eqnarray}
		where, by implying $\dfrac{\partial g^k}{\partial t} = \dfrac{\partial g}{\partial t} (\mathbf{v}^k, \mathbf{w}^k)  $, with $g = \left\{ \Ion, R \right\}$ and $t = \left\{ v, w \right\}$,
		\begin{adjustwidth*}{0mm}{-10mm}
			\begin{equation*}\label{jacobiancoupledmatrix}
				\mathbf{JF}^k =	
				\begin{bmatrix}
					\chi C_m M + \tau A_i + \tau M \dfrac{\partial \Ion^k}{\partial v}  		&- \chi C_m M - \tau M \dfrac{\partial \Ion^k}{\partial v} 	& \tau M \dfrac{\partial \Ion^k}{\partial w}  \\
					- \chi C_m M - \tau M \dfrac{\partial \Ion^k}{\partial v} 				&\chi C_m M + \tau A_e + \tau M \dfrac{\partial \Ion^k}{\partial v} 		&-\tau M \dfrac{\partial \Ion^k}{\partial w} \\
					-\tau M \dfrac{\partial R^k}{\partial v} 			 	& \tau M \dfrac{\partial R^k}{\partial v}   	& \left( 1 - \tau \dfrac{\partial R^k}{\partial w}  \right) M
				\end{bmatrix},
			\end{equation*}
		\end{adjustwidth*}
		\begin{equation*}
			\mathbf{s}^{k+1} = 
			\begin{pmatrix}
				\mathbf{\si}^{k+1} \\ \mathbf{\se}^{k+1} \\ \mathbf{w}^{k+1}
			\end{pmatrix},
			\qquad 
			\mathbf{F} (\mathbf{u}^k) = 
			\begin{pmatrix}
				- M  F_1(\mathbf{u_i}^k, \mathbf{u_e}^k, \mathbf{w}^k)   \\ -M  F_2( \mathbf{u_i}^k, \mathbf{u_e}^k , \mathbf{w}^k) \\  -M  F_3( \mathbf{u_i}^k, \mathbf{u_e}^k , \mathbf{w}^k) 
			\end{pmatrix},
		\end{equation*}
		with the same stiffness and mass matrices defined in (\ref{stiffandmass}).
		\item[2.2] Update
		\begin{equation*}
			\uikk = \uik + \sikk, 	\qquad   \uekk = \uek + \sekk, 	\qquad	\wkk = \wk + \swkk.
		\end{equation*}
	\end{enumerate}
\end{enumerate}

We drop the index $k$ from now on, unless an explicit ambiguity occurs.

\subsection{Properties of the symmetric part of the bilinear form associated with the Bidomain Jacobian system}
The Jacobian system $\mathbf{JF}$ in (\ref{Jacobiancoupledsystem}) is non-symmetric, due to the inclusion of the ionic model. For this reason, the iterative solver must be addressed to the solution of such type of systems, such as the Generalized Minimal Residual (GMRES) method \cite{saad1986gmres}. 

Following the work of  \cite{tu2008balancing} (where BDDC preconditioners are applied to the solution of non-symmetric problems arising from the discretization of advection-diffusion PDEs) in order to properly prove the convergence rate estimate of the solver, we need to associate to the linear system (\ref{Jacobiancoupledsystem}) a bilinear form and to analyze its symmetric and skew-symmetric parts. 

We reformulate problem (\ref{jacobcoupledbido}) in variational form: find $s = \left( \si, \se, \sw \right) \in \mathbf{V_h}$, being $\mathbf{V_h} = V_h \times V_h \times V_h$,  such that
\begin{equation}\label{bilcoupledform}
	a (s, \phi) = -F_1 (s)  - F_2 (s) - F_3 (s) 	\qquad 	\forall \phi = \left( \varphi_i, \varphi_e, \varphi_w \right) \in \mathbf{V_h},
\end{equation}
where
\begin{equation*} 
	\begin{aligned}
		a (s, &\phi) = \chi C_m \innerprod{\si - \se}{\varphi_i - \varphi_e} + \innerprod{\sw}{\varphi_w} +  \tau \bilformi{\si}{\varphi_i} + \tau \bilforme{\se}{\varphi_e} \\
		& + \tau \innerprod{\sumcoupdidv \left( \sil - \sel \right) \psi_l}{\varphi_i - \varphi_e} + \tau \innerprod{\sumcoupdidw \swl \ \psi_l}{\varphi_i - \varphi_e} \\
		& - \tau \innerprod{\sumcoupdRdv \left( \sil - \sel \right) \psi_l}{\varphi_w} - \tau \innerprod{\sumcoupdRdw \swl \ \psi_l}{\varphi_w}
	\end{aligned}
\end{equation*}
being $\psi_l$ the $l$-th nodal basis function.
The symmetric and skew-symmetric parts of $a(s, \phi)$ respectively are denoted by
\begin{equation*}
	\begin{aligned}
		b (s, \phi) &= 2 \chi C_m \innerprod{\si - \se}{\varphi_i - \varphi_e} + 2 \innerprod{\sw}{\varphi_w} +  2 \tau \bilformi{\si}{\varphi_i} + 2 \tau \bilforme{\se}{\varphi_e} \\
		& + 2 \tau \innerprod{\sumcoupdidv \left( \sil - \sel \right) \psi_l}{\varphi_i - \varphi_e} - 2 \tau \innerprod{\sumcoupdRdw \swl \ \psi_l}{\varphi_w} \\
		& + \tau \innerprod{\sumcoupdidw \left( \sil - \sel \right) \psi_l}{\varphi_w} +  \tau \innerprod{\sumcoupdidw \swl \ \psi_l}{\varphi_i - \varphi_e} \\
		& - \tau \innerprod{\sumcoupdRdv \left( \sil - \sel \right) \psi_l}{\varphi_w} -  \tau \innerprod{\sumcoupdRdv \swl \ \psi_l}{\varphi_i - \varphi_e} ,
	\end{aligned}
\end{equation*}
%and
\begin{equation*}
	\begin{aligned}
		z (s, &\phi) = - \tau \innerprod{\sumcoupdidw \left( \sil - \sel \right) \psi_l}{\varphi_w} +  \tau \innerprod{\sumcoupdidw \swl \ \psi_l}{\varphi_i - \varphi_e} \\
		& - \tau \innerprod{\sumcoupdRdv \left( \sil - \sel \right) \psi_l}{\varphi_w} +  \tau \innerprod{\sumcoupdRdv \swl \ \psi_l}{\varphi_i - \varphi_e} 
	\end{aligned}
\end{equation*}
In the same way the system of linear equations (\ref{jacobcoupledbido}) correspond to the finite element problem (\ref{bilcoupledform}), we can denote by $B$ and $Z$ the symmetric and skew-symmetric parts of $\mathbf{JF}$, which correspond to the bilinear forms $b(\cdot, \cdot)$ and $z(\cdot, \cdot)$ respectively:
\begin{equation*}
	B = 
	\begin{bmatrix}
		2 \left( \chi C_m M + \tau A_i + \tau M \dfrac{\partial \Ion}{\partial v} \right)		&2 \left( - \chi C_m M - \tau M \dfrac{\partial \Ion}{\partial v} \right)	& \tau M \left( \dfrac{\partial \Ion}{\partial w} - \dfrac{\partial R}{\partial v} \right) \\
		2 \left( - \chi C_m M - \tau M \dfrac{\partial \Ion}{\partial v} \right) 	&2 \left( \chi C_m M + \tau A_e + \tau M \dfrac{\partial \Ion}{\partial v} \right)		&-\tau M \left( \dfrac{\partial \Ion}{\partial w} - \dfrac{\partial R}{\partial v} \right) \\
		\tau M \left( \dfrac{\partial \Ion}{\partial w} - \dfrac{\partial R}{\partial v} \right) 	& - \tau M \left( \dfrac{\partial \Ion}{\partial w} - \dfrac{\partial R}{\partial v}\right) 	&2  \left( 1 - \tau \dfrac{\partial R}{\partial w}  \right) M
	\end{bmatrix}
\end{equation*}
\begin{equation*}
	Z = 
	\begin{bmatrix}
		0 		&0 			&\tau M \left( \dfrac{\partial \Ion}{\partial w} + \dfrac{\partial R}{\partial v}	\right)		\\
		0 		&0			&-\tau M \left( \dfrac{\partial \Ion}{\partial w} + \dfrac{\partial R}{\partial v} \right)	\\
		-\tau M \left( \dfrac{\partial \Ion}{\partial w} + \dfrac{\partial R}{\partial v}	\right)	&\tau M \left( \dfrac{\partial \Ion}{\partial w} + \dfrac{\partial R}{\partial v}	\right)	&0
	\end{bmatrix},
\end{equation*}
where we simplify the notation by writing $\dfrac{\partial g}{\partial t} := \dfrac{\partial g}{\partial t}  (\mathbf{v}, \mathbf{w})$, with $g = \left\{ \Ion, R \right\}$ and $t = \left\{ v, w \right\}$. \\

In the same spirit as in \cite{huynh2021parallel, munteanu2009decoupled}, it is possible to show that $b(\cdot, \cdot)$ is continuous and coercive with respect to an appropriate norm. 

\begin{lemma} \label{symbilformcontinuity}
	Assume that 
	\begin{equation*}
		\chi C_m + \tau \dfrac{\partial \Ion}{\partial v_l} (v, w) \geq c_1, \quad 1 - \tau \dfrac{\partial R}{\partial w_l} (v, w) \geq c_2, \quad \dfrac{\partial \Ion}{\partial w_l} (v, w) - \dfrac{\partial R}{\partial v_l} (v, w) \geq 0,   		
	\end{equation*}
	$c_{1,2} \in \mathbb{R}^+$ and $\forall l = 1, \dots, N$.
	Then the bilinear form $b(\cdot, \cdot)$ is continuous and coercive with respect to the norm $\normtaucoup{\cdot}$, defined as
	\begin{equation*}
		\normtaucoup{u}^2 := (1 + \tau) \norm{u_1 - u_2}{\Omega}^2 + (1 - \tau) \norm{u_3}{\Omega}^2 + \tau \bilformi{u_1}{u_1} + \tau \bilforme{u_2}{u_2},
	\end{equation*}
	$\forall u = (u_1, u_2, u_3) \in \mathbf{V_h}$.
\end{lemma}

\begin{remark}
	\emph{The norm $\normtaucoup{\cdot}$ is well defined, as the quantity $1-\tau$ is always positive (typical computational values for $\tau$ are less than $10^{-2}$).}
\end{remark}

\begin{remark}
	\emph{As in the case of the decoupled strategy (see \cite{huynh2021parallel, munteanu2009decoupled}), the hypothesis of non-negativity of the above Lemma is always satisfied for any time step $\tau \leq 0.37$ ms if we consider the Roger-McCulloch ionic model. Indeed, numerical computations of $\chi C_m + \tau \dfrac{\partial I_\text{ion}}{\partial v}$ validate this assumption (see Fig. (\ref{hypcoupplot}), above). Regarding the other two hypothesis, it is easy to compute analytically that $ 1 - \tau \dfrac{\partial R}{\partial w_l} = 1 + \eta_2 \tau \geq 0$ for any value of $\tau$, being $\eta_2$ a physiological parameter, while the last inequality is always satisfied for any $v \geqslant 2 \cdot 10^{-4}$. This request is not restrictive, as for those values of the transmembrane potential the tissue is almost at rest. }
\end{remark}

\begin{figure}[H]
	\centering
	\includegraphics[scale=.35]{ 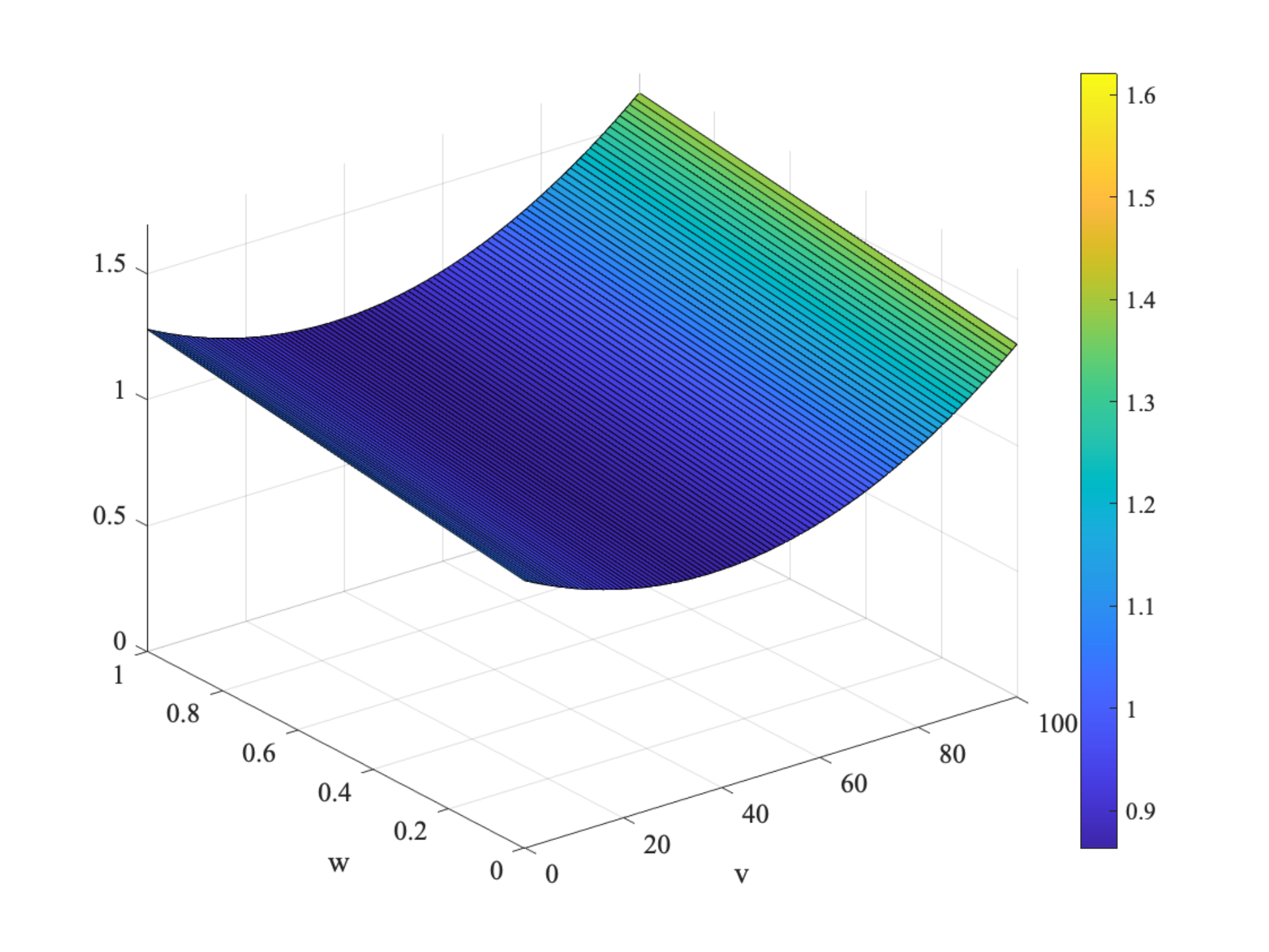}
	\includegraphics[scale=.35]{ 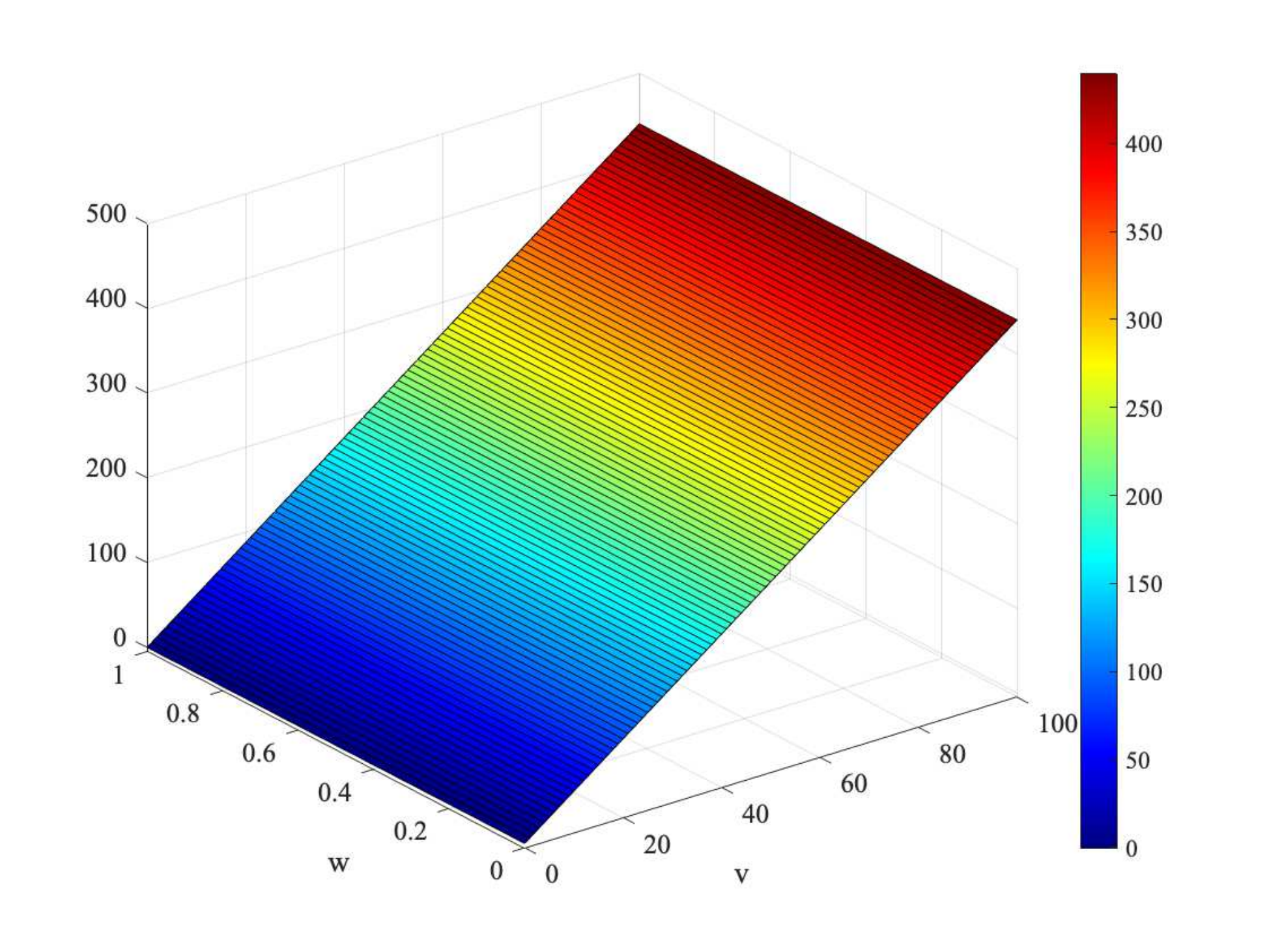}
	\caption{Surface plots of $\chi C_m + \tau \frac{\partial I_\text{ion}}{\partial v}$ (top) and $\frac{\partial \Ion}{\partial w} - \frac{\partial R}{\partial v}$ (bottom), with $C_m = 1 \frac{mF}{cm^3}$, $\chi = 1$ and $\tau=0.05$ ms, which are values usually employed in numerical experiments.}
	\label{hypcoupplot}
\end{figure}

As an immediate consequence of the continuity and coercivity of the symmetric bilinear form $b(\cdot, \cdot)$, it is possible to prove the following bounds.

\begin{lemma}\label{ellipcoupbound}
	Assuming that the conductivity coefficients are constant in space, the bilinear form $b(\cdot, \cdot)$ satisfies the bounds
	\begin{align*}
		b (s,s) &\leq 2 \left[ \left( \chi C_m + \tau K_{M,I} \right) \norm{\si - \se}{\Omega}^2 + (1 - \tau K_{M, R}) \norm{\sw}{\Omega}^2 \right. \\ &\left. \qquad \qquad \qquad\qquad \qquad  \qquad \qquad + \tau \sigma^i_M \seminormHone{\si}{\Omega}^2 + \tau \sigma^e_M \seminormHone{\se}{\Omega}^2 \right], \\
		b (s, s) &\geq 2 \left[ \left( \chi C_m + \tau K_{m,I} \right) \norm{\si - \se}{\Omega}^2 + (1 - \tau K_{m, R}) \norm{\sw}{\Omega}^2 \right. \\ &\left. \qquad \qquad \qquad \qquad \qquad \qquad \qquad + \tau \sigma^i_m \seminormHone{\si}{\Omega}^2 + \tau \sigma^e_m \seminormHone{\se}{\Omega}^2 \right]
	\end{align*}
	where 
	\begin{equation*}
		\sigma^{i,e}_M = \max_{\bullet = \left\{ l, t, n \right\}} \sigma^{i,e}_\bullet, 		\qquad \qquad  \sigma^{i,e}_m = \min_{\bullet = \left\{ l, t, n \right\}} \sigma^{i,e}_\bullet,
	\end{equation*}
	and $K_{M, \ast}$, $K_{m, \ast}$, independent from the subdomain diameter $H$ and the mesh size $h$.
\end{lemma}

\begin{remark}
	\emph{This result is extensible to the case of conductivity coefficients almost constant over each subdomain.}
\end{remark}

\section{Dual-Primal Iterative Substructuring Methods}\label{dual-primal methods}

% ---------------------------------------------------%
\subsection{Non-overlapping Dual-Primal Algorithms} 
Let us decompose the cardiac domain $\Omega$ into $N$ non-overlapping subdomains $\Omega_j$, with $j=1,\dots,N$, such that
$
%\begin{equation*}
\overline{\Omega} = \cup_{j=1}^{N} \overline{\Omega}_j, 
\Omega_j \cap \Omega_k = \emptyset, 
\text{if }  j \neq k,
%\end{equation*}
$
and the intersection between boundaries from different subdomains is either empty, a vertex, an edge or a face.
We define the interface $\Gamma$ as the set of points that belong to at least two subdomains, 
\begin{equation*}
	\Gamma := \cup_{j \neq k} \pO_j \cap \pO_k,
\end{equation*}
where $\pO_j$, $\pO_k$ are the boundaries of $\Omega_j$ and $\Omega_k$ respectively. 
To our purposes, we assume the subdomains to be shape-regular with a typical diameter of size $H$; moreover let them be the union of shape-regular finite elements of diameter $h$.

Denote the associated local finite element spaces by $W_j$ and partition it into its interior part $W_I^{(j)}$ and the finite element trace space $W_\Gamma^{(j)}$, such that
\begin{equation*}
	W_j =  W_I^{(j)} \oplus W_\Gamma^{(j)}.
\end{equation*}
In this dissertation, we consider variables on the Neumann boundaries $\pO_N$ as interior to a subdomain. By introducing the product spaces by
\begin{equation*}
	W := W_1 \times \cdots \times W_N = \prod_{j=1}^{N} W_j  ,	
	\qquad
	W_\Gamma :=  \prod_{j=1}^{N} W_\Gamma^{(j)},
\end{equation*}
we define $\widehat{W} \subset W$ as the subspace of functions of $W$, which are continuous in all interface variables between subdomains. Similarly we denote by $\widehat{W}_\Gamma \subset W_\Gamma$, the subspace formed by the continuous elements of $W_\Gamma$. 

The idea of dual-primal algorithms is to solve iteratively in the space $W$, while imposing continuity constraints (\emph{primal constraints}). 
Let $\widetilde{W}$ be the space of finite element functions in $W$, continuous in all primal variables, such that $\widehat{W} \subset \widetilde{W} \subset W$ and likewise $\widehat{W}_\Gamma \subset \widetilde{W}_\Gamma \subset W_\Gamma$. 

Let $W_\Pi^{(j)} \subset W_\Gamma^{(j)}$ be the primal subspace of functions, continuous across the interface, that will be subassembled between the subdomains that share $\Gamma^{(j)}$. We denote by \emph{dual}, the subspace \mbox{$W_\Delta^{(j)} \subset 	W_\Gamma^{(j)}$} which contains the finite element functions that can be discontinuous across the interface and which vanish at the primal degrees of freedom.
Let $W_\Pi$ and $W_\Delta$ be two subspaces such that 
\begin{equation*}
	W_\Pi = \prod_{j=1}^{N} W_\Pi^{(j)}, \qquad W_\Delta = \prod_{j=1}^{N} W_\Delta^{(j)}
\end{equation*}
and
$
W_\Gamma = W_\Pi \oplus W_\Delta.
$

With this notation, it is possible to decompose $\widetilde{W}_\Gamma$ into a primal subspace $\widehat{W}_\Pi$ which has continuous elements only and a dual subspace $W_\Delta$ which contains finite element functions which are not continuous, 
\begin{equation*}
	\widetilde{W}_\Gamma = \widehat{W}_\Pi \oplus W_\Delta,
\end{equation*}
and in the same fashion
\begin{equation*}
	\widehat{W} = \widehat{W}_\Pi \oplus \widehat{W}_\Delta \oplus W_I.
\end{equation*}

In this paper we will denote with subscripts $I$, $\Delta$ and $\Pi$ the interior, the dual and the primal variables respectively. \\

In the non-overlapping framework, the global system matrix Eq. (\ref{Jacobiancoupledsystem}) is never formed explicitly, but a local matrix with the same structure is assembled on each subdomain, by restricting the integration set and by defining the local bilinear forms 
\begin{align*}
	&a^{(j)} (s, \phi) = \chi C_m \innerprod{\si - \se}{\varphi_i - \varphi_e}_{|\Omega_j} + \innerprod{\sw}{\varphi_w}_{|\Omega_j} +  \tau a_i^{(j)} \left( \si, \varphi_i \right) + \tau a_e^{(j)} \left( \se, \varphi_e \right)\\
	& + \tau \innerprod{\sumcoupdidv \left( \sil - \sel \right) \psi_l}{\varphi_i - \varphi_e}_{|\Omega_j} 
	- \tau \innerprod{\sumcoupdRdw \swl \ \psi_l}{\varphi_w}_{|\Omega_j} \\
	&+ \tau \innerprod{\sumcoupdidw \swl \ \psi_l}{\varphi_i - \varphi_e}_{|\Omega_j} 
	- \tau \innerprod{\sumcoupdRdv \left( \sil - \sel \right) \psi_l}{\varphi_w}_{|\Omega_j} 
\end{align*}
and the symmetric and skew-symmetric counterparts
\begin{align*}
	b&^{(j)} (s, \phi) = 2 \chi C_m \innerprod{\si - \se}{\varphi_i - \varphi_e}_{\Omega_j}  + 2 \innerprod{\sw}{\varphi_w}_{\Omega_j}  +  2 \tau a_i^{(j)} \left( \si, \varphi_i \right) + 2 \tau a_e^{(j)} \left( \se, \varphi_i \right) \\
	& + 2 \tau \innerprod{\sumcoupdidv \left( \sil - \sel \right) \psi_l}{\varphi_i - \varphi_e}_{\Omega_j}  - 2 \tau \innerprod{\sumcoupdRdw \swl \ \psi_l}{\varphi_w}_{\Omega_j}  \\
	& + \tau \innerprod{\sumcoupdidw \left( \sil - \sel \right) \psi_l}{\varphi_w}_{\Omega_j}  +  \tau \innerprod{\sumcoupdidw \swl \ \psi_l}{\varphi_i - \varphi_e}_{\Omega_j}  \\
	& - \tau \innerprod{\sumcoupdRdv \left( \sil - \sel \right) \psi_l}{\varphi_w}_{\Omega_j}  -  \tau \innerprod{\sumcoupdRdv \swl \ \psi_l}{\varphi_i - \varphi_e}_{\Omega_j} \\
	%\end{align*}
	%\begin{align*}
	z&^{(j)} (s, \phi) = \\ &\tau \innerprod{\sumcoupdidw \swl \ \psi_l}{\varphi_i - \varphi_e}_{\Omega_j}  - \tau \innerprod{\sumcoupdidw \left( \sil - \sel \right) \psi_l}{\varphi_w}_{\Omega_j}   \\
	& \tau \innerprod{\sumcoupdRdv \swl \ \psi_l}{\varphi_i - \varphi_e}_{\Omega_j} - \tau \innerprod{\sumcoupdRdv \left( \sil - \sel \right) \psi_l}{\varphi_w}_{\Omega_j}  
\end{align*}
where $\innerprod{\cdot\ }{\cdot}_{\Omega_j}$ denotes the restriction of the $L^2$-inner product to the $j$-th subdomain. 

As the proposed theory allows constant non-negative distribution of the diffusion coefficients among all subdomains, with large jumps aligned to the interfaces, these definitions are valid here.\\

Following the workflow in \cite{tu2008balancing}, after we introduce the space $\widetilde{W}$ of partially subassembled finite element space, we can define the corresponding bilinear forms by
\begin{equation*}
	\widetilde{a} (s, \phi) = \sumJN{a^{(j)} (s, \phi)}, 
	\qquad
	\widetilde{b} (s, \phi) = \sumJN{b^{(j)} (s, \phi)}, 
	\qquad
	\widetilde{z} (s, \phi) = \sumJN{z^{(j)} (s, \phi)}, 
\end{equation*}
We denote the partially subassembled matrices corresponding to the bilinear forms above with $\widetilde{A}$, $\widetilde{B}$ and $\widetilde{Z}$ respectively, and let
\begin{equation*}
	A = \widetilde{R}^T \widetilde{A} \widetilde{R},
	\qquad
	B = \widetilde{R}^T \widetilde{B} \widetilde{R},
	\qquad
	Z = \widetilde{R}^T \widetilde{Z} \widetilde{R},
\end{equation*}
being $\widetilde{R}$ the injection operator from $\widehat{W}$ to $\widetilde{W}$. 

In the same fashion as in \cite{tu2008balancing}, we define the truncated norms on the space $\widetilde{W}$
\begin{equation*}
	\norm{w}{\Omega}^2 = \sumJN{ \norm{w}{\Omega_j}^2 },
	\qquad
	\seminormHone{w}{\Omega}^2 = \sumJN{ \seminormHone{w}{\Omega_j}^2 },
	\qquad \qquad
	\forall w \in \widetilde{W}.
\end{equation*}
In this work $\norm{w}{\Omega}$ and $\seminormHone{w}{\Omega}$ for $w \in \widetilde{W}$ always represent these truncated norms.
As for construction the bilinear forms $b^{(j)} (\cdot, \cdot)$ for $j = 1, \dots, N$ are symmetric and positive definite on $W^{(j)}$, it is possible to define
\begin{equation*}
	\normBj{ u }^2 = b^{(j)} (u, u),
	\qquad
	\forall u \in W^{(j)}
\end{equation*}
and
\begin{equation*}
	\normB{u}^2 = \sumJN{ \normBj{u}^2 },
	\qquad
	\forall u \in \widehat{W},
	%\end{equation*}
	%\begin{equation*}
	\qquad \qquad
	\normBtilde{u}^2 = \sumJN{ \normBj{u}^2 },
	\qquad
	\forall u \in \widetilde{W}.
\end{equation*}

In dual-primal methods, the reordering of the degrees of freedom, lead to consider a reordered system matrix: assuming that the system (\ref{Jacobiancoupledsystem}) can be written as $\mathcal{K} u = f$, then it is equivalent to write
\begin{equation*}
	\mathcal{K}^{(j)} = 
	\begin{bmatrix}
		K_{II}^{(j)} 					& K_{I\Gamma}^{(j)} \\
		K_{I \Gamma}^{(j) T}  	& K_{\Gamma \Gamma}^{(j)}
	\end{bmatrix},
	\qquad
	\mathcal{K} = 
	\begin{bmatrix}
		K_{II} 					& K_{I\Gamma} \\
		K_{I \Gamma} 	& K_{\Gamma \Gamma}
	\end{bmatrix},
\end{equation*}
where 
$
K_{II} = diag \left[ K_{II}^{(1)}, \dots, K_{II}^{(N)} \right]
%	\begin{bmatrix}
%		K_{II}^{(1)}					&							& 			\\
%		& \ddots				&			\\
%		&							& K_{II}^{(N)}
%	\end{bmatrix}.
$
is a block-diagonal matrix.
As in many iterative substructuring algorithms, we eliminate all the interior variables (step known as \emph{static condensation}), obtaining the local Schur complement $S_{\Gamma}^{(j)}$ on the $j$-th subdomain $\Omega_j$
\begin{equation*}
	S_{\Gamma}^{(j)} = K_{\Gamma \Gamma}^{(j)} - K_{I \Gamma}^{(j) T} K_{II}^{(j) -1}  	 K_{I\Gamma}^{(j)}.
\end{equation*}
By defining the block-diagonal matrix
$
S_\Gamma = diag \left[ S_\Gamma^{(1)}, \dots, S_\Gamma^{(N)} \right]
$, 
$
\widehat{S}_\Gamma = R_\Gamma^T S_\Gamma R_\Gamma^T,
$
and the quantities
$
g = \left[ g_I, \	\widehat{g}_\Gamma \right]^T,
$
$
\widehat{g}_\Gamma = R_\Gamma^T g_\Gamma 
$,
where $R_\Gamma$ is the direct sum of local restriction operators $R_\Gamma^{(j)}$ (which  returns the local interface components), $S_\Gamma$ the unassembled global Schur complement system and $g$ the right-hand side of the linear system, the resulting system which we need to solve is
\begin{equation}\label{schursys}
	\widehat{S}_\Gamma u_\Gamma = \widehat{f}_\Gamma, 
	\qquad
	\widehat{f}_\Gamma = \widehat{g}_\Gamma - K_{I \Gamma}^{ T} K_{II}^{ -1}  g_I.
\end{equation}
Once this problem is solved, it is possible to retrieve the solution on the internal degrees of freedom (dofs) by using $\uG$
\begin{equation*}
	u_I = K_{II}^{-1} \left( f_I  - K_{I \Gamma} u_\Gamma \right).
\end{equation*}
In the case of this application, the Schur complement system Eq. \ref{schursys} is non-symmetric, thus it is necessary to apply a solver for non-symmetric problems, such as the GMRES iterative algorithm. 
For any $\uG \in \widetilde{W}_\Gamma$, we define the harmonic extension to the interior of subdomains $\uag$ as
\begin{equation*}
	\uag = 
	\begin{bmatrix}
		-K_{II}^{-1} \widetilde{K}_{I \Gamma} \uG \\
		\uG 
	\end{bmatrix}
	\in \widetilde{W},
\end{equation*}
and analogously, it is possible to define its counterpart $\uag \in \widehat{W}$ for $\uG \in \widehat{W}_\Gamma$. 
%
%
%We can define the Jacobian Bidomain local discrete harmonic extension operators as follows:
%\begin{equation*}
%	\mathcal{H}_j : W_\Gamma^{(j)} \longrightarrow W^{(j)}, \qquad \mathcal{H}_j w_\Gamma^{(j)} = 
%	\begin{dcases*}
%		- K_{II}^{(j) -1}  	 K_{I\Gamma}^{(j)}w_\Gamma^{(j)}		\qquad		\text{on } W^{(j)}_I	\\
%		w_\Gamma^{(j)}													\qquad	\qquad \qquad \qquad		\text{on } W^{(j)}_\Gamma
%	\end{dcases*}.
%\end{equation*}
%It is appropriate to notice that the local discrete harmonic extension of a constant vector is the vector itself. 
%
%
%From now on, we will use the component-wise notation 
%\begin{equation*}
%	\mathcal{H}_j u_\Gamma^{(j)} = \left( \mathcal{H}_j^i u_\Gamma^{(j)}, \mathcal{H}_j^e u_\Gamma^{(j)}, \mathcal{H}_j^w u_\Gamma^{(j)} \right),
%\end{equation*}
%where the superscripts $i, e$ denote the usual intra- and extracellular component, while $w$ denotes the gating counterpart.\\

We define the following bilinear forms for vectors in $\widehat{W}_\Gamma$ and $\widetilde{W}_\Gamma$ 
\begin{equation}\label{def32}
	\langle \uG, \vG \rangle_{B_\Gamma} = \vag^T B \uag, \quad 		\langle \uG, \vG \rangle_{Z_\Gamma} = \vag^T Z \uag,  \qquad \forall \uG, \vG \in \widehat{W}_\Gamma 
\end{equation}
\begin{equation} \label{def33}
	\langle \uG, \vG \rangle_{\widetilde{B}_\Gamma} = \vag^T \widetilde{B} \uag, \quad 		\langle \uG, \vG \rangle_{\widetilde{Z}_\Gamma} = \vag^T \widetilde{Z} \uag,  \qquad \forall \uG, \vG \in \widetilde{W}_\Gamma 
\end{equation}

We observe, as needed for further calculations, that \cite[Lemma 7.2]{tu2008balancing} follows from these definitions.

Following \cite{tu2008balancing}, it is useful to define $B_\Gamma$ and $\widetilde{B}_\Gamma$ norms:
\begin{equation*}
	\normBgamma{\uG}^2 = \langle \uG, \uG \rangle_{B_\Gamma} \	\text{ for }  \uG \in \widehat{W}_\Gamma 
	%\end{equation*}
	\qquad \text{and } \qquad
	%\begin{equation*}
	\normBtildeGamma{\uG}^2 = \langle \uG, \uG \rangle_{\widetilde{B}_\Gamma} 	\	\text{ for } \uG \in \widetilde{W}_\Gamma	.
\end{equation*}

% ----- Preconditioning
\subsection{Preconditioning}
As already mentioned, with the solution strategy described previously, the Schur complement matrix $\widehat{S}_\Gamma$ of the Jacobian Bidomain system (\ref{Jacobiancoupledsystem}) is non-symmetric but positive semidefinite: therefore it is necessary to apply a solver for non-symmetric systems, such as the Generalized Minimal Residual method (GMRES) \cite{saad1986gmres}.

Additionally, in order to enable fast convergence, preconditioning occurs. We hereby present theoretical results related to the Balancing Domain Decomposition with Constraints (BDDC) preconditioning algorithm \cite{dohrmann2016bddc}. 

When working with these methods, an interface averaging is needed: the standard scaling (\emph{$\rho$-scaling}) has weights built from the values of the elliptic coefficients in each substructure. On the contrary, the \emph{stiffness-scaling} takes its weights from the diagonal elements of both local and global stiffness matrix, while the more recent \emph{deluxe-scaling} (see \cite{dohrmann2016bddc, da2014isogeometric}) is based on the solution of local problems built from local Schur complements associated with the dual unknowns. 

As in our previous work \cite{huynh2021parallel}, we provide here a convergence rate estimate that holds both with the classic $\rho$-scaling and with the deluxe-scaling.

\subsubsection{Restriction operators and scaling. }
Before going into details of the proposed preconditioners, we define the restriction operators 
\begin{eqnarray}
	R_\Delta^{(j)}: W_\Delta \rightarrow W_\Delta^{(j)},  		&\quad R_{\Gamma \Delta}: W_\Gamma \rightarrow W_\Delta, 	 \nonumber \\
	R_\Pi^{(j)}: \widehat{W}_\Pi \rightarrow W_\Pi^{(j)},  	 	&\quad R_{\Gamma \Pi}: W_\Gamma \rightarrow \widehat{W}_\Pi, 	\nonumber
\end{eqnarray}
and the direct sums $R_\Delta = \oplus R_\Delta^{(j)}$, $R_\Pi = \oplus R_\Pi^{(j)} $ and $\widetilde{R}_\Gamma = R_{\Gamma \Pi} \oplus R_{\Gamma \Delta}$, which maps $W_\Gamma $ into $\widetilde{W}_\Gamma$. 
We also need a proper scaling of the dual variables.

The $\rho$-scaling, originally proposed for Neumann-Neumann methods, can be defined  for the coupled Bidomain model at each node $x \in \Gammaj$ as 
\begin{equation}\label{pseudoinv}
	\diej (x) = \dfrac{\sigma_M^{{i,e}^{(j)}}}{\sumNx{\sigma_M^{{i,e}^{(k)}}}}, 		\qquad 			\sigma_M^{{i,e}^{(j)}} = \max_{\bullet = \left\{ l, t, n \right\}} \sigma^{{i,e}^{(j)}}_\bullet,
	%\end{equation}
	\qquad
	%\begin{equation}
	\dwj (x) = \dfrac{1}{|\mathcal{N}_x |},
\end{equation}
where $\mN_x$ is the set of indices of all subdomains with $x$ in the closure of the subdomain. We note that $\mN_x$  induces the definition of an equivalence relation that classifies the interface degrees of freedom into faces, edges and vertices equivalence classes. 

The deluxe scaling \cite{dohrmann2016bddc, da2014isogeometric} computes the average $\bar{w} = \Ed w$ for each face $\mathcal{F}$ or edge $\mathcal{E}$ equivalence class as follows:
%We recall that the class of faces  contains all the nodes shared by two subdomains.  
suppose that $\mathcal{F}$ is shared by subdomains $\Omega_j$ and $\Omega_k$. Denote by $\SjF$ and $\SkF$ be the principal minors obtained from $\Sj_\Gamma$ and $\Sk_\Gamma$ by extracting all rows and columns related to the degrees of freedom of the face $\mathcal{F}$.  
Let $\uFj = \RF u_j$ be the restriction of $u_j$ to the face $\mathcal{F}$ through the restriction operator $\RF$. Then, the deluxe average across $\mathcal{F}$ can be defined as 
\begin{eqnarray}\nonumber
	\umeanF = \left( \SjF + \SkF \right)^{-1} \left( \SjF \uFj + \SkF \uFk \right).
\end{eqnarray}
The action of $ ( \SjF + \SkF )^{-1} $ can be computed by solving a Dirichlet problem over the two subdomains involved, with zero value on the right-hand side entries that correspond with the interior degrees of freedom. 

It is possible to extend this definition when considering the deluxe average across an edge $\mathcal{E}$. 
Suppose for simplicity that $\mathcal{E}$ is shared by only three subdomains with indices $j_1$, $j_2$ and $j_3$; the extension to more than three subdomains is straightforward. 
Denote by $\uEj = \RE u_j$ be the restriction of $u_j$ to the edge $\mathcal{E}$ through the restriction operator $\RE$ and define $\SjallE = \SjunoE + \SjdueE + \SjtreE$; the deluxe average across an edge $\mathcal{E}$ is given by
\begin{eqnarray}\nonumber
	\umeanE = \left( \SjallE  \right)^{-1} \left( \SjunoE \uEjuno + \SjdueE \uEjdue + \SjtreE \uEjtre \right).
\end{eqnarray}

The relevant equivalence classes, involving the substructure $\Omega_j$, will contribute to the values of $\bar{u}$. These contributions will belong to $\widehat{W}_\Gamma$, after being extended by zero to $\Gamma \backslash \mathcal{F}$ or $\Gamma \backslash \mathcal{E}$; the sum of all contributions will result in $R^T_\ast \bar{u}_\ast$. We then add the contributions from the different equivalence classes to obtain
\begin{eqnarray}\nonumber
	\bar{u} = \Ed u = u_\Pi + \sum_{\ast = \{ \mathcal{F}, \mathcal{E} \} } R^T_\ast \bar{u}_\ast,
\end{eqnarray}
where $\Ed$ is a projection. We define its complementary projection by
\begin{eqnarray}\label{PdDeluxe}
	\Pd u := (I - \Ed) u = u_\Delta - \sum_{\ast = \{ \mathcal{F}, \mathcal{E} \} } R^T_\ast \bar{u}_\ast.
\end{eqnarray}
We define the scaling matrix for each subdomain $\Omega_j$
\begin{equation}\label{deluxescaling}
	D^{(j)} =
	\begin{bmatrix}
		D^{(j)}_{\ast_{k_1}} 	&		&	\\
		& \ddots	&	\\
		&		&D^{(j)}_{\ast_{k_j}}
	\end{bmatrix},
	\qquad
	\ast = \left\{ \mathcal{F}, \mathcal{E} \right\}
\end{equation}
being $k_1, \dots, k_j \in \varXi_j^\ast$, a set containing the indices of the subdomains that share the face $\mathcal{F}$ or the edge $\mathcal{E}$ and where the diagonal blocks are given by $D^{(j)}_\mathcal{F} = ( \SjF + \SkF )^{-1} \SjF $ or $D^{(j)}_\mathcal{E}=  (\SjunoE + \SjdueE + \SjtreE)^{-1} \SjunoE$.

Lastly, we define the scaled local restriction operators
\begin{equation*}
	R_{D, \Gamma}^{(j)} = D^{(j)} R_\Gamma^{(j)}, 	\qquad 	\qquad R_{D, \Delta}^{(j)} = R_{\Gamma \Delta}^{(j)} R_{D, \Gamma}^{(j)} ,
\end{equation*}
$R_{D, \Delta}$ as direct sum of $R_{D, \Delta}^{(j)}$ and the global scaled operator 	$\widetilde{R}_{D, \Gamma} = R_{\Gamma \Pi} \oplus R_{D, \Delta} R_{\Gamma \Delta}$. \\

\subsubsection{BDDC preconditioner}
We recall that the Jacobian linear system (\ref{Jacobiancoupledsystem}) for the original non-linear reaction-diffusion problem has been reduced to the non-symmetric Schur complement system $\widehat{S}_\Gamma u_\Gamma = \widehat{f}_\Gamma$ for the subdomain interface variables. 
The interface problem is then solved with a preconditioned GMRES iteration. 

BDDC algorithms were initially proposed by \cite{dohrmann2016bddc} for the solution of symmetric, positive definite problems, but its formulation can be equally applied to non-symmetric problems, as done in \cite{tu2008balancing} for advection-diffusion equations. 

If we partition the degrees of freedom of the interface $\Gamma$ into those internal ($I$), those dual ($\Delta$) and those primal ($\Pi$), the matrix $\mathcal{K}^{(j)}$ from the problem $\mathcal{K} u = f$ can be written as
\begin{equation*}
	\mathcal{K}^{(j)} = 
	\begin{bmatrix}
		K_{II}^{(j)} 					& K_{I\Gamma}^{(j)} \\
		K_{I \Gamma}^{(j) T}  	& K_{\Gamma \Gamma}^{(j)}
	\end{bmatrix}
	=
	\begin{bmatrix}
		K_{II}^{(j)} 					& K_{I \Delta}^{(j)} 			   & K_{I \Pi}^{(j)} \\
		K_{I \Delta}^{(j) T}  		& K_{\Delta \Delta}^{(j)} 		& K_{\Delta \Pi}^{(j)}	\\
		K_{I \Pi}^{(j) T}			  & K_{\Delta \Pi}^{(j) T}			& K_{\Pi \Pi}^{(j)}
	\end{bmatrix}.
\end{equation*} 
We define the BDDC preconditioner using the restriction operators as 
\begin{equation}\label{BDDCoperators}
	M^{-1}_\text{BDDC} = \widetilde{R}_{D, \Gamma}^T \widetilde{S}_\Gamma^{-1}  \widetilde{R}_{D, \Gamma}, 	\qquad 		\widetilde{S}_\Gamma = \widetilde{R}_\Gamma S_\Gamma \widetilde{R}_\Gamma^T,
\end{equation}
where the action of the inverse of $\widehat{S}_\Gamma$ can be evaluated with a block-Cholesky elimination procedure
\begin{equation*}
	\widetilde{S}_\Gamma^{-1} = \widetilde{R}_{\Gamma \Delta }^T \left( \sum_{j=1}^{N} 
	\begin{bmatrix}
		0 		&R_\Delta^{(j) T}
	\end{bmatrix}
	\begin{bmatrix}
		K_{II}^{(j)} 					& K_{I \Delta}^{(j)} 			  \\
		K_{I \Delta}^{(j) T}  		& K_{\Delta \Delta}^{(j)}
	\end{bmatrix}^{-1}
	\begin{bmatrix}
		0 	\\		R_\Delta^{(j)}
	\end{bmatrix}
	\right) \widetilde{R}_{\Gamma \Delta } + \varPhi S_{\Pi \Pi}^{-1} \varPhi ,
\end{equation*}
where the first term is the sum of local solvers on each substructure $\Omega_j$, while the latter is a coarse solver for the primal variables where
\begin{equation*}
	\varPhi = R_{\Gamma \Pi}^T - R_{\Gamma \Delta}^T \sum_{j=1}^N 
	\begin{bmatrix}
		0 		&R_\Delta^{(j) T}
	\end{bmatrix}
	\begin{bmatrix}
		K_{II}^{(j)} 					& K_{I \Delta}^{(j)} 			  \\
		K_{I \Delta}^{(j) T}  		& K_{\Delta \Delta}^{(j)}
	\end{bmatrix}^{-1}
	\begin{bmatrix}
		K_{I \Pi}^{(j)} 	\\		R_{\Delta \Pi}^{(j)}
	\end{bmatrix}
	R_\Pi^{(j)} ,
\end{equation*}
\begin{equation*}
	S_{\Pi \Pi } = \sum_{j=1}^N  R_\Pi^{(j) T} 
	\left( K_{\Pi \Pi}^{(j)} - 
	\begin{bmatrix}
		K_{I \Pi}^{(j) T}		&K_{\Delta \Pi}^{(j) T}   
	\end{bmatrix}
	\begin{bmatrix}
		K_{II}^{(j)} 					& K_{I \Delta}^{(j)} 			  \\
		K_{I \Delta}^{(j) T} 		& K_{\Delta \Delta}^{(j)}
	\end{bmatrix}^{-1}
	\begin{bmatrix}
		K_{I \Pi}^{(j)} 	\\		R_{\Delta \Pi}^{(j)}
	\end{bmatrix}
	\right) R_\Pi^{(j)},
\end{equation*}
are the matrix which maps the primal degrees of freedom to the interface variables and the primal problem respectively. 

%% --------------------------

\section{Convergence rate estimate} \label{convergence coupled}
The key-point of the proof for the convergence rate estimate relies in this Lemma, which can be also proved for the $\rho$-scaling. 
%A trace of the proof can be found in \cite{huynh2021parallel}, by paying attention to the treatment of the additional gating term. 

\begin{lemma}\label{projectionlemma}
	Assume that the primal space is spanned by the vertex nodal finite element functions and the edge cutoff functions. Let the projection operator be scaled by either the standard $\rho$-scaling or the deluxe-scaling. Then 
	\begin{equation*}
		\normBtildeGamma{E_D u}^2 \lesssim  \left[ \max_{\substack{ k = 1, \dots, N \\ \star = i,e}} \dfrac{\tau \sigma_M^{\star (k)} + H^2 \left( \chi C_m + \tau K_{M,I} \right)}{\tau \sigma_m^{\star (k)}}  + \dfrac{1 - \tau K_{M,R} }{1 - \tau K_{m, R}} \right]  \left( 1 + \log \dfrac{H}{h} \right)^n \normBtildeGamma{u}^2,
	\end{equation*}
	holds $\forall u \in \widetilde{W}_\Gamma$, with $n = 2$ in case the $\rho$-scaling is applied, $n=3$ in case of the deluxe-scaling.
\end{lemma}

\begin{proof}
	We report here a sketch of the proof for the deluxe scaling. 
	Two crucial points are the equivalence between the norms $\normBtildeGamma{\cdot}$ and $| \cdot |_{\widetilde{S}_\Gamma}$, for any $u \in \widetilde{W}_\Gamma$ (from \cite[Lemma $7.2$]{tu2008balancing}) and, as usual in the substructuring framework, estimating the local contributions for the complementary projection $\Pd$,
	\begin{equation*}
		\seminormSj{R_{\pO_j} \Pd u}^2 \leq | \varXi_j^\ast | \sum_{\substack{\ast = \{ \mathcal{F}, \mathcal{E} \}, \ \ast \in \varXi_j^\ast}} \seminormSj{ R_\ast^T \left( u^{i,e,w}_{j, \ast} - \bar{u}^{i,e,w}_{\ast}  \right)}^2   , 
	\end{equation*}
	where $\varXi_j^\ast$ is the index set containing the indices of the subdomains that share the face $\mathcal{F}$ or the edge $\mathcal{E}$. We will denote by 
	%	$
	\begin{equation} \label{eq: def mean value}
		\bar{u}^{i,e,w}_{j, \mathcal{G}} = \left(  \bar{u}^i_{j, \mathcal{G}} , \bar{u}^e_{j, \mathcal{G}} , 0 \right)
	\end{equation}
	%	$
	the vector containing the mean value of the intra- and extra-cellular potentials over $\mathcal{G} = \{ \mathcal{F}, \mathcal{E} \}$ on the subdomain $j$ and null value corresponding to the gating component. 
	%	Let us distinguish between face and edge contributions.
	%	\newline
	Regarding the face contributions, by simple algebra, it is easy to bound the quantity $u_{j, \mathcal{F}} - \umeanF$ with
	\begin{equation*}
		2 \seminormSjF{u_{j, \mathcal{F}} - \umeaniewFj}^2 + 2 \seminormSkF{u_{k, \mathcal{F}} - \umeaniewFk}^2 + \seminormSjF{( \SjF + \SkF )^{-1} \SkF ( \umeaniewFj - \umeaniewFk )}^2,
	\end{equation*}
	$\forall u_{j, \mathcal{F}} \in \widetilde{W}_\Gamma$, by using the inequalities that arise from the generalized eigenvalue problem $\SjF \phi = \lambda \SkF \phi$ and by observing that all eigenvalues are strictly positive \cite{da2014isogeometric} 
	\begin{align*}
		\SkF ( \SjF + \SkF )^{-1}  \SjF ( \SjF + \SkF )^{-1} \SkF  &\leq \SjF \\
		\SkF ( \SjF + \SkF )^{-1}  \SjF ( \SjF + \SkF )^{-1} \SkF  &\leq \SkF.
	\end{align*}
	It is sufficient to estimate $\seminormSjF{u_{j, \mathcal{F}} - \umeaniewFj}^2$ and $\seminormSjF{ ( \SjF + \SkF )^{-1} \SkF ( \umeaniewFj - \umeaniewFk )}^2$; we highlight that, in case also the subdomain faces averages are included in the primal space, the latter is zero. 
	The first term can be bounded by applying the ellipticity Lemma \ref{ellipcoupbound}, the Poincarè-Friedrichs inequality and the Trace theorem, by
	\begin{equation*}
		\left[ \tau \sigma_M^{i,e} + H^2 \left( \chi C_m + \tau K_{M,I} \right) \right] \left( 1 + \log \dfrac{H}{h} \right)^2 \seminormHone{ \harmextie{j}{u_j} }{\Omega_j}^2 + \left( 1 - K_{M,R} \right)  \norm{  u_j^w }{\Omega_j}^2,
	\end{equation*}
	where $\harmextie{j}{u_j}$ is the discrete Laplacian extension operator. 
	%	thanks to \cite[Lemma $7.2$]{tu2008balancing} and to the definition of $\bar{u}_{j, \mathcal{F}}^{i,e,w}$.
	Regarding the second term $\seminormSjF{ ( \SjF + \SkF )^{-1} \SkF ( \umeaniewFj - \umeaniewFk )}^2$, let $\mathcal{E} \subset \partial \mathcal{F}$ be a primal edge, such that $\uiewmeanEj = \uiewmeanEk$. Then, it is straightforward to see that 
	\begin{equation*}
		\seminormSjF{\left( \SjF + \SkF \right)^{-1} \SkF \left( \umeaniewFj - \umeaniewFk \right)}^2 \leq 2 \ \seminormSjF{  \uiewmeanEj - \umeaniewFj  }^2 + 2 \ \seminormSkF{ \uiewmeanEk - \umeaniewFk }.
	\end{equation*}
	Combining Lemma $7.2$ from \cite{tu2008balancing}, the result of ellipticity Lemma (\ref{ellipcoupbound}), the Poincarè-Friedrichs inequality and the Trace theorem, we get
	\begin{align*}
		\seminormSjF{  \uiewmeanEj - &\umeaniewFj  }^2 \lesssim  \sum_{\star = i,e} \left[ \tau \sigma^{\star}_M + H^2 \left( \chi C_m + \tau K_M \right) \right] \seminormHone{\overline{\left( u_j^{i,e} - \umeanieFj \right) }_{j, \mathcal{E}} }{\Omega_j}^2  \\
		&\leq C \left( 1 + \log \dfrac{H}{h} \right)^3 \sum_{\star = i,e} \left[ \tau \sigma^{\star}_M + H^2 \left( \chi C_m + \tau K_{M,I} \right) \right]  \seminormHone{\harmextie{j}{u_j}}{\Omega_j}^2 ,
	\end{align*}
	as the mean value of the gating component vanishes by construction Eq. \ref{eq: def mean value} and the norm related to the potentials can be obtained with results from Ref. \cite[Chapter 4, Lemmas 4.26 and 4.30]{toselli2006domain}. 
	To conclude, the face contribution for the bound of $\seminormSj{P_D u}^2$ is given by
	\begin{equation*}
		\sum_{\mathcal{F} \in \varXi_j^{\mathcal{F}}} \left[ \max_{\star = i,e}  \dfrac{ \tau \sigma_M^{\star} + H^2 \left( \chi C_m + \tau K_{M,I}  \right) }{\tau \sigma_m^{\star}} +  \dfrac{1 - \tau K_{M,R} }{1 - \tau K_{m, R} } \right] \left( 1 + \log \dfrac{H}{h} \right)^3 \normBtildeGammaj{u_j}^2.
	\end{equation*}	
	In a similar manner, the edge contribution can be obtained with
	\begin{equation*}
		\uEjuno - \umeanE = \left( \SjallE \right)^{-1} \left[ \left( \SjdueE + \SjtreE \right) \uEjuno - \SjdueE \uEjdue - \SjtreE \uEjtre \right],
	\end{equation*}
	where for simplicity, by supposing that an edge $\mathcal{E}$ is shared only by three substructures, each with indexes $j_1$, $j_2$ and $j_3$ (the extension to the case of more subdomains is then similar), we define 
	$
	\SjallE := \SjunoE + \SjdueE + \SjtreE
	$
	and the average operator as
	$
	\umeanE := ( \SjallE )^{-1} ( \SjunoE \uEjuno + \SjdueE \uEjdue + \SjtreE \uEjtre ).
	$
	\newline
	Proceeding in the same fashion as for the face contribution, it follows
	\begin{equation*}
		\seminormSjuno{\RE^T \left(\uEjuno - \umeanE \right)}^2 \leq 3 \uEjuno^T \SjunoE \uEjuno \ + 3 \uEjdue^T \SjdueE \uEjdue \ + \uEjtre^T  \SjtreE \uEjtre 
	\end{equation*} 
	By adding and subtracting $\uiewmeanEjuno$ (which assume the same value over the three subdomain, as we have included the edge averages into the primal space) we can get the counterpart estimate for the egdes:
	\begin{align*}
		\uEjuno^T \SjunoE \uEjuno &\leq \left[ \tau \sigma_M^{i,e} + H^2 \left( \chi C_m + \tau K_{M,I}  \right) \right] \left( 1 + \log \dfrac{H}{h} \right) \seminormHone{\harmextie{j_1}{\uEjuno}}{\Omega_j}^2\\
		&\qquad +  \left( 1 - \tau K_{M,R} \right) \norm{\uEjuno^w  }{\Omega_j}^2 
	\end{align*}
	In conclusion, the edge contribution for the estimate of $\seminormSj{P_D u}^2$ is given by
	\begin{equation*}
		\sum_{\mathcal{E} \in \varXi_j^\mathcal{E}}  \left[ \max_{\star = i,e} \dfrac{ \tau \sigma_M^{i,e} + H^2 \left( \chi C_m + \tau K_{M,I}  \right) }{\tau \sigma_m^\star} + \dfrac{1 - \tau K_{M, R}}{1 - \tau K_{m, R}} \right] \left( 1 + \log \dfrac{H}{h} \right) \normBtildeGammaj{u_{j}}^2.
	\end{equation*}
	where the index $j$ collects all contributions from the subdomains that share edge $\mathcal{E}$. 
	\qed
\end{proof}

The convergence rate of the preconditioned GMRES iteration can be obtained using the result in \cite{eisenstat1983variational} and following the proof techniques proposed in \cite{tu2008balancing}.

\begin{theorem}\label{coupledconvergence}
	Let $H$ be the subdomain size and let the mesh size $h$ be small enough. \\ Assume, for $u \in \widehat{W}_\Gamma$, that there exists two positive constants $c$ and $C$ such that
	\begin{equation*}
		%		\begin{aligned}
		c \langle u, u \rangle_{B_\Gamma} \leq \langle u, Tu \rangle_{B_\Gamma}	,
		\qquad
		\langle Tu, Tu \rangle_{B_\Gamma} \leq C \langle u, u \rangle_{B_\Gamma}
		%		\end{aligned}
	\end{equation*}
	hold, with $c = \dfrac{c_0}{K^2} $ and $C = \Phi^{\star, k} (H,h) \ K^2 $, where
	\begin{equation*}
		\begin{aligned}
			c_0 &= 1 - K^4 \ \dfrac{H^2}{h} \max_{\substack{ k = 1,\dots,N \\ \star = i,e}} \dfrac{\left[ \sigma_M^{\star (k)} \right]^{\frac{1}{2}} }{\sqrt{\tau} \sigma_m^{\star (k)}} \Phi^{\star, k} (H,h) \left[ \Phi^{\star, k} (H,h)  - 1 \right]^{\frac{1}{2}} , \\
			\Phi^{\star, k} (H,h) &= \left[ \max_{\substack{ k = 1, \dots, N \\ \star = i,e}} \dfrac{\tau \sigma_M^{\star (k)} + H^2 \left( \chi C_m + \tau K_{M,I} \right)}{\tau \sigma_m^{\star (k)}}  + \dfrac{1 - \tau K_{M,R} }{1 - \tau K_{m, R}} \right]  \left( 1 + \log \dfrac{H}{h} \right)^n  ,\\ 
			K^2 &= \dfrac{1}{4} \dfrac{\tau^2 \ | C_{I_w} - C_{R_v}|^2}{\left( \chi C_m + \tau K_{m, I} \right) \left( 1 - \tau K_{m,R} \right)},
		\end{aligned}
	\end{equation*}
	where $n=2,3$ depends on a standard $\rho$-scaling or deluxe scaling procedure and $T$ is the preconditioned operator
	%	\begin{equation*}
	$
	T = M^{-1}_{BDDC} \widetilde{S}_\Gamma =  \widetilde{R}_{D, \Gamma}^T \widetilde{S}_\Gamma^{-1}  \widetilde{R}_{D, \Gamma} \widetilde{R}_\Gamma S_\Gamma \widetilde{R}_\Gamma^T .
	$
	%	\end{equation*}
	Then
	\begin{equation*}
		\dfrac{\normBgamma{r_m} }{\normBgamma{r_0}} \leq \left( 1 - \dfrac{c^2}{C} \right)^\frac{m}{2} ,
	\end{equation*}
	where $r_m$ is the residual at the $m$-th iteration.
\end{theorem}

% ---------------
\subsection{Proof of the upper bound} 
For the proof of the upper bound $\langle Tu, Tu \rangle_{B_\Gamma} \leq C \langle u, u \rangle_{B_\Gamma}$, we need the following results.

\begin{lemma} \label{lemma4.3}		% Lemma 4.3
	There exists a constant $C_1 >0$ such that $\forall u_j, v_j \in W^{(j)}$ with $j=1,\dots,N$,
	\begin{equation*}
		| z^{(j)} \left( u_j, v_j \right) | \leq C_1 \ K | u_j |_{B^{(j)}} | v_j |_{B^{(j)}},
		%	\end{equation*}
		\qquad
		%	\begin{equation*}
		| a^{(j)} \left( u_j, v_j \right) | \leq C_1 \ K | u_j |_{B^{(j)}} | v_j |_{B^{(j)}},
	\end{equation*}
	where
	\begin{equation*}
		K^2 = \dfrac{1}{4} \dfrac{\tau^2 \ | C_{I_w} - C_{R_v}|^2}{\left( \chi C_m + \tau K_{m, I} \right) \left( 1 - \tau K_{m,R} \right)}.
	\end{equation*}
\end{lemma}

\begin{proof}
	Thanks to the definition of the $B^{(j)}$-norm and by using the ellipticity Lemma \ref{ellipcoupbound}, we can bound from below the norm 
	\begin{align*}
		| u_j |_{B^{(j)}}^2 
		%			&= b^{(j)} \left( u_j, u_j \right)	\\
		%			&\geq 2 \left[ \left( \chi C_m + \tau K_{m,I} \right) \norm{u^i_j - u^e_j}{\Omega_j}^2 
		%			+ \left( 1 - \tau K_{m,R} \right) \norm{u^w_j}{\Omega_j}^2  
		%			+ \tau \sigma_m^i \seminormHone{u^i_j}{\Omega_j}^2 + \tau \sigma_m^e \seminormHone{u^e_j}{\Omega_j}^2 \right] \\
		&\geq 2 \left[ \left( \chi C_m + \tau K_{m,I} \right) \norm{u^i_j - u^e_j}{\Omega_j}^2	
		+ \ \left( 1 - \tau K_{m,R} \right) \norm{u^w_j}{\Omega_j}^2 \right].
	\end{align*}
	Therefore, $\forall u_j, v_j \in W^{(j)}$
	\begin{align*}
		| u_j |_{B^{(j)}}^2 | v_j |^2_{B^{(j)}} &\geq 4 \left( \chi C_m + \tau K_{m,I} \right) \left( 1 - \tau K_{m,R} \right) \times \\ & \qquad \qquad \left[ \norm{u^i_j - u^e_j}{\Omega_j}^2 \norm{v^w_j}{\Omega_j}^2 + \norm{u^w_j}{\Omega_j}^2 \norm{v^i_j - v^e_j}{\Omega_j}^2 \right],
	\end{align*}
	from which
	\begin{equation*}
		\begin{aligned}
			| u_j |_{B^{(j)}} | v_j |_{B^{(j)}} &\geq 2 \left( \chi C_m + \tau K_{m,I} \right)^\frac{1}{2} \left( 1 - \tau K_{m,R} \right)^\frac{1}{2} \times \\ &\qquad \qquad  \sqrt{\norm{u^i_j - u^e_j}{\Omega_j}^2 \norm{v^w_j}{\Omega_j}^2 + \norm{u^w_j}{\Omega_j}^2 \norm{v^i_j - v^e_j}{\Omega_j}^2 }.
		\end{aligned}
	\end{equation*}
	We can estimate the bound for the skew-symmetric bilinear form
	\begin{equation*}
		\begin{aligned}
			| z^{(j)} \left( u_j, v_j \right) | &\leq \tau | C_{I_w} - C_{R_v} | \times \\ &\qquad \qquad    | \norm{u^i_j - u^e_j}{\Omega_j} \norm{v^w_j}{\Omega_j} + \norm{u^w_j}{\Omega_j} \norm{v^i_j - v^e_j}{\Omega_j} | \\
			&\leq \dfrac{1}{2} \dfrac{\tau | C_{I_w} - C_{R_v} |}{\left( \chi C_m + \tau K_{m,I} \right)^\frac{1}{2} \left( 1 - \tau K_{m,R} \right)^\frac{1}{2} } |  u_j |_{B^{(j)}} | v_j |_{B^{(j)}}.
		\end{aligned}
	\end{equation*}
	The bound for the bilinear form $a (\cdot, \cdot)$ follows easily from its decomposition
	\begin{equation*}
		a^{(j)} \left( u_j , v_j \right) = \dfrac{1}{2} b^{(j)} \left( u_j , v_j \right) + \dfrac{1}{2} z^{(j)} \left( u_j , v_j \right) 
	\end{equation*}
	and from the continuity of both symmetric and skew-symmetric forms.
	\qed
\end{proof}

% ----------- 

\begin{lemma}\label{lemma4.4}			% Lemma 4.3
	There exists a constant $C_2 > 0 $ such that $\forall u, v \in \widehat{W}$ 
	\begin{equation*}
		| z(u,v) | \leq C_2 \ K |u|_B \norm{v}{\Omega},
	\end{equation*}
	with $K$ defined in Lemma \ref{lemma4.3}.
\end{lemma}

% --------------------

\begin{lemma}\label{lemma7.3} 			% Lemma 7.3
	Let $C_3, C_4 > 0 $ be two positive constants, independent from $H$ and $h$, such that $\forall \uG, \vG \in \widetilde{W}_\Gamma$ 
	\begin{align*} 
		\text{(i) } \qquad | \langle \uG, \vG \rangle_{\widetilde{Z}_\Gamma} | \leq C_3 \ K \ \normBtildeGamma{\uG} \normBtildeGamma{\vG} 	\nonumber \\
		\text{(ii) } \qquad | \langle \uG, \vG \rangle_{\widetilde{S}_\Gamma} | \leq C_4 \ K \ \normBtildeGamma{\uG} \normBtildeGamma{\vG} 	\nonumber \\
	\end{align*}
	where $K$ is defined in Lemma \ref{lemma4.3}.
\end{lemma}

\begin{proof}
	\begin{itemize}
		\item[(i)] Regarding the first bound, it is straightforward to see that $\forall \uG, \vG \in \widetilde{W}_\Gamma$
		\begin{equation*}
			| \langle \uG, \vG \rangle_{\widetilde{Z}_\Gamma} | \leq K \normBtilde{\uag} \normBtilde{\vag}			% Lemma 4.3
			= K \normBtildeGamma{\uG} \normBtildeGamma{\vG}				,	% Def. 33
		\end{equation*}
		by applying Definition \ref{def33} and Lemma \ref{lemma4.3}.
		\item[(ii)] By using \cite[Lemma $7.2$]{tu2008balancing} and \ref{lemma4.3}, we obtain
		\begin{equation*}
			| \langle \uG, \vG \rangle_{\widetilde{S}_\Gamma} |  = | \langle \uag, v \rangle_{\widetilde{A}} | 			% Lemma 7.2
			= | \widetilde{a} \left( \uag, v \right) | 	
			\leq K \normBtildeGamma{\uG} \normBtildeGamma{\vG}.
		\end{equation*}
		\qed
	\end{itemize}
\end{proof}

We are ready to prove the upper bound of Theorem $\ref{coupledconvergence}$. \\

Given $\uG \in \widehat{W}_\Gamma$, by defining $\wG = \widetilde{S}_\Gamma^{-1}  \widetilde{R}_{D, \Gamma} \uG$ and using \cite[Lemmas $7.2$ and $7.9$]{tu2008balancing} and Lemmas \ref{projectionlemma} and \ref{lemma7.3}
\begin{align*}
	\langle T \uG, T \uG \rangle_{B_\Gamma} 	
	&= \langle \widetilde{R}_{D, \Gamma}^T \widetilde{S}_\Gamma^{-1}  \widetilde{R}_{D, \Gamma} \uG, \ \widetilde{R}_{D, \Gamma}^T \widetilde{S}_\Gamma^{-1}  \widetilde{R}_{D, \Gamma} \uG \rangle_{S_\Gamma}  \\
	&= \langle \widetilde{R}_{D, \Gamma}^T \wG , \ \widetilde{R}_{D, \Gamma}^T \wG \rangle_{S_\Gamma} 
	= \langle \widetilde{R}_\Gamma  \widetilde{R}_{D, \Gamma}^T \wG , \ \widetilde{R}_\Gamma  \widetilde{R}_{D, \Gamma}^T \wG \rangle_{\widetilde{S}_\Gamma } \\ 
	&= \langle \Ed \wG , \ \Ed \wG \rangle_{\widetilde{S}_\Gamma } 
	= | \Ed \wG |^2_{\widetilde{S}_\Gamma} 
	= \normBtildeGamma{\Ed \wG}^2 \\
	&\leq  \Phi^{\star, k} (H,h) \ \normBtildeGamma{\wG}^2  		% projection lemma
	=  \Phi^{\star, k} (H,h) \ \langle \uG, T \uG \rangle_{S_\Gamma} \\ 		% Lemma 7.9
	&\leq \Phi^{\star, k} (H,h) K \ \langle \uG, \uG \rangle_{B_\Gamma}^{1/2} \langle T \uG, T \uG \rangle_{B_\Gamma}^{1/2} ,
\end{align*}
from which we conclude
\begin{equation*}
	\langle T \uG, T \uG \rangle_{B_\Gamma} \leq  \Phi^{\star, k} (H,h) K^2 \ \langle \uG, \uG \rangle_{B_\Gamma},
\end{equation*}
where
\begin{equation} \label{def: K and Phi}
	\begin{aligned}
		\Phi^{\star, k} (H,h) &= \left[ \max_{\substack{ k = 1, \dots, N \\ \star = i,e}} \dfrac{\tau \sigma_M^{\star (k)} + H^2 \left( \chi C_m + \tau K_{M,I} \right)}{\tau \sigma_m^{\star (k)}}  + \dfrac{1 - \tau K_{M,R} }{1 - \tau K_{m, R}} \right]  \left( 1 + \log \dfrac{H}{h} \right)^n,  \\
		K^2 &= \dfrac{1}{4} \dfrac{\tau^2 \ | C_{I_w} - C_{R_v}|^2}{\left( \chi C_m + \tau K_{m, I} \right) \left( 1 - \tau K_{m,R} \right)},
	\end{aligned}
\end{equation}
and $n=2,3$ depends on which scaling we are considering.

% ---------------
\subsection{Proof of the lower bound} 
Conversely, for the proof of the lower bound, we need Lemma $7.9$ from \cite{tu2008balancing} and the following results.

\begin{lemma}\label{lemma7.4}
	There exists a constant $C_5 > 0$ such that $\forall \uG, \vG \in \widehat{W}_\Gamma$,
	\begin{equation*}
		| \langle \uG, \vG \rangle_{Z_\Gamma} |\leq C_5 \ K | \uG |_{B_\Gamma} \norm{\vag}{\Omega},
	\end{equation*}
	with $K$ defined in Lemma \ref{lemma4.3}.
\end{lemma}

%% Lemma 7.10

\begin{lemma}\label{lemma7.10}
	Given $\wG = \widetilde{S}_\Gamma^{-1}  \widetilde{R}_{D, \Gamma} S_\Gamma \uG$ for $\uG \in \widehat{W}_\Gamma$, there exists a constant $C_6 > 0$ such that
	\begin{equation*}
		\normBtildeGamma{\wG}^2 \leq C_6 \ K^2  \Phi^{\star, k} (H,h) \ \normBgamma{\uG}^2,
	\end{equation*}
	with $\Phi$ and $K$ defined in Equations (\ref{def: K and Phi}).
\end{lemma}

\begin{proof}
	We make use of \cite[Lemma $7.2$]{tu2008balancing} and Lemma \ref{projectionlemma} to obtain
	\begin{align*}
		\langle T \uG, T \uG \rangle_{B_\Gamma}
		&= \langle \widetilde{R}_{D, \Gamma}^T \widetilde{S}_\Gamma^{-1}  \widetilde{R}_{D, \Gamma} S_\Gamma \uG, \ \widetilde{R}_{D, \Gamma}^T \widetilde{S}_\Gamma^{-1}  \widetilde{R}_{D, \Gamma} S_\Gamma \uG \rangle_{S_\Gamma}  \\
		&= \langle \widetilde{R}_{D, \Gamma}^T \wG , \ \widetilde{R}_{D, \Gamma}^T \wG \rangle_{S_\Gamma} 
		= | \Ed \wG |^2_{\widetilde{S}_\Gamma} 
		= | \Ed \wG |^2_{\widetilde{B}_\Gamma} \\
		&\leq C_6  \Phi^{\star, k} (H,h) \normBtildeGamma{\wG}^2 		% Lemma: projection 
		%			= C_6  \Phi^{\star, k} (H,h) \ \langle \uG, T \uG \rangle_{S_\Gamma}		\\		% Lemma 7.9
		\leq C_6 \ K \Phi^{\star, k} (H,h) \ \langle \uG, \uG \rangle_{B_\Gamma}^{1/2} \langle T \uG, T \uG \rangle_{B_\Gamma}^{1/2} ,			% Lemma 7.3
	\end{align*}
	where in the last inequality we combined Lemma $7.9$ from \cite{tu2008balancing} and Lemma \ref{lemma7.3}. It follows
	\begin{equation*}
		\langle T \uG, T \uG \rangle_{B_\Gamma} \leq C_6 \  K^2 \left[ \Phi^{\star, k} \right]^2 (H,h) \ \langle \uG, \uG \rangle_{B_\Gamma}.
	\end{equation*}
	It is straightforward to conclude, by applying again \cite[Lemma $7.9$]{tu2008balancing} and Lemma \ref{lemma7.3}
	\begin{equation*}
		\normBtildeGamma{\wG}^2 = \langle \uG, T \uG \rangle_{S_\Gamma}
		\leq K \normBgamma{\uG} \normBgamma{T \uG} 	
		\leq C_6 \ K^2  \Phi^{\star, k} (H,h)  \normBgamma{\uG}^2.
	\end{equation*}
	\qed
\end{proof}

%% Lemma 7.11
\begin{lemma}\label{lemma7.11}
	Let $\wG = \widetilde{S}^{-1}_\Gamma \widetilde{R}_{D, \Gamma} S_\Gamma \uG$, for $\uG \in \widehat{W}_\Gamma$. Then, the following property holds
	\begin{equation*}
		\langle \wag, v \rangle_{\widetilde{B}} = \langle \widetilde{R} \uag, v \rangle_{\widetilde{B}}, 
	\end{equation*}
	for all $v \in \widetilde{R} ( \widehat{W} )$.
\end{lemma}

\begin{proof}
	Given $v \in  \widetilde{R} ( \widehat{W} )$, we denote $\vG \in  \widetilde{R}_\Gamma ( \widehat{W}_\Gamma )$ its continuous interface part. Then, given $\uG \in \widehat{W}_\Gamma$, by using Lemma $7.2$ from \cite{tu2008balancing} and the identity $\widetilde{R}_{\Gamma} \widetilde{R}_{D, \Gamma}^T \vG = \vG$, we get
	\begin{equation*}
		\begin{aligned}
			\langle \wag, v \rangle_{\widetilde{B}} &= \langle \wG, \vG \rangle_{\widetilde{S}_\Gamma} 	
			%			&= \vG^T \widetilde{S}_\Gamma \widetilde{S}_\Gamma^{-1} \widetilde{R}_{D, \Gamma} S_\Gamma \uG 	\\
			%			&= \vG^T  \widetilde{R}_{D, \Gamma} S_\Gamma \uG 		\\
			%			&= \vG^T  \widetilde{R}_{D, \Gamma} \widetilde{R}_{\Gamma} \widetilde{R}_{D, \Gamma}^T S_\Gamma \uG 		\\
			%			&= \vG^T  \widetilde{R}_{D, \Gamma} \widetilde{R}_{\Gamma} \widetilde{S}_\Gamma \widetilde{R}_\Gamma \uG 		\\
			= \langle \widetilde{R}_\Gamma \uG, \widetilde{R}_\Gamma \widetilde{R}_{D, \Gamma}^T \vG \rangle_{\widetilde{S}_\Gamma} \\
			&= \langle \widetilde{R}_\Gamma \uG, , \vG \rangle_{\widetilde{S}_\Gamma} 	
			= \langle \widetilde{R}_\Gamma \uag, , v \rangle_{\widetilde{B}} .					% Lemma 7.2
		\end{aligned}
	\end{equation*}
	\qed
\end{proof}

%% Lemma 7.12

\begin{lemma}\label{lemma7.12}
	For $h$ sufficiently small, given $\wG = \widetilde{S}_\Gamma^{-1}  \widetilde{R}_{D, \Gamma} S_\Gamma \uG$ for $\uG \in \widehat{W}_\Gamma$, there exists a positive constant $C_7 >0$ such that $\forall \uG \in \widehat{W}_\Gamma$
	\begin{equation*}
		\norm{\wag - \uag}{\Omega}^2 \leq C_7 \ H^2 \ K^2 \max_{\substack{ k = 1,\dots,N \\ \star = i,e}} \dfrac{ \Phi^{\star, k} (H,h)  - 1}{\tau \sigma_m^{\star, k}} \ \normBgamma{\uG},
	\end{equation*}
	with $\Phi$ and $K$ are defined in Equations (\ref{def: K and Phi}).
\end{lemma}

\begin{proof}
	It is useful to observe that, being $\uG \in \widehat{W}_\Gamma$, by defining $\widetilde{R}: \widehat{W} \longrightarrow \widetilde{W}$, then \mbox{$\widetilde{R} \uG \in \widetilde{R} \left( \widehat{W} \right)$}. By using the Poincaré-Friedrichs inequality, the ellipticity Lemma \ref{ellipcoupbound} and \cite[Lemma $7.2$]{tu2008balancing}, we can obtain
	\begin{equation*}
		\begin{aligned}
			\norm{\wag - \uag}{\Omega}^2 &\leq C H^2 \seminormHone{\wag - \uag}{\Omega}^2 	\\		% Poincaré-Friedrichs
			&\leq C H^2  \max_{\substack{ k = 1,\dots,N \\ \star = i,e}} \dfrac{1}{\tau \sigma_m^{\star, k}} \normBtilde{\wag - \widetilde{R} \uag}^2 	\\ 	% ellipticity lemma
			%							&= C H^2  \max_{\substack{ k = 1,\dots,N \\ \star = i,e, w}} \dfrac{1}{\sigma_m^{\star, k}} | \wag - \widetilde{R} \uag |^2_{\widetilde{A}}.		% Lemma 7.2
		\end{aligned}
	\end{equation*}
	We focus on the right-hand side term:
	\begin{equation*}
		\begin{aligned}
			| \wag - \widetilde{R} \uag |^2_{\widetilde{B}} &= \langle \wag - \widetilde{R} \uag, \wag - \widetilde{R} \uag \rangle_{\widetilde{B}}		\\
			&=  \langle \wag - \widetilde{R} \uag, \wag \rangle_{\widetilde{B}}	-  \langle \wag - \widetilde{R} \uag, \widetilde{R} \uag \rangle_{\widetilde{B}}	.
		\end{aligned}
	\end{equation*}
	We treat separately the two terms.
	\begin{itemize}
		\item[(i)] Using Lemma \ref{lemma7.10} and \cite[Lemma $7.2$]{tu2008balancing},
		\begin{equation*}
			\begin{aligned}
				\langle \wag - \widetilde{R} \uag, \wag \rangle_{\widetilde{B}} 
				%				&= | \wag |^2_{\widetilde{B}} - \langle \widetilde{R} \uag, \wag \rangle_{\widetilde{B}}		\\
				&= \normBtildeGamma{\wG}^2 - \normBtilde{\widetilde{R} \uag}^2 
				%				&= \normBtildeGamma{\wG}^2 - \normB{\uag}^2 	\\
				=  \normBtildeGamma{\wG}^2 - \normBgamma{\uG}^2 	\\
				&\leq C \left[ K^2  \Phi^{\star, k} (H,h)  - 1 \right] \normBgamma{\uG}^2 ,
			\end{aligned}
		\end{equation*}
		where we observe that, for $\uag \in \widehat{W}$, $\widetilde{R} \uag \in \widetilde{R} (\widehat{W})$ and, since $\widetilde{B}$ is symmetric, Lemma \ref{lemma7.11} holds:
		\begin{equation*}
			\langle \widetilde{R} \uag, \wag \rangle_{\widetilde{B}} = \langle \wag,  \widetilde{R} \uag \rangle_{\widetilde{B}} = \langle \widetilde{R} \uag, \widetilde{R} \uag \rangle_{\widetilde{B}}.
		\end{equation*}
		\item[(ii)] From Lemma \ref{lemma7.11}, it holds
		\begin{equation*}
			\langle \wag - \widetilde{R} \uag, \ \widetilde{R} \uag \rangle_{\widetilde{B}} = 0,
		\end{equation*}
		since $\widetilde{R} \uag \in \widetilde{R} ( \widehat{W} )$.
	\end{itemize}
	%	It is possible to conclude that
	Then
	\begin{equation*}
		| \wag - \widetilde{R} \uag |^2_{\widetilde{B}} \leq C \left[ K^2   \Phi^{\star, k} (H,h)  - 1 \right] \normBgamma{\uG}^2 ,
	\end{equation*}
	which leads to 
	\begin{equation*}
		\norm{\wag - \uag}{\Omega}^2 \leq C \ H^2 \ K^2 \max_{\substack{ k = 1,\dots,N \\ \star = i,e}} \left[ \dfrac{ \Phi^{\star, k} (H,h)  - 1}{\tau \sigma_m^{\star, k}} \right] \normBgamma{\uG}^2 .
	\end{equation*}
	\qed
\end{proof}

%% Lemma 7.13

\begin{lemma}\label{lemma7.13}
	Given $\vG =  \widetilde{R}_{D, \Gamma}^T \wG$ for $\wG \in \widetilde{W}_\Gamma$, there exists a positive constant $C >0$ such that 
	\begin{equation*}
		\norm{\vag }{\Omega}^2 \leq C \ \dfrac{H^2}{h^2}  \max_{\substack{ k = 1,\dots,N \\ \star = i,e}} \dfrac{\sigma_M^{\star, k}}{\sigma_m^{\star, k}} \left[  \Phi^{\star, k} (H,h) \right]^2 \norm{\wag}{\Omega}^2,
	\end{equation*}
	with $\Phi$ defined in Equation (\ref{def: K and Phi}).
\end{lemma}

\begin{proof}
	As the quantity $\vag - \wag$  has zero average on the interface, $\forall \wG \in \widetilde{W}_\Gamma$, then thanks to Poincaré-Friedrichs inequality, the ellipticity Lemma \ref{ellipcoupbound}, Lemma \ref{projectionlemma} and \cite[Lemma 7.2]{tu2008balancing} we have
	\begin{align*}
		\norm{\vag - \wag}{\Omega}^2 
		%			&\leq C \ H^2 \seminormHone{\vag - \wag}{\Omega}^2 	\\							% Poincaré-Friedrichs
		%			&\leq \max_{\substack{ k = 1,\dots,N \\ \star = i,e}} C \dfrac{H^2}{\tau \sigma_m^{\star, k}} \normBtilde{\vag - \wag}^2 	\\		% Ellipticity lemma
		%			&\leq \max_{\substack{ k = 1,\dots,N \\ \star = i,e}} C \dfrac{H^2}{\tau \sigma_m^{\star, k}} \normBtildeGamma{\vG - \wG}^2	\\			% Lemma 7.2
		&\leq \max_{\substack{ k = 1,\dots,N \\ \star = i,e}} C \dfrac{H^2}{\tau\sigma_m^{\star, k}} \left[  \Phi^{\star, k} (H,h) \right]^2 \normBtildeGamma{\wG}^2 	\\			% Projection lemma
		%			&\leq \max_{\substack{ k = 1,\dots,N \\ \star = i,e}} C H^2 \ \dfrac{\tau \sigma_M^{\star, k}}{\tau \sigma_m^{\star, k}} \left[  \Phi^{\star, k} (H,h) \right]^2 \seminormHone{\wag}{\Omega}^2 	\\		% Ellipticity lemma
		&\leq \max_{\substack{ k = 1,\dots,N \\ \star = i,e}} C \ \dfrac{H^2}{h^2} \ \dfrac{\sigma_M^{\star, k}}{\sigma_m^{\star, k}} \left[  \Phi^{\star, k} (H,h) \right]^2 \norm{\wag}{\Omega}^2 		% inverse inequality
	\end{align*}
	where we used the definition of projection $\Ed$,
	\begin{equation*}
		\vG - \wG = \left( I - \widetilde{R}_{D, \Gamma}^T \right) \wG = \Ed \wG,
	\end{equation*}
	and an inverse inequality.
	\qed
\end{proof}

%% Lemma 7.14

\begin{lemma}\label{lemma7.14}
	Set $\vG = T \uG - \uG$, for $\uG \in \widehat{W}_\Gamma$. For $h$ sufficiently small, there exists a positive constant $C>0$ such that for $\uG \in \widehat{W}_\Gamma$,
	\begin{equation*}
		\norm{\vag}{\Omega}^2 \leq C \ K^2 \dfrac{H^4}{h^2} \max_{\substack{ k = 1,\dots,N \\ \star = i,e}} \dfrac{\sigma_M^{\star, k} }{\tau \left( \sigma_m^{\star, k}\right)^2} \left[ \Phi^{\star, k} (H,h) \right]^2 \left[ \Phi^{\star, k} (H,h) -1 \right]  \normBgamma{\uG}^2,
	\end{equation*}
	with $K$, $\Phi$ defined in Equations (\ref{def: K and Phi}).
\end{lemma}

\begin{proof}
	Since $T \uG = \widetilde{R}^T_{D, \Gamma} \wG $ and $\widetilde{R}^T_{D, \Gamma} \widetilde{R}_\Gamma = I$, then
	$
	%	\begin{align*}
	\vG 
	%			&= T \uG - \uG 		\\
	%			&= \widetilde{R}^T_{D, \Gamma} \wG - \widetilde{R}^T_{D, \Gamma} \widetilde{R}_\Gamma \uG 	\\
	= \widetilde{R}^T_{D, \Gamma} \left( \wG - \widetilde{R}_\Gamma \uG \right). 		\\
	%	\end{align*}
	$
	Therefore, thanks to Lemmas \ref{lemma7.13} and \ref{lemma7.12},
	\begin{align*}
		\norm{\vag}{\Omega}^2 &\leq C \ \dfrac{H^2}{h^2} \max_{\substack{ k = 1,\dots,N \\ \star = i,e}} \dfrac{ \sigma_M^{\star, k} }{\sigma_m^{\star, k}} \left[ \Phi^{\star, k} (H,h)  \right]^2 \norm{\wag - \widetilde{R}_\Gamma \uag}{\Omega}^2 		\\			% Lemma 7.13
		&\leq C \ K^2 \ \dfrac{H^4}{h^2} \max_{\substack{ k = 1,\dots,N \\ \star = i,e}} \dfrac{ \sigma_M^{\star, k} }{\tau \left( \sigma_m^{\star, k} \right)^2 } \left[ \Phi^{\star, k} (H,h)  \right]^2 \left[ \Phi^{\star, k} (H,h) - 1 \right] \normBgamma{\uG}^2.			% Lemma 7.12
	\end{align*}
	\qed
\end{proof}

Finally, we are able to conclude the proof of Theorem \ref{coupledconvergence} by proving the lower bound of Theorem \ref{coupledconvergence}.
By using Lemma \ref{lemma7.3} and \cite[lemmas 7.2 and 7.9]{tu2008balancing},
\begin{align*}
	\langle \uG , \uG \rangle_{B_\Gamma} 
	&= \langle \uG , \uG \rangle_{S_\Gamma}					% Lemma 7.2
	%		&= \uG^T S_\Gamma \uG	\\
	= \uG^T \widetilde{R}^T_\Gamma \widetilde{R}_{D, \Gamma} S_\Gamma \uG 
	%		&= \uG^T \widetilde{R}^T_\Gamma \widetilde{S}_\Gamma  \widetilde{S}_\Gamma^{-1}  \widetilde{R}_{D, \Gamma} S_\Gamma \uG \\
	= \uG^T \widetilde{R}^T_\Gamma \widetilde{S}_\Gamma \wG \\
	&\leq K \normBtildeGamma{\wG} \normBtildeGamma{\widetilde{R}_\Gamma \uG} 			% Lemma 7.3
	= K \langle \uG , T \uG \rangle_{S_\Gamma}^{1/2}  \langle \uG , \uG \rangle_{B_\Gamma}^{1/2},  		% Lemma 7.9 
\end{align*}
which means
\begin{equation}\label{eqlowerbound}
	\langle \uG , \uG \rangle_{B_\Gamma}  \leq K^2 \langle \uG , T \uG \rangle_{S_\Gamma} .
\end{equation}
Since \cite[Lemma 7.2]{tu2008balancing} and $\langle \uG , \uG \rangle_{Z_\Gamma} = 0 $, we obtain
\begin{align*}
	\langle \uG , T \uG \rangle_{S_\Gamma}  
	%		&= \langle \uG , T \uG \rangle_{B_\Gamma}  + \langle \uG , T \uG \rangle_{Z_\Gamma}  - \langle \uG , \uG \rangle_{Z_\Gamma} 	\\ 	% Lemma 7.2
	&= \langle \uG , T \uG \rangle_{B_\Gamma}  	+ \langle \uG , T \uG - \uG \rangle_{Z_\Gamma} .
\end{align*}
Therefore, (\ref{eqlowerbound}) becomes
\begin{equation*}
	\begin{aligned}
		\langle \uG , \uG \rangle_{B_\Gamma} \leq K^2 \langle \uG , T \uG \rangle_{B_\Gamma}  + K ^2\langle \uG , T \uG - \uG \rangle_{Z_\Gamma}.
	\end{aligned}
\end{equation*}
However, thanks to Lemmas \ref{lemma7.4} and \ref{lemma7.14},
\begin{equation*}
	\begin{aligned}
		\langle \uG , T \uG - &\uG \rangle_{Z_\Gamma} \leq K \normBgamma{\uG} \norm{\vag}{\Omega}			\\ 		% Lemma 7.4
		&\leq C \ K^2 \dfrac{H^2}{h} \max_{\substack{ k = 1,\dots,N \\ \star = i,e}} \dfrac{\left[ \sigma_M^{\star, k} \right]^{1/2} }{\tau^{\frac{1}{2}} \sigma_m^{\star, k}} \Phi^{\star, k} (H,h)  \left[  \Phi^{\star, k} (H,h) - 1\right]^{1/2} \langle \uG , \uG \rangle_{B_\Gamma},
	\end{aligned}
\end{equation*}
from which it follows
\begin{equation*}
	\begin{aligned}
		\langle \uG , \uG \rangle_{B_\Gamma} &\leq K^2 \langle \uG , T \uG \rangle_{B_\Gamma} \\ &+ C \ K^4 \dfrac{H^2}{h} \max_{\substack{ k = 1,\dots,N \\ \star = i,e}} \dfrac{\left[ \sigma_M^{\star, k} \right]^{1/2} }{\tau^{\frac{1}{2}} \sigma_m^{\star, k}} \Phi^{\star, k} (H,h)  \left[  \Phi^{\star, k} (H,h) - 1 \right]^{1/2} \langle \uG , \uG \rangle_{B_\Gamma}.
	\end{aligned}
\end{equation*}
In conclusion, we have the lower bound
\begin{equation*}
	c_0 \langle \uG , \uG \rangle_{B_\Gamma}  \leq K^2 \langle \uG , T \uG \rangle_{B_\Gamma},
\end{equation*}
where
\begin{equation*}
	c_0 = 1 - K^4 \dfrac{H^2}{h} \max_{\substack{ k = 1,\dots,N \\ \star = i,e}} \dfrac{\left[ \sigma_M^{\star, k} \right]^{1/2} }{\tau^{\frac{1}{2}} \sigma_m^{\star, k}} \Phi^{\star, k} (H,h)  \left[  \Phi^{\star, k} (H,h) - 1 \right]^{1/2}.
\end{equation*}

\section{Parallel numerical tests} \label{parallel tests}

We now validate our theoretical results with several parallel numerical experiments, performed on the supercomputer Galileo from the Cineca centre (\hyperlink{https://www.hpc.cineca.it/}{www.hpc.cineca.it}), a Linux Infiniband cluster equipped with 1084 nodes, each with 36 2.30 GHz Intel Xeon E5-2697 v4 cores and 128 GB/node, for a total of 28 804 cores. Our C code is based on the PETSc library \cite{balay2019petsc} from the Argonne National Laboratory. 
\\
The numerical tests are run both on a thin Cartesian slab and on a idealized left ventricular geometry, modeled as a truncated ellipsoid (see Fig. \ref{fig_ell_meshgrid}). 

\begin{figure}[!h]
	\centering
	\includegraphics[scale=.25]{ 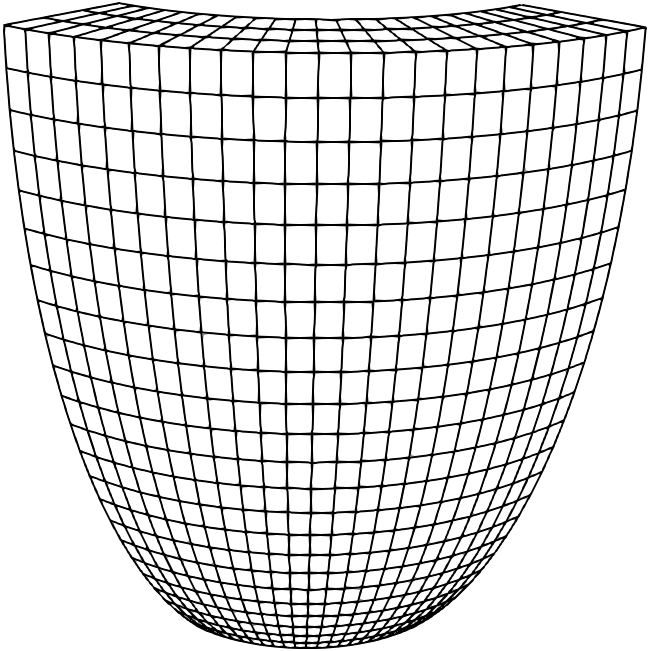} 
	\includegraphics[scale=.25]{ 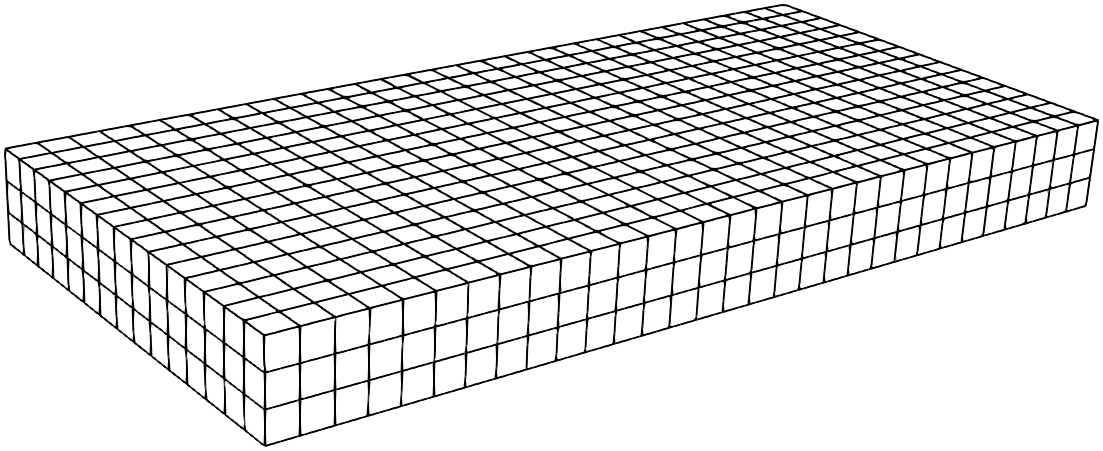} 
	\caption{Portion of left ventricle, idealized as a truncated ellipsoid (top) and slab geometry (bottom). }
	\label{fig_ell_meshgrid}
\end{figure}
The latter is described in ellipsoidal coordinates by the parametric equations
%\begin{align*}
%	\mathbf{x} = a (r) \cos \theta \cos \varphi,		 \qquad &\mathbf{y} = b (r) \cos \theta \sin \varphi,  \qquad \mathbf{z} = c (r) \sin \varphi , \\
%	\theta_\text{min} \leq \theta \leq \theta_\text{max}, 		\qquad &\varphi_\text{min} \leq \varphi \leq \varphi_\text{max}, 	\qquad \qquad 0 \leq r \leq 1,		
%\end{align*}
\begin{eqnarray}
	\begin{dcases}
		\mathbf{x} = a (r) \cos \theta \cos \varphi,		&\theta_\text{min} \leq \theta \leq \theta_\text{max},		\nonumber \\
		\mathbf{y} = b (r) \cos \theta \sin \varphi, 		&\varphi_\text{min} \leq \varphi \leq \varphi_\text{max}, 	\nonumber  \\
		\mathbf{z} = c (r) \sin \varphi ,	\qquad			 &0 \leq r \leq 1,		\nonumber 
	\end{dcases}
\end{eqnarray}
where $a (r)  = a_1 + r (a_2 - a_1)$, $b (r)  = b_1 + r (b_2 - b_1)$ and $c (r)  = c_1 +r (c_2 - c_1)$ with $a_{1,2}$, $b_{1,2}$ and $c_{1,2}$ given coefficients defining the main axes of the ellipsoid.  
\\
We assume that fibers rotate intramurally linearly with the depth for a total amount of $120^o$ proceeding counterclockwise from epicardium ($r = 1$, outer surface of the truncated ellipsoid) to endocardium ($r=0$, inner surface). \\

We choose the Roger-McCulloch ionic model for the description of the ionic current flows, whose physiological parameters on Table \ref{ioniccoeff} are taken from \cite{rogers1994collocation}.
\\
Regarding the Bidomain physiological coefficients in Table \ref{bidocoeff}, we refer to \cite{franzone2014mathematical}. \\

\begin{table}[ht]
	\centering
	\caption{Conductivity coefficients for the Roger-McCulloch ionic model.}
	\label{ioniccoeff}
	\begin{tabular}{*{5}{r}}
		\hline
		\multicolumn{5}{c}{Ionic parameters}	\\
		\hline
		G 				 &$1.2$ $\Omega^{-1}$ cm$^{-2}$			 &  &$v_{th}$	 &13 mV 	\\
		$\eta_1$	  	&$4.4$ $\Omega^{-1}$ cm$^{-1}$		  & &$v_p$		  &100 mV \\
		$\eta_2$  	 &$0.012$	 											&	&$C_m$ 		&$1$ mF/cm$^2$ \\
		\hline
	\end{tabular}
	
	\vspace*{5mm}
	
	\caption{Conductivity coefficients for the Bidomain model.}
	\label{bidocoeff}
	\begin{tabular}{*{5}{c}}
		\hline
		\multicolumn{5}{c}{Bidomain conductivity coefficients}	\\
		\hline
		$\sigma_l^i$		  & $3 \times 10^{-3} \Omega^{-1} \text{ cm}^{-1}$ 				& & $\sigma_l^e$		& $2 \times 10^{-3} \Omega^{-1} \text{ cm}^{-1}$	\\
		$\sigma_t^i$		 & $3.1525 \times 10^{-4} \Omega^{-1} \text{ cm}^{-1}$		& &$\sigma_t^e$	    & $1.3514 \times 10^{-3} \Omega^{-1} \text{ cm}^{-1}$	\\
		$\sigma_n^i$   		& $3.1525 \times 10^{-5} \Omega^{-1} \text{ cm}^{-1}$	    &  &$\sigma_n^e$	   & $6.757 \times 10^{-4} \Omega^{-1} \text{ cm}^{-1}$	\\
		\hline
	\end{tabular}
\end{table}
In order to produce the propagation of the electric potential, we apply for $1$ ms to the surface of the domain representing the endocardium an external stimulus of $\Iapp = 100$ mA/cm$^3$. Instead, on the slab geometry, the current is applied in one corner of the domain, over a spheric volume of radius $0.1$ cm. \\ 
We report in Figures \ref{fig_bido_snapshots_v_multi} and \ref{fig_bido_snapshots_ue_multi} the time evolution of the transmembrane and extra-cellular potentials respectively from the endocardial side of the portion of the truncated ellipsoid. The external stimulus $\Iapp$ is applied on five different sites on the endocardium layer, positioned at the apex of the idealized left ventricle. \\
%Figures \ref{fig_bido_snapshots_ue} and \ref{fig_bido_snapshots_v} show the time evolution of the extra-cellular and transmembrane potentials respectively, from the epicardial view of a portion of the idealized left ventricle when the external stimulus $\Iapp$ is applied on the epicardium layer.
%
The boundary conditions are for an insulated tissue, while the initial conditions represent a resting potential.
The time step is fixed $\tau = 0.05$ ms. \\
%
%We use the non-linear solver SNES within PETSc library \cite{balay2019petsc}, which implements a Newton method with cubic backtracking linesearch.
We adopt a classical implementation of the Newton method, with residual stopping criterion with tolerance $10^{-4}$.
The non-symmetric linear system arising from the discretization of the Jacobian problem at each Newton step is solved with the Generalized Minimal Residual (GMRES) method, preconditioned by the BDDC preconditioner (included in the PETSC library) and the {\it Boomer} Algebraic MultiGrid (bAMG, from the Hypre library \cite{falgout2002hypre}).   

\begin{figure}[H]
	\caption{Snapshots (every 5 ms) of transmembrane potential $v$ time evolution. For each time frame, we report the epicardial view of a portion of the left ventricle, modeled as a truncated ellipsoid. }
	\begin{center}
		\begin{multicols}{3}
			t = 10 ms 	\vspace*{2mm} \\	\includegraphics[scale=.17]{  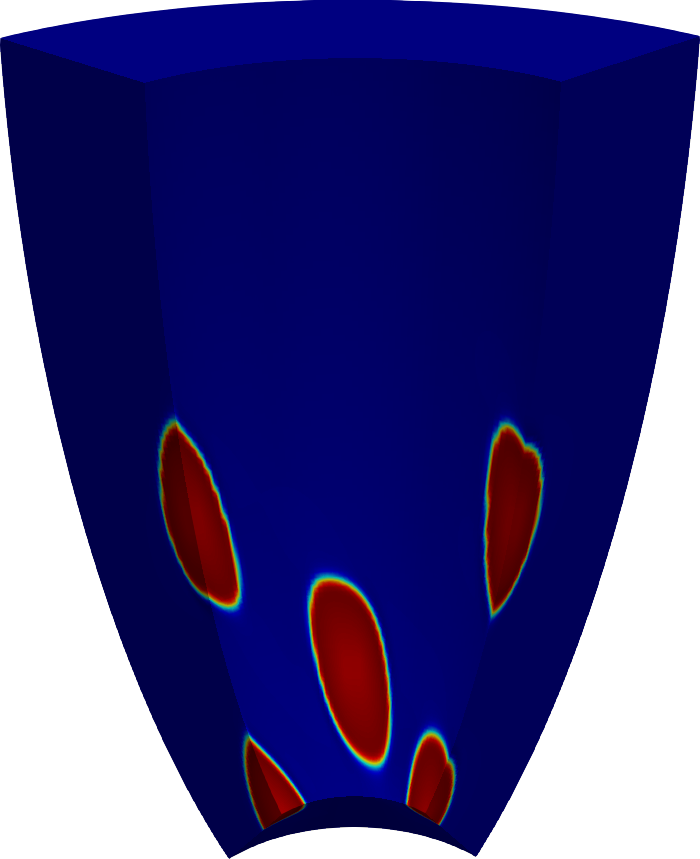} \\	\vspace*{2mm}
			t = 25 ms 	\vspace*{2mm} \\	\includegraphics[scale=.17]{  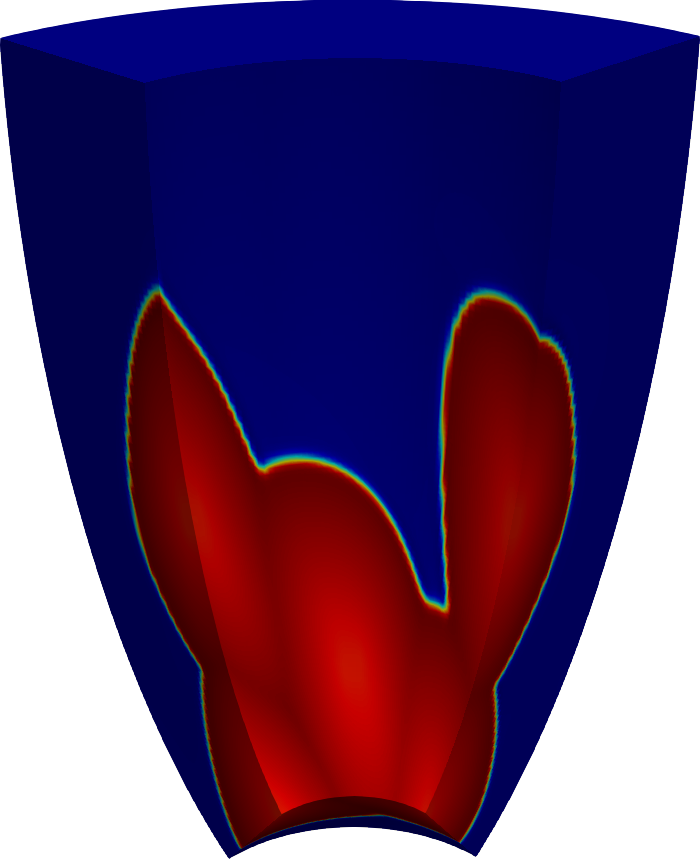} \\	\vspace*{2mm}
			t = 40 ms 	\vspace*{2mm} \\	\includegraphics[scale=.17]{  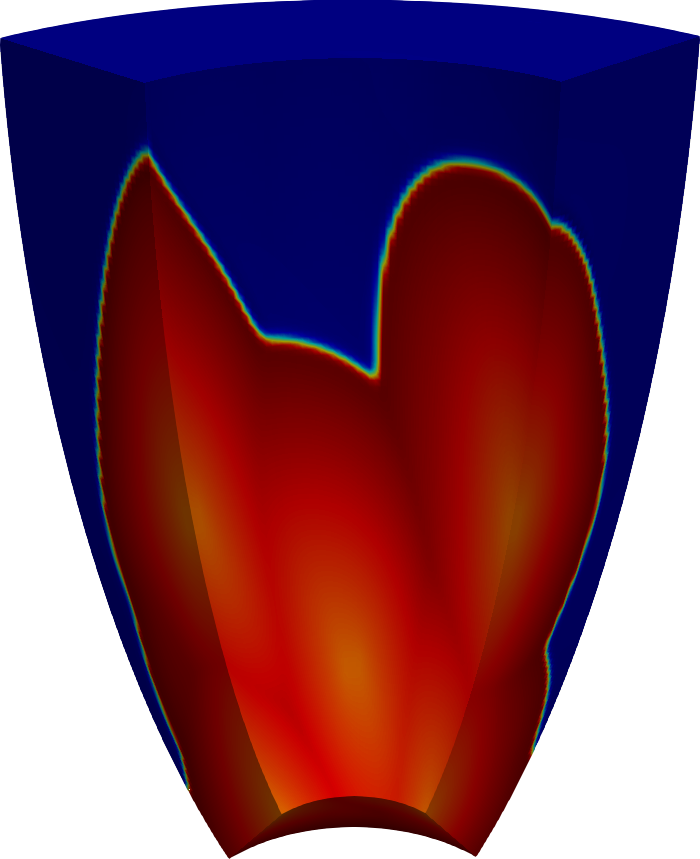} \\
			\columnbreak
			t = 15 ms 	\vspace*{2mm} \\	\includegraphics[scale=.17]{  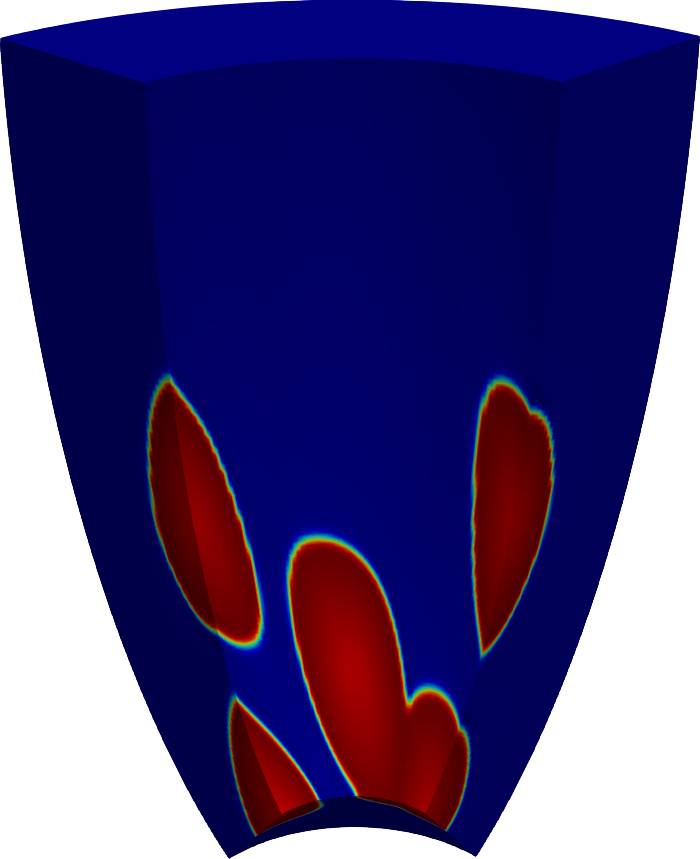} \\	\vspace*{2mm}
			t = 30 ms 	\vspace*{2mm} \\	\includegraphics[scale=.17]{  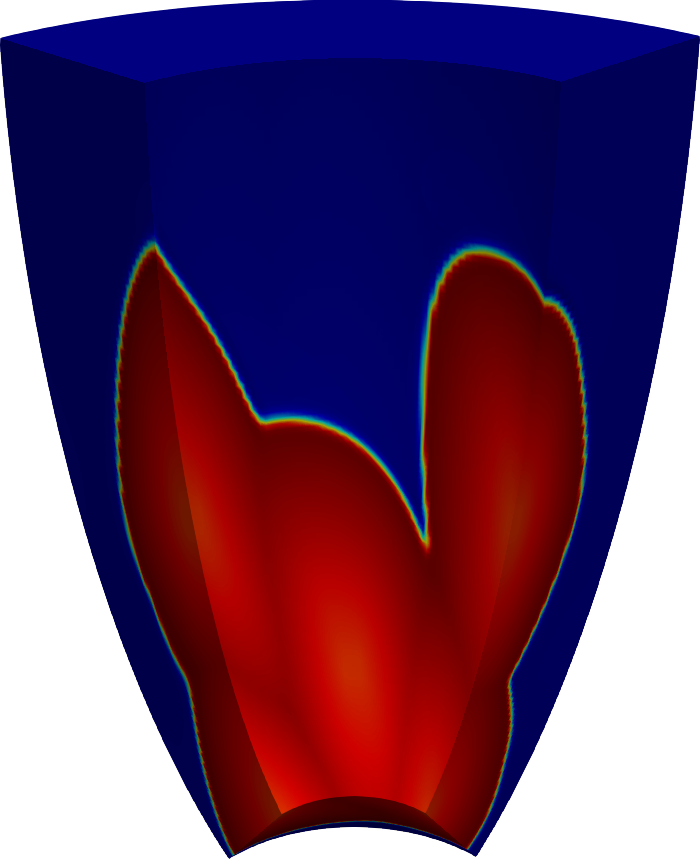} \\	\vspace*{2mm}
			t = 45 ms 	\vspace*{2mm} \\	\includegraphics[scale=.17]{  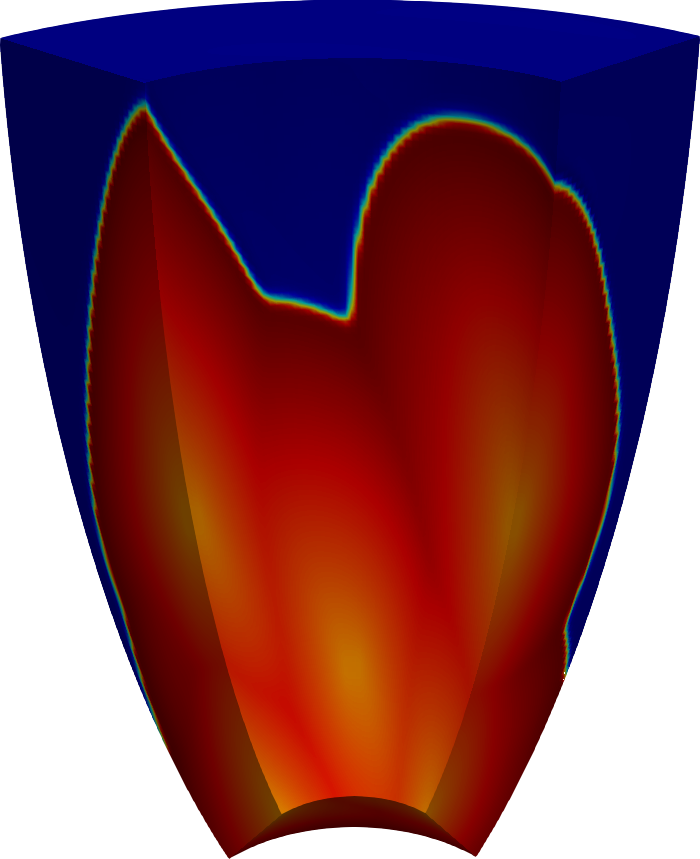} \\
			\columnbreak
			t = 20 ms 	\vspace*{2mm} \\	\includegraphics[scale=.17]{  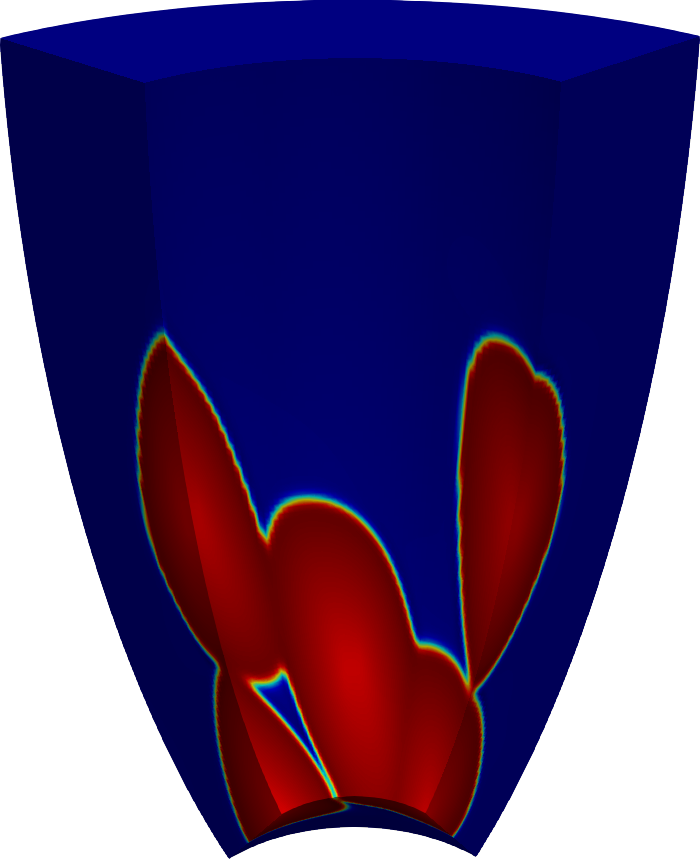} \\	\vspace*{2mm}
			t = 35 ms 	\vspace*{2mm} \\	\includegraphics[scale=.17]{  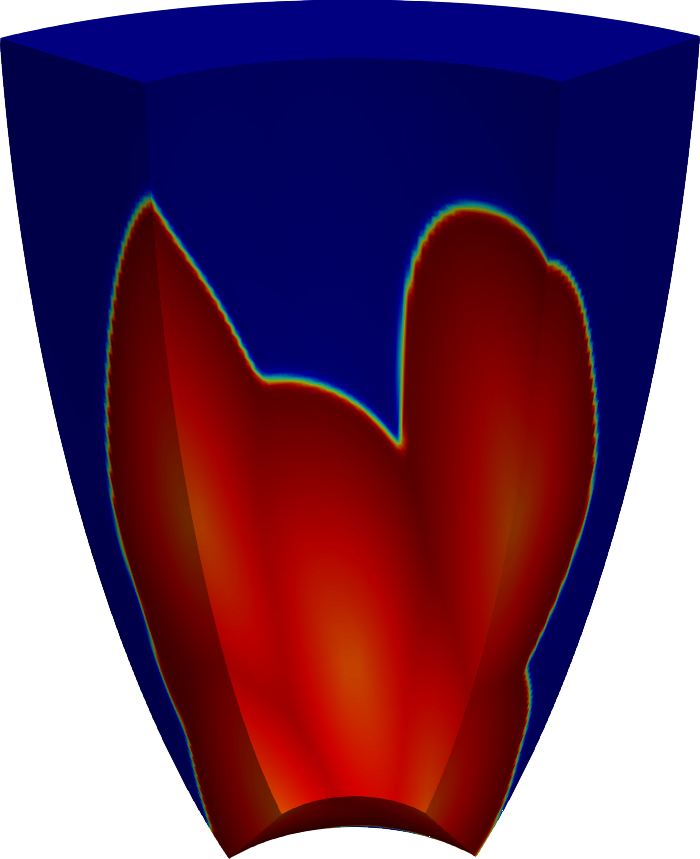} \\	\vspace*{2mm}
			t = 50 ms 	\vspace*{2mm} \\	\includegraphics[scale=.17]{  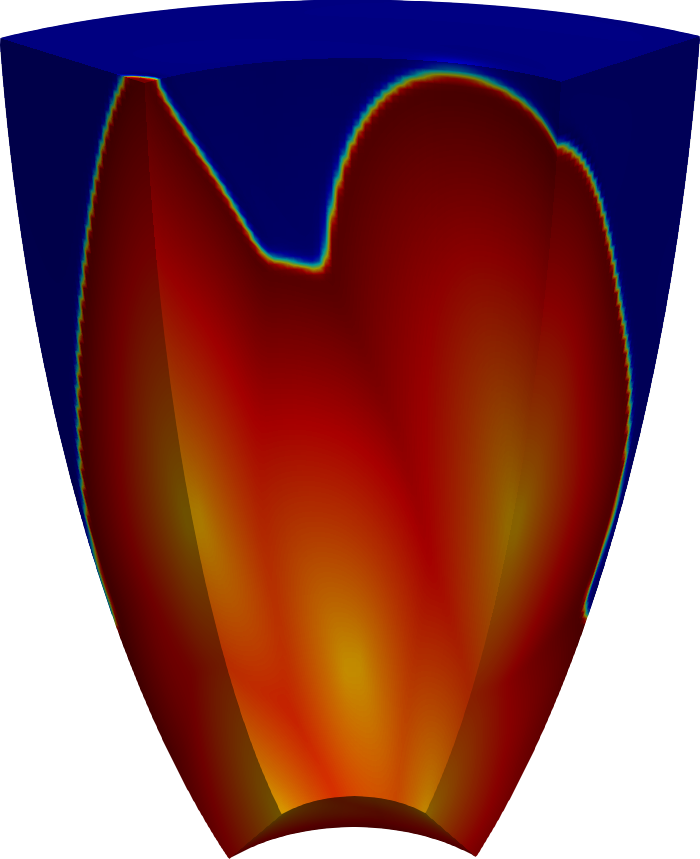}
		\end{multicols}
		\includegraphics[scale=.55]{  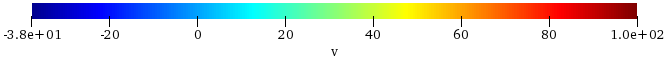}
	\end{center}
	\label{fig_bido_snapshots_v_multi}
\end{figure}

\begin{figure}[H]
	\caption{Snapshots (every 5 ms) of extra-cellular potential $u_e$ time evolution. For each time frame, we report the epicardial view of a portion of the left ventricle, modeled as a truncated ellipsoid. }
	\begin{center}
		\begin{multicols}{3}
			t = 10 ms 	\vspace*{2mm} \\	\includegraphics[scale=.17]{  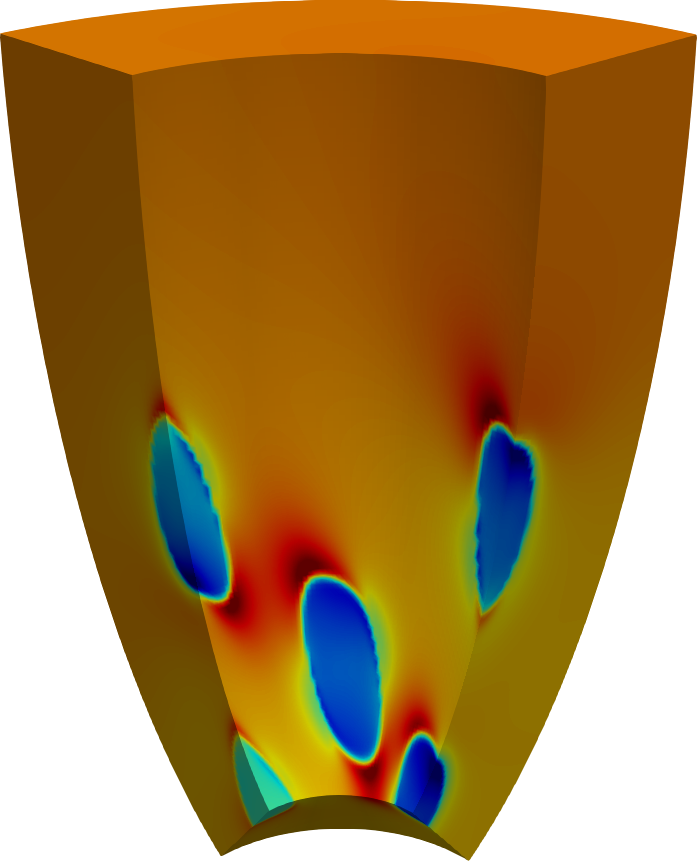} \\	\vspace*{2mm}
			t = 25 ms 	\vspace*{2mm} \\	\includegraphics[scale=.17]{  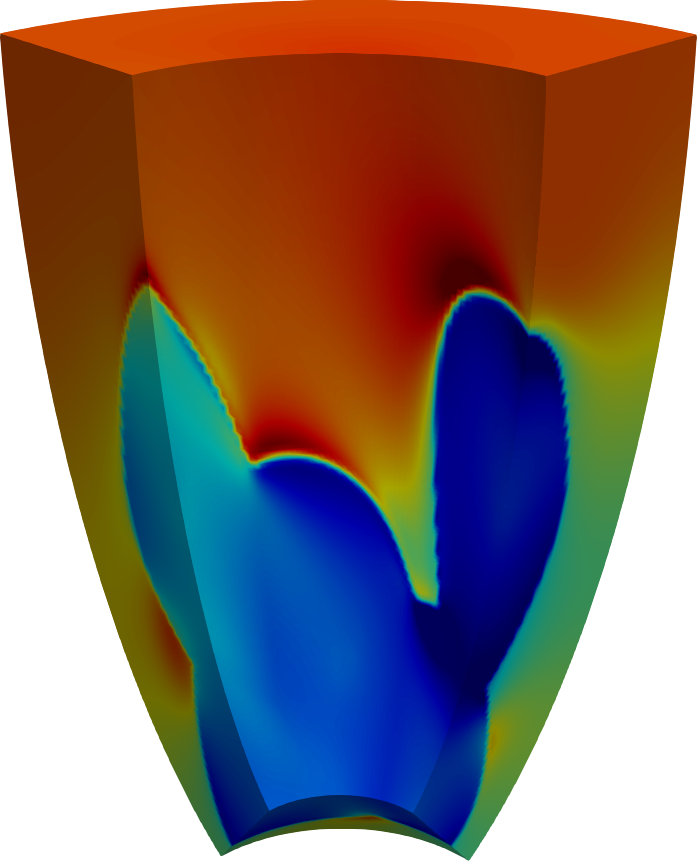} \\	\vspace*{2mm}
			t = 40 ms 	\vspace*{2mm} \\	\includegraphics[scale=.17]{  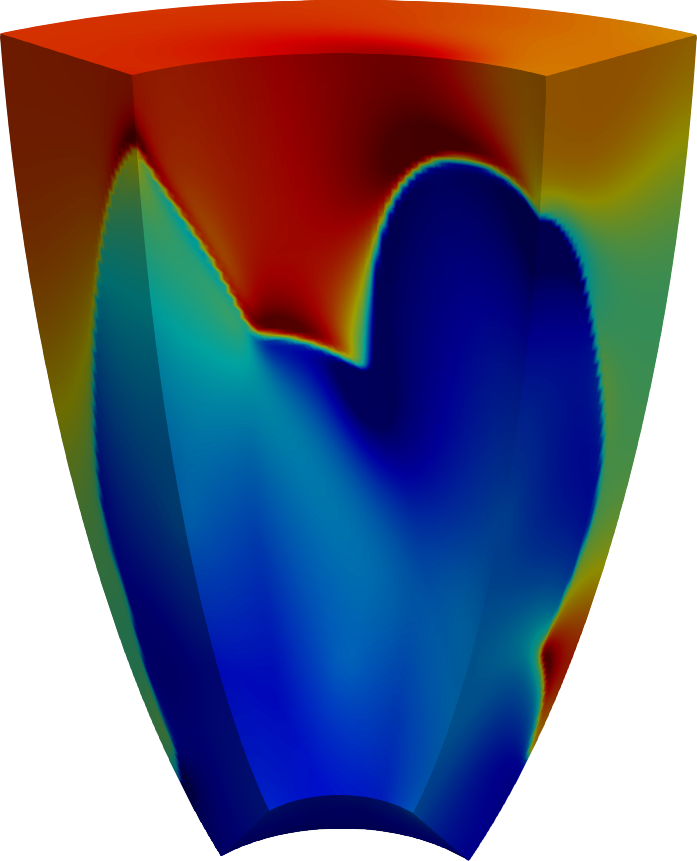} \\
			\columnbreak
			t = 15 ms 	\vspace*{2mm} \\	\includegraphics[scale=.17]{  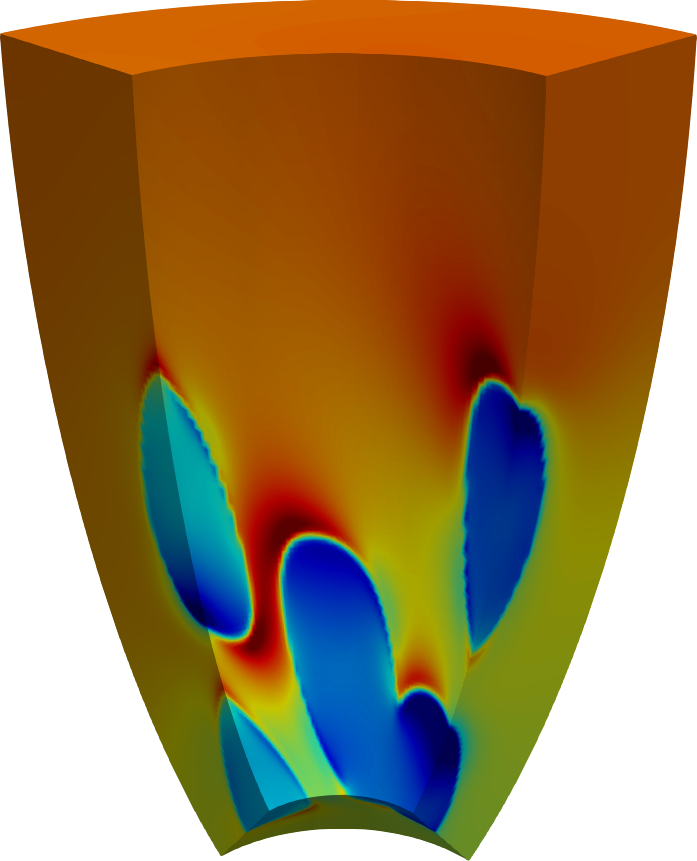} \\	\vspace*{2mm}
			t = 30 ms 	\vspace*{2mm} \\	\includegraphics[scale=.17]{  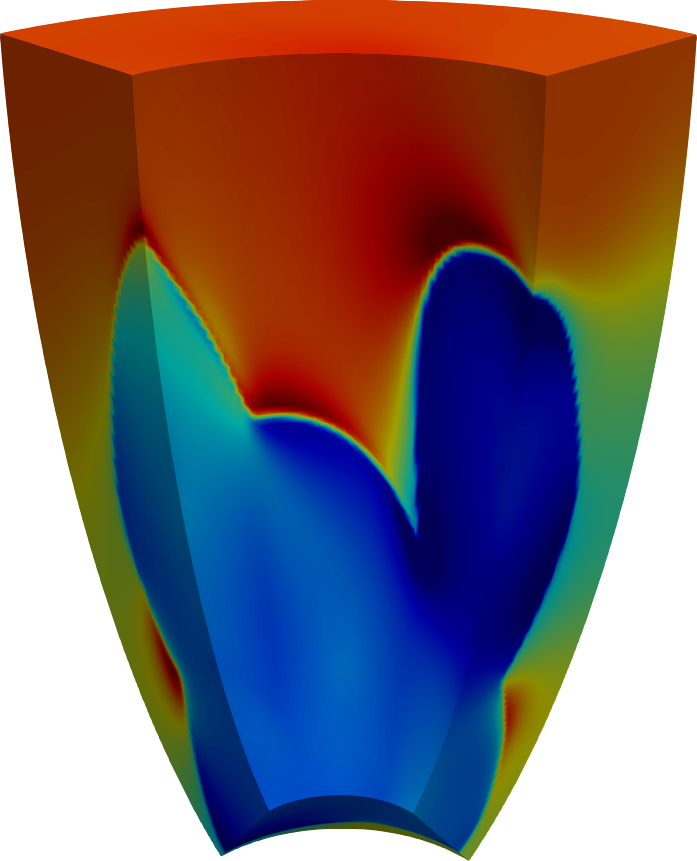} \\	\vspace*{2mm}
			t = 45 ms 	\vspace*{2mm} \\	\includegraphics[scale=.17]{  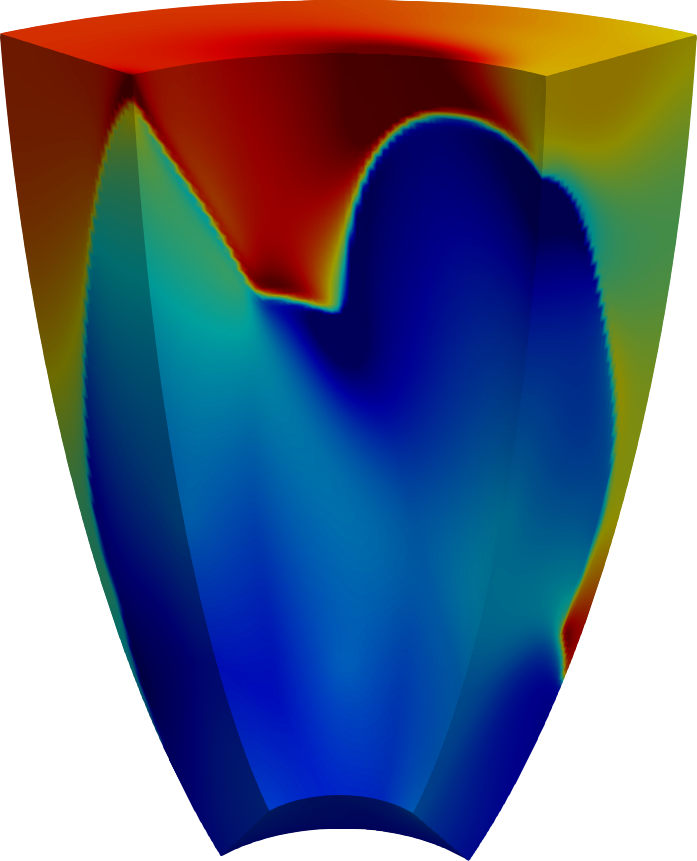} \\
			\columnbreak
			t = 20 ms 	\vspace*{2mm} \\	\includegraphics[scale=.17]{  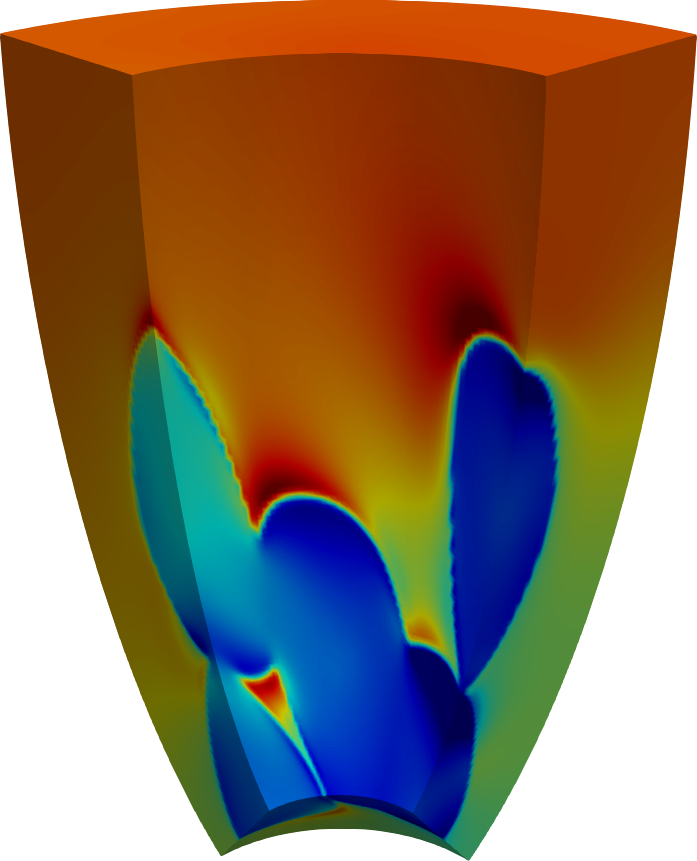} \\	\vspace*{2mm}
			t = 35 ms 	\vspace*{2mm} \\	\includegraphics[scale=.17]{  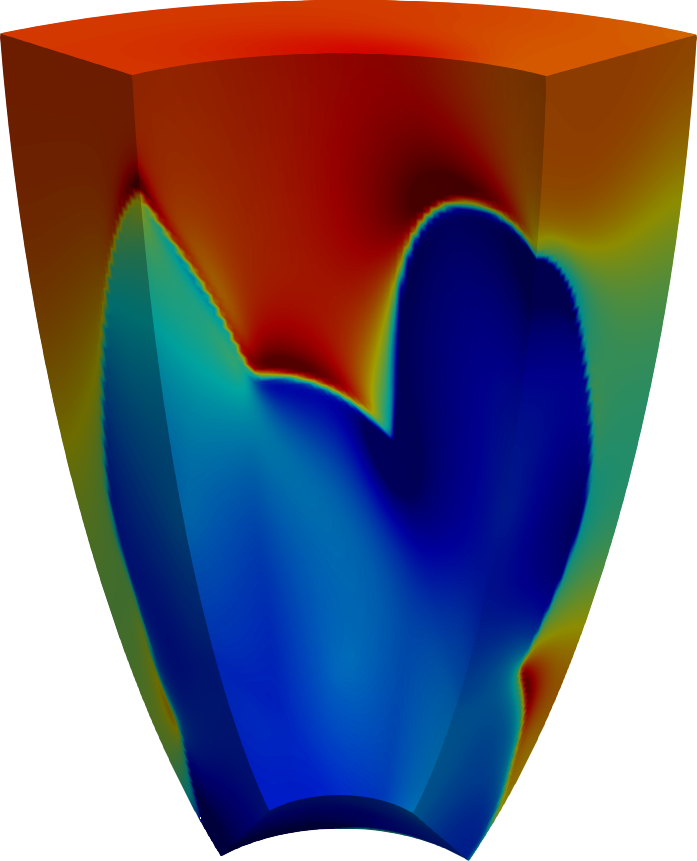} \\	\vspace*{2mm}
			t = 50 ms 	\vspace*{2mm} \\	\includegraphics[scale=.17]{  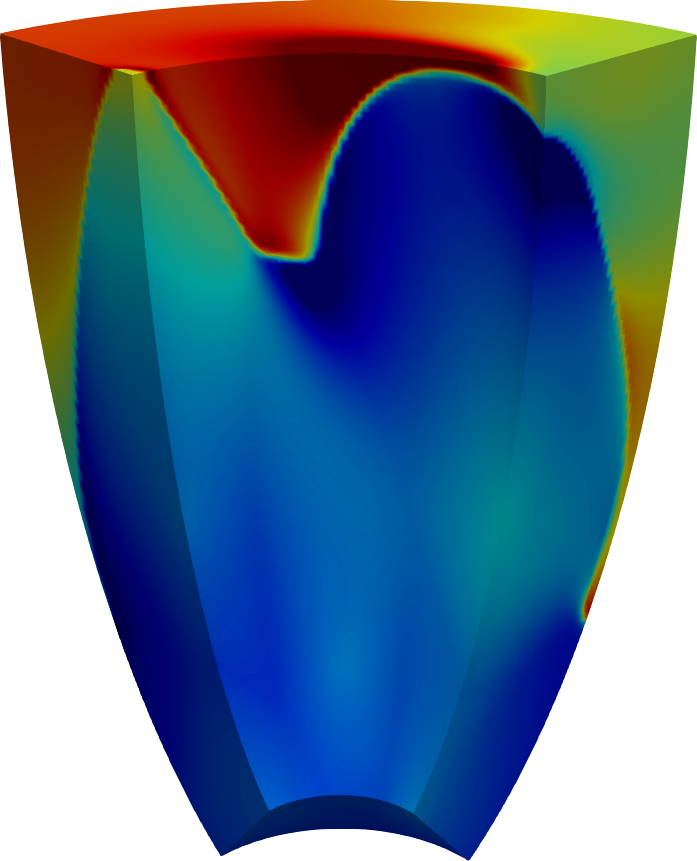}
		\end{multicols}
		\includegraphics[scale=.55]{  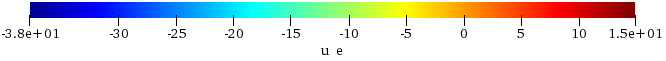}
	\end{center}
	\label{fig_bido_snapshots_ue_multi}
\end{figure}

{\bf Test 1: weak scaling.}
The first set of tests we report here is a weak scaling test on both slab and ellipsoidal domains. For both cases, we fix the local mesh size to $12\cdot 12\cdot 12$ and we increase the number of subdomains from $32$ to $256$, thus resulting in an increasing slab geometry and in an increasing portion of ellipsoid. In this way, the dofs are increasing from $180$k up to $1$ million and a half.
From Tables \ref{table_bido_weak_slab_coupled} and \ref{table_bido_weak_ell_coupled}, it is evident how the dual-primal algorithm has a better performance respect to the bAMG: the average number of linear iteration per Newton iteration (lit) does not increase with the number of subdomains and is clearly lower.
As a matter of fact, for the slab geometry we can observe an increasing reduction rate from $85\%$ up to $93\%$ for the average linear iterations, while for the ellipsoid it varies between $65\%$ and $90\%$.
In contrast, BDDC's average computational time is higher due to the need of communication between the implemented structures. 

\begin{table}[H]
	\caption{Weak scaling test.  Slab domain, local mesh of $12\cdot 12\cdot 12$ elements. Simulations of 2 ms of cardiac activation with $dt = 0.05$ ms (40 time steps). Comparison of Newton-Krylov solvers preconditioned by bAMG and BDDC. Average Newton iterations per time step (nit); average GMRES iterations per Newton iteration (lit); average CPU solution time per time step (time) in seconds.}
	\label{table_bido_weak_slab_coupled}
	\centering
	\begin{tabular}{*{13}{r}}
		\hline
		\multirow{2}{*}{subds.}	&& \multirow{2}{*}{global mesh}				&& \multirow{2}{*}{dofs}			&& \multicolumn{3}{c}{bAMG}	&& \multicolumn{3}{c}{BDDC}	\\
		&& 				&& 		&& nit 	& lit	& time		&& nit	& lit	& time		\\
		\cline{1-1} \cline{3-3} \cline{5-5} \cline{7-9} \cline{11-13} 
		32		&& $48\cdot 48\cdot 24$  	 && 180,075		&& 1	& 100	& 0.6		&& 1	& 16	& 3.2		\\
		64		&& $96\cdot 48\cdot 24$	 	&& 356,475		&& 1	& 127	& 0.9		&& 1	& 16	& 3.4		\\
		128		&& $96\cdot 96\cdot 24$		&& 705,675		&& 1	& 168	& 1.3		&& 1	&17		& 3.4		\\
		256		&& $192\cdot 96\cdot 24$	&& 1,404,075	&& 1	& 243	& 1.9		&& 1	& 17	& 4.3		\\
		%		512		&& $192\cdot 192\cdot 24$	&& 2,793,675	&& 1	&318 	& 20.0		&& 1	& 48	& 10.1	\\
		%		1024	&& $384\cdot 192\cdot 24$	&& 5,572,875	&& 1	&405	& 29.6		&& 1	& 63	& 13.8		\\
		%		2048	&& $384\cdot 384\cdot 24$	&& 11,116,875	&& 1	&536	& 40.3		&& 1	& 78	& 33.5		\\
		\hline
	\end{tabular}
	\vspace*{5mm}
	%\end{table}
	%	
	%\begin{table}[H]
	\caption{Weak scaling test.  Ellipsoidal domain, local mesh of $12\cdot 12\cdot 12$ elements. Simulations of 2 ms of cardiac activation with $dt = 0.05$ ms (40 time steps). Comparison of Newton-Krylov solvers preconditioned by bAMG and BDDC. Average Newton iterations per time step (nit); average GMRES iterations per Newton iteration (lit); average CPU solution time per time step (time) in seconds.}
	\label{table_bido_weak_ell_coupled}
	\centering
	\begin{tabular}{*{13}{r}}
		\hline
		\multirow{2}{*}{subds.}	&& \multirow{2}{*}{global mesh}				&& \multirow{2}{*}{dofs}			&& \multicolumn{3}{c}{bAMG}	&& \multicolumn{3}{c}{BDDC}	\\
		&& 				&& 		&& nit 	& lit	& time		&& nit	& lit	& time		\\
		\cline{1-1} \cline{3-3} \cline{5-5} \cline{7-9} \cline{11-13} 
		32		&& $48\cdot 48\cdot 24$  	 && 180,075		&& 2	& 142	& 1.5		&& 2	& 45	& 6.8 		\\
		64		&& $96\cdot 48\cdot 24$	 	&& 356,475		&& 2	& 145	& 1.9		&& 2	& 32	& 6.9		\\
		128		&& $96\cdot 96\cdot 24$		&& 705,675		&& 2	& 158	& 2.1		&& 2	&23		& 7.0		\\
		256		&& $192\cdot 96\cdot 24$	&& 1,404,075	&& 2	& 212	& 3.2		&& 2	& 23	& 8.5		\\
		%		512		&& $192\cdot 192\cdot 24$	&& 2,793,675	&& 2	&318 	& 20.0		&& 1	& 48	& 10.1	\\
		%		1024	&& $384\cdot 192\cdot 24$	&& 5,572,875	&& 2	&405	& 29.6		&& 1	& 63	& 13.8		\\
		%		2048	&& $384\cdot 384\cdot 24$	&& 11,116,875	&& 2	&536	& 40.3		&& 1	& 78	& 33.5		\\
		\hline
	\end{tabular}
\end{table}

\begin{figure}[H]
	\centering
	\includegraphics[scale=.42]{ 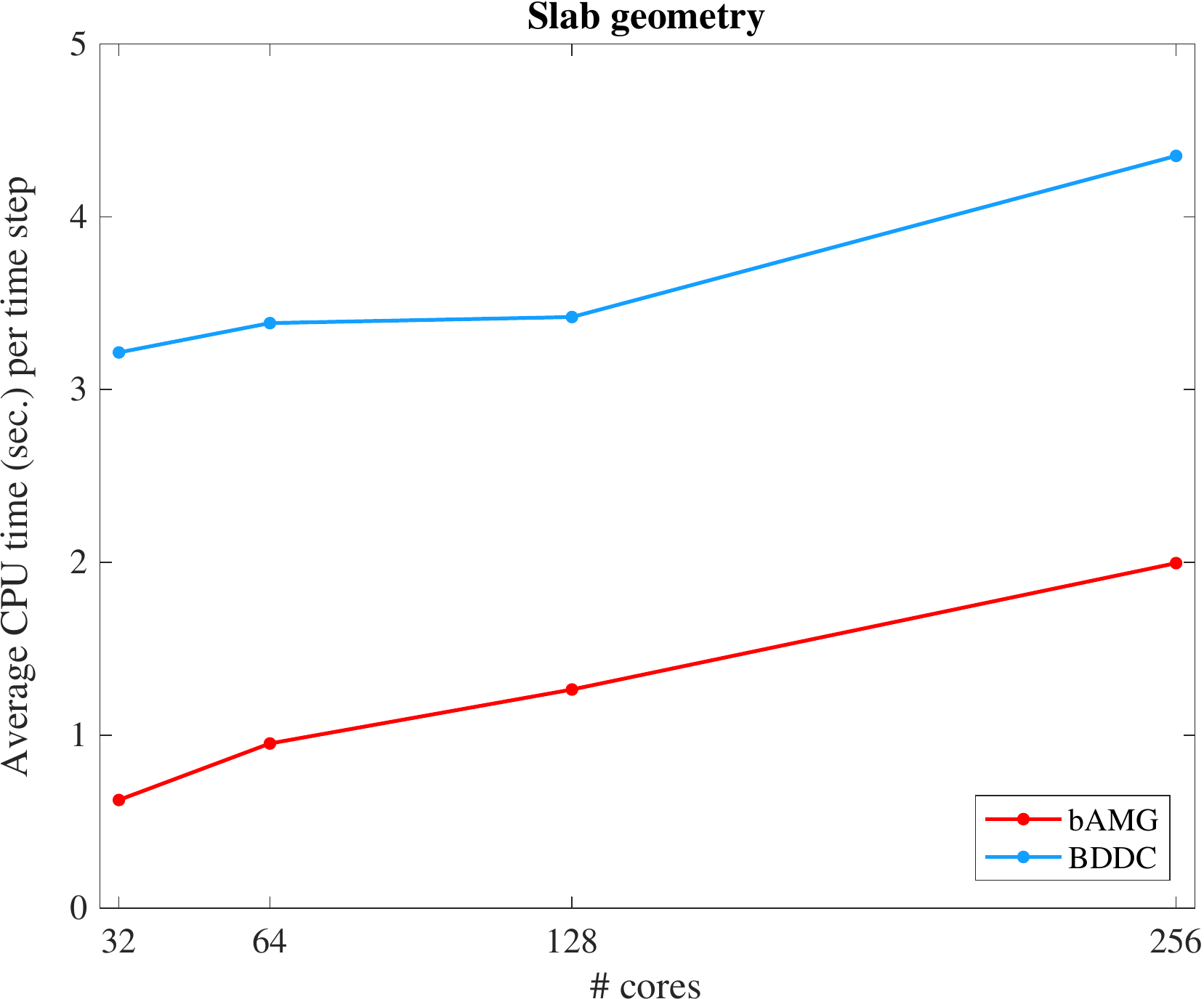} 
	\includegraphics[scale=.42]{ 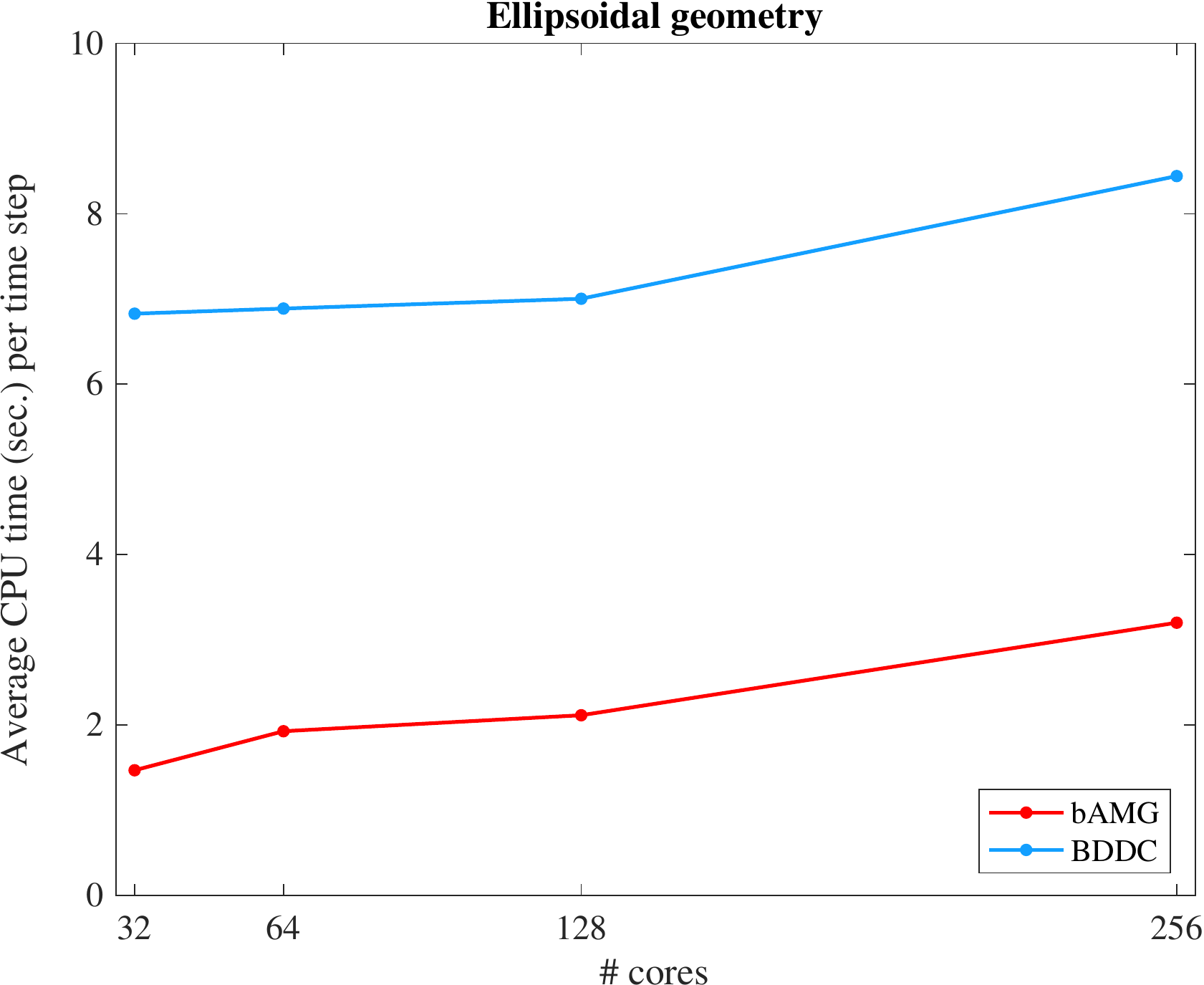}
	\caption{Weak scaling test. Slab (left) and ellipsoidal (right) geometries, local mesh of $12 \cdot 12 \cdot 12$ elements. Simulations of 2 ms of cardiac activation with $dt = 0.05$ ms (40 time steps). Comparison of average CPU time per time step, in seconds.}
	\label{fig_bido_weak_time_coupled}
\end{figure}

{\bf Test 2: strong scaling.} 
We perform a strong scaling test for the two geometries: we fix the global mesh to $128 \cdot 128 \cdot 24$ elements (resulting in more than $3$ millions of dofs) and we increase the number of subdomains from $32$ to $256$.\\
We can observe from Table \ref{table_bido_strong_slab_galileo_coupled} that, as the local number of dofs decreases, the preconditioner with the better balance in term of average linear iterations and CPU time per time step is BDDC: although slightly higher CPU times, the performance is balanced by the average number of linear iterations, which is certainly lower. 
We test the efficiency of the proposed solver on the parallel architecture by computing the parallel speedup $= \frac{T_1}{T_N}$, which is the ratio between the runtime needed by 1 processor ($T_1$) and the average runtime needed by $N$ processors ($T_N$) to solve the problem.
In both cases, BDDC preconditioner does not  reach ideal speedup, while bAMG is sub-optimal (see Fig. \ref{fig_bido_speedup_coupled}).  

\begin{table}[H]
	\caption{Strong scaling test on slab and ellipsoidal domain. Global mesh of $128\cdot 128\cdot 24$ elements, 3,244,995 dofs.	Simulations of 2 ms of cardiac activation with $dt = 0.05$ ms, for a total amount fo 40 time steps. Comparison of Newton-Krylov solvers with bAMG and BDDC preconditioners. Average Newton iterations per time step (nit); average GMRES iterations per Newton iteration (lit); average CPU solution time per time step (time) in seconds; parallel speedup ($S_p$), with ideal speedup in brackets.}
	\label{table_bido_strong_slab_galileo_coupled}
	\centering
	\begin{tabular}{*{11}{r}}
		%		\hline
		\multicolumn{11}{c}{Slab domain} \\
		\hline
		\multirow{2}{*}{subds.} 		&& \multicolumn{4}{c}{bAMG}     	&& \multicolumn{4}{c}{BDDC}     	\\
		&& nit   & lit   & time	& $S_p$		&& nit   & lit   & time	& $S_p$		\\
		\cline{1-1} \cline{3-6} \cline{8-11}  
		32     	&& 1	& 183	& 9.5	& -			    && 1	& 20	& 98.7	& -				\\
		64      && 1	& 196	& 5.4	& 1.7 (2)	&& 1	& 23	& 30.8	& 3.2 (2)	 \\
		128     && 1	& 201	& 2.9	& 3.2 (4)	&& 1	& 17	& 10.7	& 9.1 (4)	 \\
		256     && 1	& 232	& 1.9	& 4.9 (8)	&&1	& 19		& 3.7	& 26.3 (8)	 \\
		\hline
		%	\end{tabular}
		%%\end{table}
		%%
		%	\vspace*{5mm}
		%%
		%%\begin{table}[H]
		%	\caption{Strong scaling test on ellipsoidal domain. Global mesh of $128\cdot 128\cdot 24$ elements, 3,244,995 dofs.	Simulations of 2 ms of cardiac activation with $dt = 0.05$ ms, for a total amount fo 40 time steps. Comparison of Newton-Krylov solvers with bAMG and BDDC  preconditioners. Average Newton iterations per time step (nit); average GMRES iterations per Newton iteration (lit); average CPU solution time per time step (time) in seconds; parallel speedup ($S_p$), with ideal speedup in brackets.}
		%	\label{table_bido_strong_ell_galileo_coupled}
		%	\centering
		%	\begin{tabular}{*{11}{r}}
		%		\hline
		%		& & & & & & & & & & \\
		%		\hline
		\multicolumn{11}{c}{Ellipsoidal domain} \\
		\hline
		\multirow{2}{*}{subds.} 		&& \multicolumn{4}{c}{bAMG}     	&& \multicolumn{4}{c}{BDDC}     	\\
		&& nit   & lit   & time	& $S_p$		&& nit   & lit   & time	& $S_p$		\\
		\cline{1-1} \cline{3-6} \cline{8-11} 
		32     	&& 2	& 187	& 15.1	& -			    && 2	& 37	& 189.3	& -				\\
		64      && 2	& 222	& 9.2	& 1.6 (2)	&& 2	& 44	& 59.1	& 3.2 (2)	 \\
		128     && 2	& 240	& 5.3	& 2.8 (4)	&& 2	& 29	& 20.1	& 9.4 (4)	 \\
		256     && 2	& 280	& 3.2	& 4.7 (8)	&& 2	& 46	& 10.2	& 18.5 (8)	 \\
		\hline
	\end{tabular}
\end{table}

\begin{figure}[H]
	\centering
	\includegraphics[scale=.43]{ 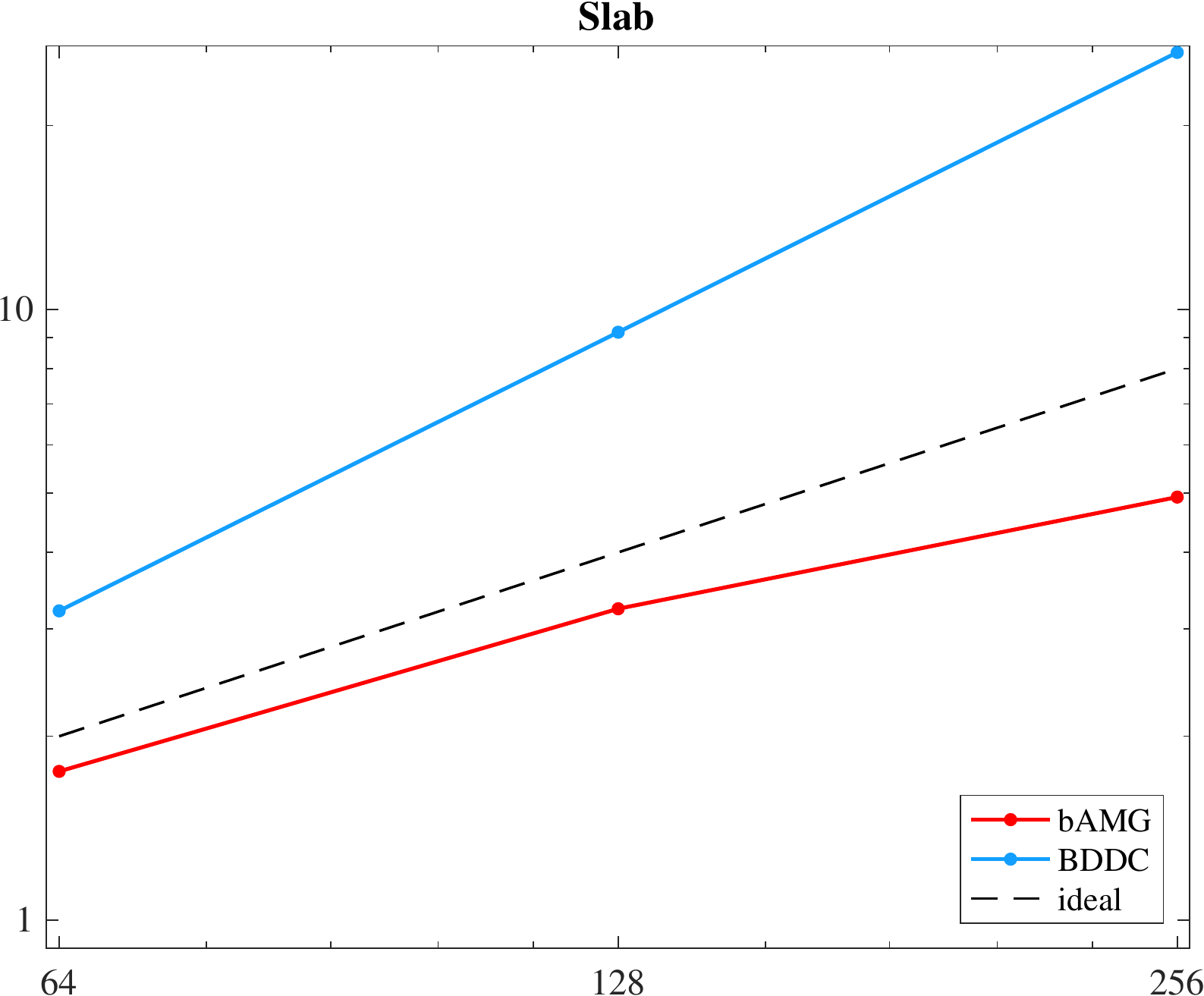}
	\includegraphics[scale=.43]{ 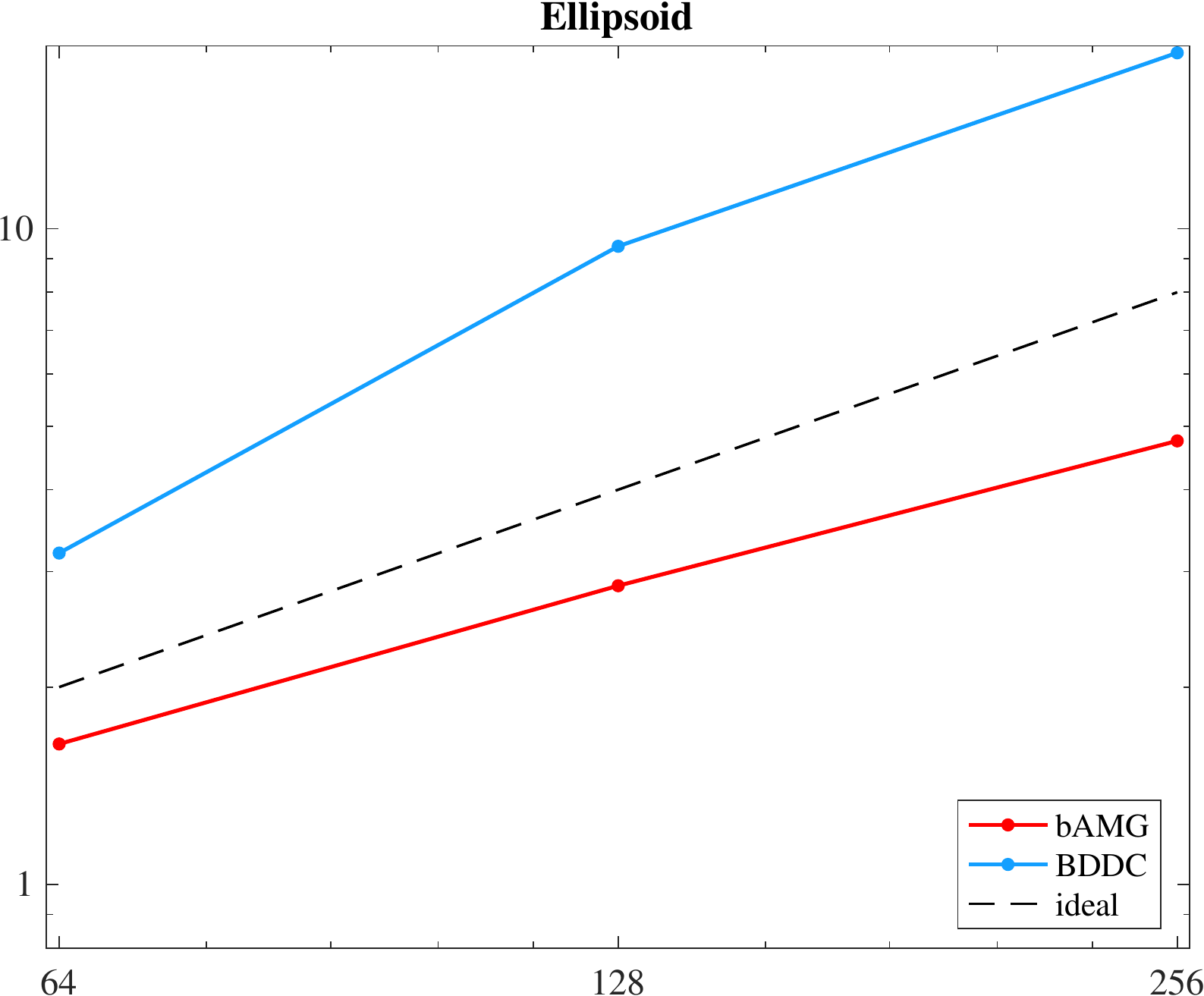}
	\caption{Strong scaling test with global mesh of $128\cdot 128\cdot 24$. Slab (left) and ellipsoidal (right) geometries. Simulations of 2 ms of cardiac activation with $dt = 0.05$ ms, for a total amount fo 40 time steps. Comparison of actual parallel speedup (ideal speedup dotted). }
	\label{fig_bido_speedup_coupled}
\end{figure}

{\bf Test 3: optimality tests.}
Tables \ref{table_bido_opt_slab_coupled} and \ref{table_bido_opt_ell_coupled} report the results of optimality tests, for both slab and ellipsoid geometries, carried on Galileo cluster. 

We fix the number of processors (subdomains) to $4 \cdot 4 \cdot 4$ and we increase the local size $H/h$ from 4 to 24, thus reducing the finite element size $h$. 
We consider both scalings ($\rho$-scaling on top, deluxe scaling at the bottom of each table) and we test the solver for increasing primal spaces: V includes only vertex constraints, V+E includes vertex and edge constraints, and V+E+F includes vertex, edge and face constraints. The deluxe scaling tests are up to a local size of $20 \cdot 20 \cdot 20$ elements, due to limited computational resources.
We consider a time interval of 2 ms during the cardiac activation phase. The time step is $dt = 0.05$ ms, for a total amount fo 40 time steps. 

Similar results hold for both geometries. The deluxe solver seems to be more robust while increasing the local mesh size, both in terms of average linear iterations (see also Figure \ref{fig_bido_opt_coupled} ) and average CPU time per time step. 
\begin{table}[!h]
	\caption{Optimality test.
		Slab domain, $4 \cdot 4 \cdot 4$ subdomains, increasing local size from $4 \cdot 4 \cdot 4$ to $24 \cdot 24 \cdot 24$ (up to $20 \cdot 20 \cdot 20$ for the deluxe scaling). Comparison between different scalings and different primal sets (V = vertices, E = edge averages, F = face averages). Average non-linear iterations (nlit), average number of linear iterations (lit) and average CPU time in seconds per time step.}	
	\label{table_bido_opt_slab_coupled}
	\centering
	\begin{tabular}{*{13}{r}}
		\hline
		\multicolumn{13}{c}{$\rho$-scaling}	\\
		\hline
		\multirow{2}{*}{H/h}	&& \multicolumn{3}{c}{V} 	&& \multicolumn{3}{c}{V+E} 		&& \multicolumn{3}{c}{V+E+F}  \\
		\cline{3-5} \cline{7-9} \cline{11-13}
		&& nlit 	& lit		&time			&& nlit 	& lit		&time			&& nlit 	& lit		&time		\\
		\cline{1-1} \cline{3-5} \cline{7-9} \cline{11-13}
		4			 &&1 &29 &0.2 		 	 	  &&1 &16 &0.1 	 	  &&1 &18 &0.2  \\
		8			 &&1 &48 &0.8 			  	  &&1 &17 &0.7 		  &&1 &16 &0.6  \\
		12			&&1 &65 &4.5 			 	 &&1 &19 &3.4 	 	&&1 &18 &3.6  \\
		16 			&& 1 &77 &20.6 			 	&&1 &21 &15.8  	   &&1 &19 &16.5  \\
		20			&& 1 &99 &70.0 		       &&1 &23 &52.5 	    &&1 &21 &54.4  \\
		24			&& 1 &219 &256.2 		 &&1 &24 &156.4		&&1 &22 &158.6 \\
		\hline
		\multicolumn{13}{c}{deluxe scaling}	\\
		\hline
		\multirow{2}{*}{H/h}	&& \multicolumn{3}{c}{V} 	&& \multicolumn{3}{c}{V+E} 		&& \multicolumn{3}{c}{V+E+F}  \\
		\cline{3-5} \cline{7-9} \cline{11-13}
		&& nlit 	& lit		&time			&& nlit 	& lit		&time			&& nlit 	& lit		&time		\\
		\cline{1-1} \cline{3-5} \cline{7-9} \cline{11-13}
		4			 &&1 &29 &0.3 		 	 	  &&1 &14 &0.2 	 	 &&1 &16 &0.2  \\
		8			 &&1 &46 &1.1 			  	  &&1 &15 &0.6 		 &&1 &15 &0.6  \\
		12			&&1 &59 &5.4 			 	 &&1 &17 &3.2 	 	&&1 &15 &3.1  \\
		16 			&& 1 &67 &21.8 			 	&&1 &17 &13.3  	 &&1 &16 &13.2  \\
		20			&& 1 &73 &66.7 		        &&1 &18 &42.6 	&&1 &17 &42.5 \\
		%		24			&& - &-    &- 		 				&&- &- &- 				&&- &- &- \\
		\hline
	\end{tabular}
	%\end{table}
	%
	\vspace*{5mm}
	%
	%\begin{table}[!h]
	\caption{Optimality test.
		Ellipsoidal domain, $4 \cdot 4 \cdot 4$ subdomains,  increasing local size from $4 \cdot 4 \cdot 4$ to $24 \cdot 24 \cdot 24$ (up to $20 \cdot 20 \cdot 20$ for the deluxe scaling). Comparison between different scalings and different primal sets (V = vertices, E = edge averages, F = face averages). Average non-linear iterations (nlit), average number of linear iterations (lit) and average CPU time in seconds per time step.}	
	\label{table_bido_opt_ell_coupled}
	\centering
	\begin{tabular}{*{13}{r}}
		\hline
		\multicolumn{13}{c}{$\rho$-scaling}	\\
		\hline
		\multirow{2}{*}{H/h}	&& \multicolumn{3}{c}{V} 	&& \multicolumn{3}{c}{V+E} 		&& \multicolumn{3}{c}{V+E+F}  \\
		\cline{3-5} \cline{7-9} \cline{11-13}
		&& nlit 	& lit		&time			&& nlit 	& lit		&time			&& nlit 	& lit		&time		\\
		\cline{1-1} \cline{3-5} \cline{7-9} \cline{11-13}
		4			 &&2 &57 &0.3 		 	 	  &&2 &27 &0.9 	 	 &&- &- &-  \\
		8			 &&2 &113 &1.6 			  	  &&2 &58 &1.4 		 &&2 &56 &1.4  \\
		12			&&2 &139 &8.7 			 	 &&2 &67 &6.4 	 	&&2 &65 &4.5  \\
		16 			&&2 &228 &47.2 			 	&&2 &75 &11.9  	 &&2 &75 &11.6  \\
		20			&&2 &277 &144.6 		  &&2 &82 &34.2 	&&2 &81 &32.5  \\
		24			&&2 &477 &494.5 		 &&2 &75 &83.1 		&&2 &87 &80.5 \\
		\hline
		\multicolumn{13}{c}{deluxe scaling}	\\
		\hline
		\multirow{2}{*}{H/h}	&& \multicolumn{3}{c}{V} 	&& \multicolumn{3}{c}{V+E} 		&& \multicolumn{3}{c}{V+E+F}  \\
		\cline{3-5} \cline{7-9} \cline{11-13}
		&& nlit 	& lit		&time			&& nlit 	& lit		&time			&& nlit 	& lit		&time		\\
		\cline{1-1} \cline{3-5} \cline{7-9} \cline{11-13}
		4			 &&2 &45 &0.3 		 	 	  &&2 &25 &0.3 	 	 &&2 &24 &0.3  \\
		8			 &&2 &86 &1.8 			  	  &&2 &40 &1.2 		 &&2 &40 &1.2  \\
		12			&&2 &118 &9.8 			 	 &&2 &53 &6.9 	 	&&2 &46 &6.3  \\
		16 			&& 2 &138 &39.8 			&&2 &50 &25.1  	 &&2 &51 &25.2  \\
		20			&& 2 &172 &130.8 		   &&2 &60 &81.6 	&&2 &62 &83.5  \\
		%		24			&& - &-    &- 		 				&&- &- &- 				&&- &- &- \\
		\hline
	\end{tabular}
\end{table}
\begin{figure}[!h]
	\centering
	\includegraphics[scale=.42]{ 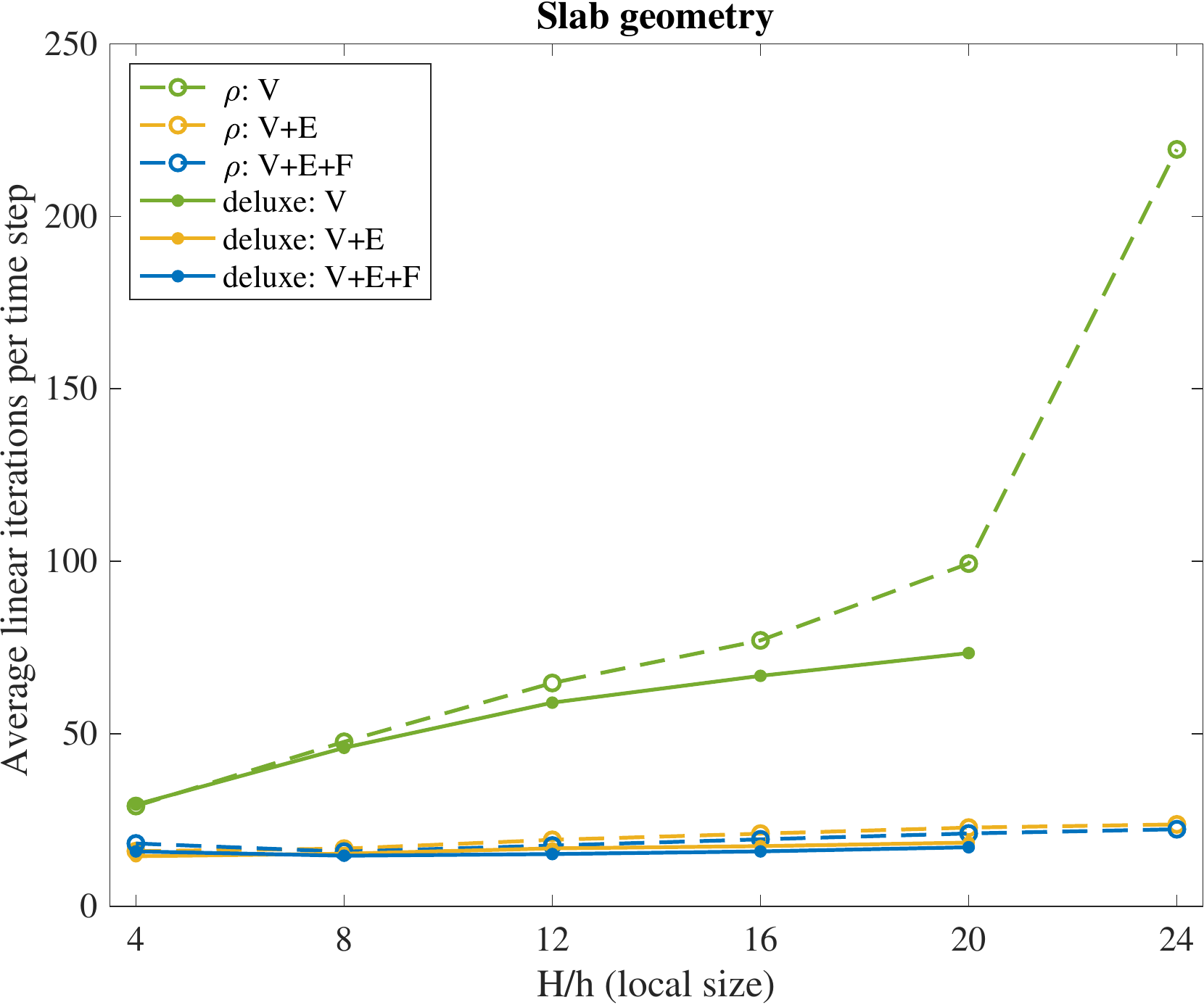}
	\includegraphics[scale=.422]{ 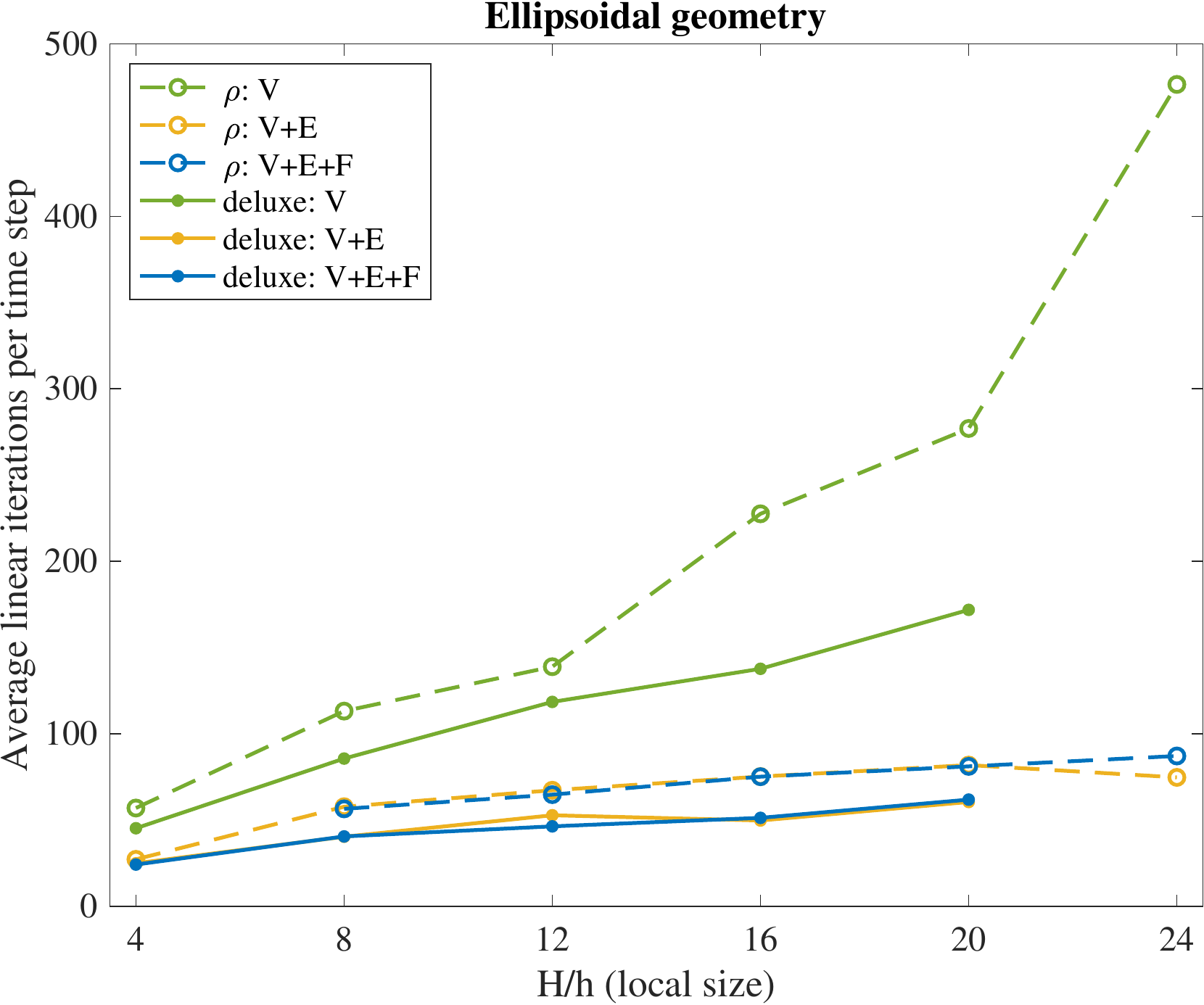}
	\caption{Optimality test with different scalings (dash-dotted $\rho$-scaling, continuous {\em deluxe} scaling) and primal sets (V = vertices, E = edge averages, F = face averages). Slab (left) and ellipsoidal (right) domains, $4 \cdot 4 \cdot 4$ subdomains, increasing local size from $4 \cdot 4 \cdot 4$ to $24 \cdot 24 \cdot 24$.  Average number of linear iterations per time step.}
	\label{fig_bido_opt_coupled}
\end{figure}

{\bf Test 4: heartbeat simulations.} 
In these last tests, we compare the performance of our dual-primal and the multigrid preconditioners during a whole heartbeat.

We fix the number of subdomains to $128=8 \cdot 8 \cdot 2$ and the global mesh size to $128 \cdot 96 \cdot 24$, thus considering local problems of 8,619 dofs.
We consider a time interval of $[0,200]$ ms for a total of 4000 time steps for a portion of ellipsoid defined by $\varphi_{\min} = -\pi/2$, $\varphi_{\max} = \pi/2$, $\theta_{\min} = -3/8 \pi$ and $\theta_{\max} = \pi/8$, while on the slab of dimensions $1.92 \times 0.96 \times 0.96$ cm$^3$ we perform the tests for 3000 time steps, on the time interval $[0, 150]$ ms.

In Fig. \ref{fig_bido_full_coupled} we report the trend of the average number of linear iteration per time step during the simulation. Firstly, we notice a huge difference between the bAMG  and the BDDC preconditioners, with a reduction of more than $85\%$ for the latter. The number of iterations remains bounded during the test. 

Both preconditioned solvers seem to affected by the different stages of the action potential: an initial peak during the activation phase is followed by a constant elevated number of linear iterations as the electric signal propagates in the cardiac tissue, ending with a lower - but always constant - number of linear iterations during the resting phase.

Despite similar qualitative trends between the two geometries, it is undeniable that there are differences from the quantitative point of view, due to the complexity of the domain taken in consideration. 
Comparable performances in terms of CPU time per time step (see Table \ref{table_bido_full_coupled}) hold for both preconditioners.

\begin{table}[!h]
	\caption{Heartbeat simulation on time interval $[0,150]$ ms, 3000 time steps for the slab and on time interval $[0,200]$ ms, 4000 time steps for the ellipsoidal domain. Fixed number of subdomains $8 \cdot 8 \cdot 2$ and fixed global mesh $ 128 \cdot 96 \cdot 24$. Comparison of average Newton steps, average linear iterations and average CPU time (in sec.) per time step.}	
	\label{table_bido_full_coupled}
	\centering
	\begin{tabular}{*{12}{c}}
		\cline{2-12}
		& \multirow{2}{*}{procs}	&& \multirow{2}{*}{dofs} && \multicolumn{3}{c}{bAMG} 	&& \multicolumn{3}{c}{BDDC} 		 \\
		&&&  && nlit 	& lit		&time	&& nlit 	& lit		&time			\\
		\cline{1-2} \cline{4-4} 	\cline{6-8} \cline{10-12} 
		slab		 & 128 		  && 8,619	 && 1 	  &235 	  &3.65 	  &&1 	&34  	&20.66	\\
		ellipsoid 	& 128		&& 8,619 	&& 6 	&1,134 	&9.54		&& 6 	& 96 	& 8.27			\\
		\hline
	\end{tabular}
\end{table}

\begin{figure}[H]
	\centering
	\includegraphics[scale=.4]{ 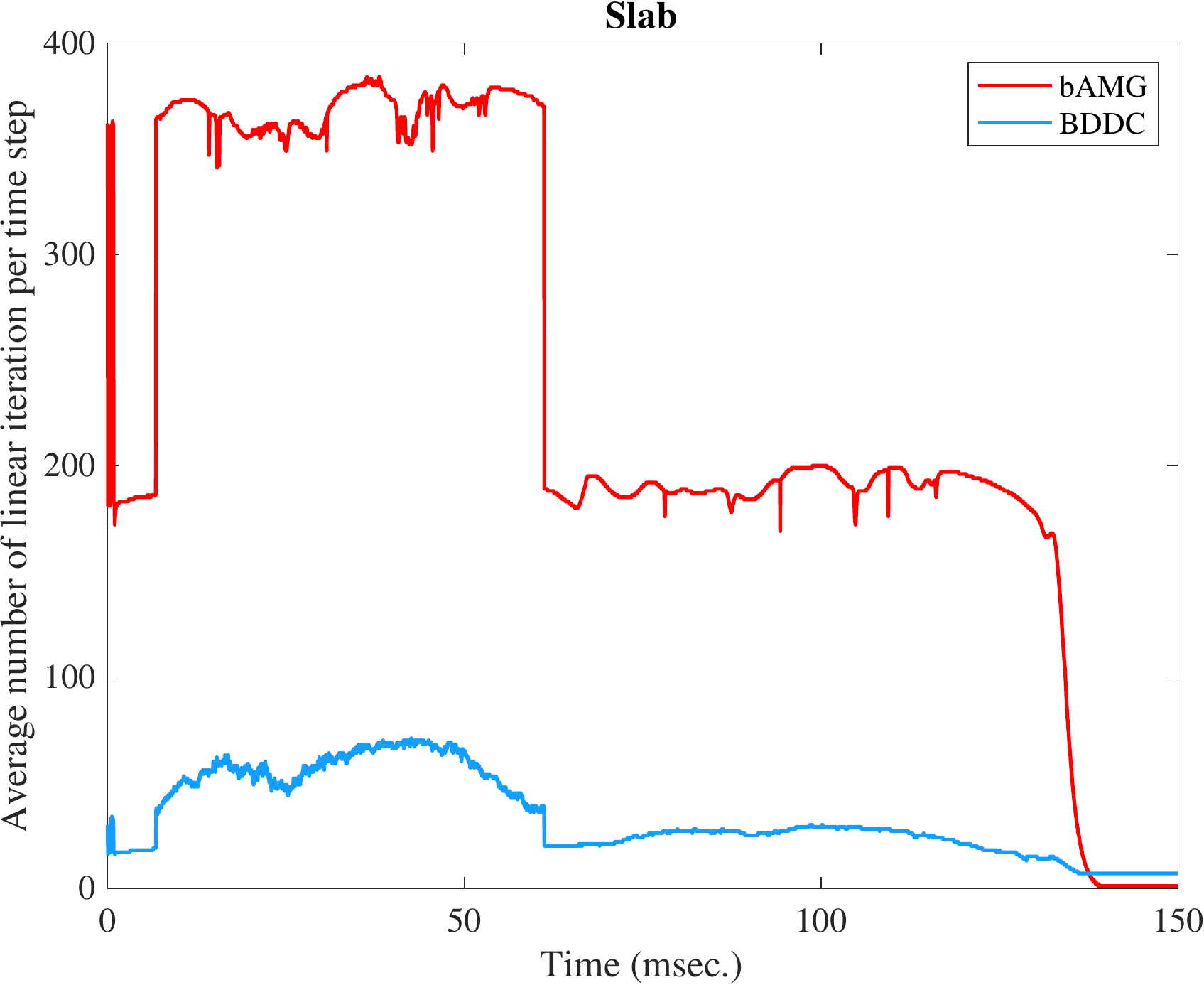}
	\includegraphics[scale=.4]{ 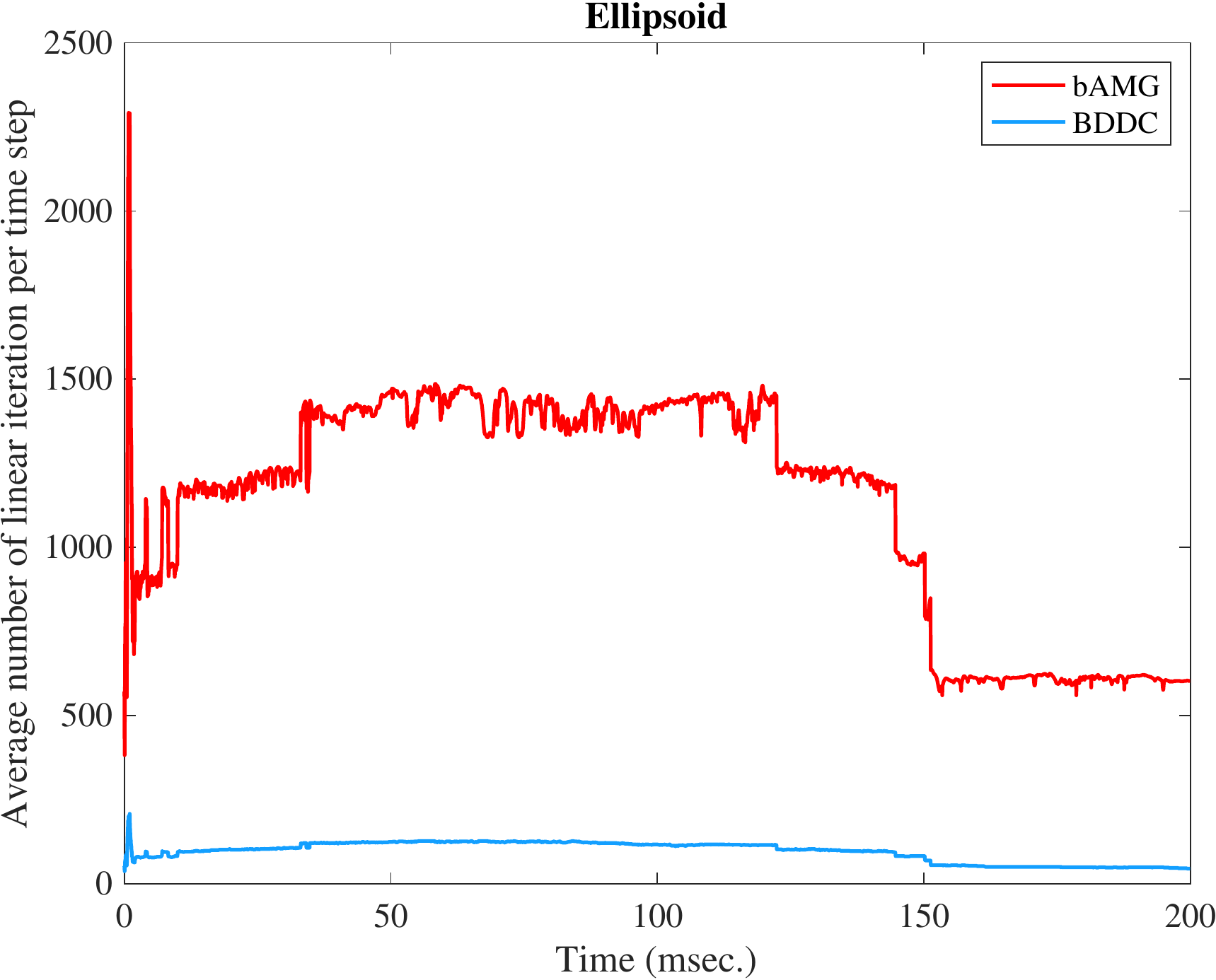}
	\caption{Whole heartbeat simulation on slab domain, time interval $[0,150]$ms, 3000 time steps (on the left) and on ellipsoidal domain, time interval $[0,200]$ms, 4000 time steps (on the right). Fixed number of subdomains $8 \cdot 8 \cdot 2$ and fixed global mesh $ 128 \cdot 96 \cdot 24$. Comparison between bAMG and BDDC average number of linear iterations per time step }
	\label{fig_bido_full_coupled}
\end{figure}

\section{Conclusion}
We have designed a dual-primal Newton-Krylov solver for the monolithic solution strategy for fully implicit time discretizations of the cardiac Bidomain model. 
Theoretical analysis for the convergence rate of the non-symmetric preconditioned operator has been provided. 
We have validated this result through extensive parallel numerical tests, showing the efficiency and robustness of the solver, thus encouraging further investigation using more complex ionic models and realistic heart geometries.

\section*{Acknowledgement}
The Author would like to thank Simone Scacchi and Luca Pavarino for many helpful discussions and comments.

%\newpage

\end{document}